\documentclass{amsart}

\usepackage{amssymb, amscd, amsthm, mathrsfs}

\newtheorem{MainThm}{Theorem}

\newtheorem{Thm}{Theorem}[subsection]
\newtheorem{Prop}[Thm]{Proposition}

\newtheorem{Lem}[Thm]{Lemma}

\newtheorem{PropApp}{Proposition}[section] 
\theoremstyle{remark}
\newtheorem{Rmk}[Thm]{Remark}
\newtheorem*{Notation}{Notation}
\newtheorem*{Ack}{Acknowledgement}

\numberwithin{equation}{subsection}

\newcommand{\Order}{\mathcal{O}}
\newcommand{\into}{\hookrightarrow}
\newcommand{\onto}{\twoheadrightarrow}
\newcommand{\isomto}{\overset{\sim}{\to}}
\newcommand{\isomfrom}{\overset{\sim}{\leftarrow}}
\newcommand{\compose}{\mathrel{\circ}}
\newcommand{\tensor}{\otimes}

\newcommand{\closure}[1]{\overline{#1}}
\newcommand{\N}{\mathbb{N}}
\newcommand{\Z}{\mathbb{Z}}
\newcommand{\Q}{\mathbb{Q}}
\newcommand{\F}{\mathbb{F}}
\newcommand{\Adele}{\mathbb{A}}

\newcommand{\et}{\mathrm{et}}
\newcommand{\zar}{\mathrm{zar}}

\newcommand{\fppf}{\mathrm{fppf}}

\newcommand{\rat}{\mathrm{rat}}

\newcommand{\perf}{\mathrm{perf}}

\newcommand{\Alg}{\mathrm{Alg}}
\newcommand{\PDual}{\mathrm{PD}}
\newcommand{\CDual}{\mathrm{CD}}
\newcommand{\SDual}{\mathrm{SD}}
\newcommand{\LDual}{\mathrm{LD}}

\newcommand{\Gm}{\mathbf{G}_{m}}
\newcommand{\Ga}{\mathbf{G}_{a}}
\newcommand{\Et}{\mathrm{Et}}
\newcommand{\FEt}{\mathrm{FEt}}
\newcommand{\FGEt}{\mathrm{FGEt}}
\newcommand{\Fin}{\mathrm{Fin}}

\newcommand{\uc}{\mathrm{uc}}
\newcommand{\ind}{\mathrm{ind}}
\newcommand{\pro}{\mathrm{pro}}
\newcommand{\tor}{\mathrm{tor}}

\newcommand{\sm}{\mathrm{sm}}
\newcommand{\set}{\mathrm{set}}
\newcommand{\Ind}{\mathrm{I}}
\newcommand{\Pro}{\mathrm{P}}
\newcommand{\Loc}{\mathrm{L}}
\newcommand{\NS}{\mathrm{NS}}
\newcommand{\ca}{\mathrm{ca}}

\newcommand{\dlog}{\mathrm{dlog}}
\newcommand{\ideal}[1]{\mathfrak{#1}}
\newcommand{\alg}[1]{\mathbf{#1}}
\newcommand{\dirlim}{\varinjlim}
\newcommand{\invlim}{\varprojlim}

\newcommand{\finvlim}[1][]{\mathop{\text{``$\invlim$''}}_{#1}}

\newcommand{\var}{\;\cdot\;}
\mathchardef\mhyphen="2D

\DeclareMathOperator{\Gal}{Gal}
\DeclareMathOperator{\Hom}{Hom}

\DeclareMathOperator{\Ker}{Ker}
\DeclareMathOperator{\Coker}{Coker}
\DeclareMathOperator{\Ext}{Ext}
\DeclareMathOperator{\Spec}{Spec}
\DeclareMathOperator{\Ab}{Ab}
\DeclareMathOperator{\Set}{Set}
\DeclareMathOperator{\Res}{Res}

\let\Im\relax
\DeclareMathOperator{\Im}{Im}
\DeclareMathOperator{\sheafhom}{\alg{Hom}}
\DeclareMathOperator{\sheafext}{\alg{Ext}}
\DeclareMathOperator{\Pic}{Pic}
\DeclareMathOperator{\Tr}{Tr}
\DeclareMathOperator{\Ch}{Ch}

\newcommand{\divis}{\mathrm{div}}
\newcommand{\Height}{\mathrm{Height}}
\newcommand{\Tran}{\mathrm{Tran}}
\newcommand{\CT}{\mathrm{CT}}
\newcommand{\ctensor}{\mathbin{\Hat{\tensor}}}

\DeclareFontFamily{U}{wncy}{}
\DeclareFontShape{U}{wncy}{m}{n}{<->wncyr10}{}
\DeclareSymbolFont{mcy}{U}{wncy}{m}{n}
\DeclareMathSymbol{\Sha}{\mathord}{mcy}{"58}

\newcommand{\BetweenThmAndList}{\leavevmode}

\title[Duality for cohomology of curves]
{Duality for cohomology of curves with coefficients in abelian varieties}
\author{Takashi Suzuki}
\thanks{The author is a Research Fellow of Japan Society for the Promotion of Science}
\address{
	Department of Mathematics, Chuo University,
	1-13-27 Kasuga, Bunkyo-ku, Tokyo 112-8551, JAPAN
}
\email{tsuzuki@gug.math.chuo-u.ac.jp}
\date{January 2, 2019}
\subjclass[2010]{Primary: 11G10; Secondary: 11R58, 14F20}
\keywords{Abelian varieties; duality; Grothendieck topologies}

\begin{document}

\begin{abstract}
	In this paper, we formulate and prove a duality
	for cohomology of curves over perfect fields of positive characteristic
	with coefficients in N\'eron models of abelian varieties.
	This is a global function field version of the author's previous work on local duality
	and Grothendieck's duality conjecture.
	It generalizes the perfectness of the Cassels-Tate pairing in the finite base field case.
	The proof uses the local duality mentioned above,
	Artin-Milne's global finite flat duality, the non-degeneracy of the height pairing
	and finiteness of crystalline cohomology.
	All these ingredients are organized under the formalism of the rational \'etale site developed earlier.
\end{abstract}

\maketitle

\tableofcontents


\section{Introduction}

\subsection{Aim of the paper}
Let $X$ be a proper smooth geometrically connected curve
over a perfect field $k$ of characteristic $p > 0$.
Let $A$ be an abelian variety over the function field $K$ of $X$,
with dual $A^{\vee}$.
We denote the N\'eron models of $A$ and $A^{\vee}$ over $X$ by
$\mathcal{A}$ and $\mathcal{A}^{\vee}$, respectively.
Let $\mathcal{A}_{0}^{\vee}$ be the maximal open subgroup scheme of $\mathcal{A}^{\vee}$
with connected fibers.
We view these group schemes over $X$ as representable sheaves on $X$
in the flat topology and hence in the \'etale topology.
In this paper, we formulate and prove a duality that relates the two \'etale cohomology complexes
	\[
			R \Gamma(X, \mathcal{A}),
		\quad
			R \Gamma(X, \mathcal{A}_{0}^{\vee})
	\]
to each other.
($R \Gamma(X, \mathcal{A})$ is a complex whose $n$-th cohomology is $H^{n}(X, \mathcal{A})$.)

This is the global function field version of the author's work \cite{Suz14}
on duality for cohomology of local fields with coefficients in abelian varieties
that solved Grothendieck's duality conjecture \cite[IX, Conj.\ 1.3]{Gro72}.
The complexes above are not viewed as complexes of mere abelian groups,
but as complexes of sheaves on the ind-rational pro-\'etale site $\Spec k^{\ind\rat}_{\pro\et}$
defined in \cite{Suz14}.
In particular, this duality treats the structure of $H^{n}(X, \mathcal{A})$
as the perfection (inverse limit along Frobenii) of a smooth group scheme over the base field $k$,
which is related to Artin-Milne's duality \cite{AM76} for cohomology of $X$
with coefficients in finite flat group schemes.
The object $H^{1}(X, \mathcal{A})$ has unipotent connected part (which is $p$-power torsion)
and is the sheafified version of the Tate-Shafarevich group of $A / K$,
which we call the \emph{Tate-Shafarevich scheme}.
See Prop.\ \ref{prop: first cohom is Tate Shafarevich} and the preceding paragraph
for a more precise (but a bit subtle) relationship to Tate-Shafarevich groups.
The above duality generalizes the perfectness of the Cassels-Tate pairing
in the finite base field case.
The duality for the part $H^{0}(X, \mathcal{A}) = A(K)$ includes the non-degeneracy of the height pairing.
When showing the finite-dimensionality of the Tate-Shafarevich scheme,
we will use the finiteness of crystalline cohomology of proper smooth surfaces over $k$.

Our duality extends a partial result of Milne \cite[III, Thm.\ 11.6]{Mil06}
on duality $T_{p} \Gamma(K, A^{\vee}) \leftrightarrow H^{2}(X, \mathcal{A})$.
He also pointed out that the part killed by $p$ (not the whole $p$-primary part)
of the Tate-Shafarevich group can be infinite when $k$ is algebraically closed (\cite[III, Rmk.\ 9.9]{Mil06}).
This phenomenon has been studied by Vvedenski\u\i\ (\cite{Vve81} for example),
which, in our formulation, can be explained by the connected part of the Tate-Shafarevich scheme.
It appears from \cite[Rmk.\ 7]{Vve78} that
Vvedenski\u\i\ at least once imagined a possibility to construct a duality theory
similar to ours.

The most interesting part of this duality theory lies in $p$-torsion.
The prime-to-$p$ part is essentially classical
(cf.\ \cite{Ray95}).
When $k$ has zero characteristic, which is out of scope of this paper,
one should probably take (non-semistable) degenerations of Hodge structures into account.

When the first version of this paper was uploaded to arXiv,
\v{C}esnavi\v{c}ius pointed the author to the preprint \cite{DH18} by Demarche-Harari,
which was written independently at almost the same time.
They develop compact support flat cohomology,
a key technical tool that we also develop,
in a way very similar to us.
The setting and the details are different.
For more details, see \S \ref{sec: Organization}.


\subsection{Statement of the main theorem}
Now we formulate our duality.
Let $\Ab(X_{\fppf})$ and $\Ab(k^{\ind\rat}_{\pro\et})$ be the categories of sheaves of abelian groups
on the fppf site $X_{\fppf}$ and the ind-rational pro-\'etale site
$\Spec k^{\ind\rat}_{\pro\et}$ \cite[\S 2.1]{Suz14}, respectively.
We define a left exact functor
	\[
			\alg{\Gamma}(X, \var)
		\colon
			\Ab(X_{\fppf})
		\to
			\Ab(k^{\ind\rat}_{\pro\et})
	\]
by sending an fppf sheaf $F$ on $X$
to the pro-\'etale sheafification of the presheaf $k' \mapsto F(X \times_{k} k')$
on $\Spec k^{\ind\rat}_{\pro\et}$,
where $k'$ runs through ind-rational $k$-algebras.
Denote its right derived functor by
	\[
			R \alg{\Gamma}(X, \var)
		\colon
			D(X_{\fppf})
		\to
			D(k^{\ind\rat}_{\pro\et})
	\]
and set $\alg{H}^{n}(X, \var) = H^{n} R \alg{\Gamma}(X, \var)$.
The complex of sheaves $R \alg{\Gamma}(X, \mathcal{A})$ on $\Spec k^{\ind\rat}_{\pro\et}$
is the main object of study.
Denote the derived sheaf-Hom functor for $\Spec k^{\ind\rat}_{\pro\et}$ by
$R \sheafhom_{k^{\ind\rat}_{\pro\et}}$.
For $G \in D(k^{\ind\rat}_{\pro\et})$, its Serre dual \cite[\S 2.4]{Suz14} is defined by
	\[
			G^{\SDual}
		=
			R \sheafhom_{k^{\ind\rat}_{\pro\et}}(G, \Z).
	\]
The Poincar\'e biextension $\mathcal{A}_{0}^{\vee} \tensor^{L} \mathcal{A} \to \Gm[1]$
as a morphism in $D(X_{\fppf})$,
the cup product and the degree map of the Picard scheme induce morphisms
	\[
			R \alg{\Gamma}(X, \mathcal{A}_{0}^{\vee})
		\tensor^{L}
			R \alg{\Gamma}(X, \mathcal{A})
		\to
			R \alg{\Gamma}(X, \Gm)[1]
		\to
			\alg{H}^{1}(X, \Gm)
		\to
			\Z
	\]
in $D(k^{\ind\rat}_{\pro\et})$.
Hence we have a morphism
	\[
			R \alg{\Gamma}(X, \mathcal{A})
		\to
			R \alg{\Gamma}(X, \mathcal{A}_{0}^{\vee})^{\SDual}.
	\]
Its Serre dual
	\begin{equation} \label{eq: duality morphism}
			R \alg{\Gamma}(X, \mathcal{A}_{0}^{\vee})^{\SDual \SDual}
		\to
			R \alg{\Gamma}(X, \mathcal{A})^{\SDual}
	\end{equation}
is our duality morphism.
Let
	\[
			V
		=
			V \alg{H}^{1}(X, \mathcal{A}_{0}^{\vee})_{\divis}
		=
			\bigl(
				T \alg{H}^{1}(X, \mathcal{A}_{0}^{\vee})_{\divis}
			\bigr) \tensor \Q
	\]
be the rational Tate module of the maximal divisible subsheaf of $\alg{H}^{1}(X, \mathcal{A}_{0}^{\vee})$.

\begin{MainThm} \label{mainthm}
	There exist canonical morphisms
		\[
				R \alg{\Gamma}(X, \mathcal{A})^{\SDual}
			\to
				V
			\to
				R \alg{\Gamma}(X, \mathcal{A}_{0}^{\vee})^{\SDual \SDual}[1]
		\]
	such that the triangle
		\[
				V[-1]
			\to
				R \alg{\Gamma}(X, \mathcal{A}_{0}^{\vee})^{\SDual \SDual}
			\to
				R \alg{\Gamma}(X, \mathcal{A})^{\SDual}
			\to
				V
		\]
	is distinguished in $D(k^{\ind\rat}_{\pro\et})$.
\end{MainThm}

Both of the objects $R \alg{\Gamma}(X, \mathcal{A}_{0}^{\vee})^{\SDual \SDual}$
and $R \alg{\Gamma}(X, \mathcal{A})^{\SDual}$ are concentrated in degrees $-1, 0, 1, 2$,
while $V \alg{H}^{1}(X, \mathcal{A}_{0}^{\vee})_{\divis}$ has no connected part,
is uniquely divisible and concentrated in the single degree zero.
Hence our duality morphism is ``close'' to be an isomorphism.
The divisible part of the usual Tate-Shafarevich group, when $k$ is finite, is conjectured to be zero.
But our Tate-Shafarevich \emph{scheme}, $\alg{H}^{1}(X, \mathcal{A})$,
has non-zero divisible part in general,
which might be called the space of ``transcendental cycles with coefficients in the N\'eron model''.

Concrete consequences of this theorem will be explained at
\S \ref{sec: Main theorem}.
We will also give a version for cohomology of dense open subschemes of $X$
in \S \ref{sec: Duality for Neron models over open curves}
and explain the link of Thm.\ \ref{mainthm} to the known duality theory in the finite base field case
in \S \ref{sec: Link to the finite base field case}.
A very small amount of explicit calculations is given in the course of
Rmk.\ \ref{rmk: connected divisible part}.

A small remark is that the first cohomology classifies torsors or principal bundles.
The Tate-Shafarevich scheme $\alg{H}^{1}(X, \mathcal{A})$ might alternatively be called
the moduli of $G$-bundles on $X$ and denoted by $\mathrm{Bun}_{G}(X)$,
where $G = \mathcal{A}$.
We do not pursue this viewpoint, merely mentioning that
our crucial point is to evaluate functors on ind-rational $k$-algebras only.


\subsection{Outline of proof}

Here is a rough outline of the proof.
The first deep input is Artin-Milne's global duality \cite[Cor.\ (4.9)]{AM76}
for a finite flat group scheme $N$ over $X$:
	\[
			R \alg{\Gamma}(X, N^{\CDual})
		\longleftrightarrow
			R \alg{\Gamma}(X, N),
	\]
where $N^{\CDual}$ is the Cartier dual of $N$,
the cohomology is taken in the fppf topology and
we ignored the shift of degrees for simplicity.
Artin-Milne uses the \'etale site of all perfect $k$-schemes
and we will bring their result to our site $\Spec k^{\ind\rat}_{\pro\et}$
by restriction and pro-\'etale sheafification.
At each closed point $x \in X$,
we have Bester's local finite flat duality \cite[Thm.\ 3.1]{Bes78}
in the form stated in \cite[Thm.\ (5.2.1.2)]{Suz14}:
	\[
			R \alg{\Gamma}_{x}(\Hat{\Order}_{x}, N^{\CDual})
		\longleftrightarrow
			R \alg{\Gamma}(\Hat{\Order}_{x}, N),
	\]
where $\Hat{\Order}_{x}$ is the completed local ring of $X$ at $x$,
and the fppf cohomology functor $R \alg{\Gamma}(\Hat{\Order}_{x}, \var)$
and its version with support $R \alg{\Gamma}_{x}(\Hat{\Order}_{x}, \var)$
with values in $D(k^{\ind\rat}_{\pro\et})$
are as defined in \cite[\S 3.3]{Suz14}.
(\cite[Thm.\ (5.2.1.2)]{Suz14} fits better in the setting of the present paper
as it is a derived categorical and sheaf-theoretic version of
the duality isomorphism of profinite abelian groups in \cite[Thm.\ 3.1]{Bes78}.
Also \cite[Thm.\ (5.2.1.2)]{Suz14} corrects some inaccuracies in Bester's paper;
see \cite[Rmk.\ 5.2.1.5]{Suz14}.)
Hence we can pass from $X$ to a dense open subscheme $U$ of $X$:
	\begin{equation} \label{eq: finite flat duality over open curve}
			R \alg{\Gamma}(U, N^{\CDual})
		\longleftrightarrow
			R \alg{\Gamma}_{c}(U, N),
	\end{equation}
where $R \alg{\Gamma}_{c}(U, \var)$ is the fppf cohomology with compact support
with values in $D(k^{\ind\rat}_{\pro\et})$ that we define and study in this paper.
The fppf cohomology with compact support, as usual abelian groups,
is already defined in \cite[III, \S 0]{Mil06}.
We need some clarification about the definition given in \cite[III, \S 0]{Mil06},
more than just bringing it to $D(k^{\ind\rat}_{\pro\et})$.
This point was simultaneously found by Demarche-Harari \cite{DH18};
see \S \ref{sec: Organization} for the details.
The duality \eqref{eq: finite flat duality over open curve} implies,
taking $U$ small enough so that $A$ has good reduction over $U$, a duality
	\[
			R \alg{\Gamma}(U, \mathcal{A}_{0}^{\vee}[n])
		\longleftrightarrow
			R \alg{\Gamma}_{c}(U, \mathcal{A}[n])
	\]
for any $n \ge 1$.
At each $x \in X$, we have local duality for abelian varieties \cite[Rmk.\ (4.2.10)]{Suz14}
	\[
			R \alg{\Gamma}_{x}(\Hat{\Order}_{x}, \mathcal{A}_{0}^{\vee})
		\longleftrightarrow
			R \alg{\Gamma}(\Hat{\Order}_{x}, \mathcal{A}).
	\]
Hence we can pass from $U$ to $X$ and take the limit in $n$:
	\begin{equation} \label{eq: Q mod Z duality for Neron}
			R \invlim_{n} \bigl(
					R \alg{\Gamma}(X, \mathcal{A}_{0}^{\vee})
				\tensor^{L}
					\Z / n \Z
			\bigr)
		\longleftrightarrow
				R \alg{\Gamma}(X, \mathcal{A})
			\tensor^{L}
				\Q / \Z.
	\end{equation}

So far we have treated essentially torsion objects.
To pass to the desired integral statement
	\[
			R \alg{\Gamma}(X, \mathcal{A}_{0}^{\vee})
		\longleftrightarrow
			R \alg{\Gamma}(X, \mathcal{A})
	\]
up to $V \alg{H}^{1}(X, \mathcal{A}_{0}^{\vee})_{\divis}$,
we need to study the structure of $\alg{H}^{n}(X, \mathcal{A})$ for all $n$.
(Of course we need the double dual $\SDual \SDual$ for the precise statement.)
The Lang-N\'eron theorem (the arbitrary base field version of the Mordell-Weil theorem)
shows that the group of rational points $\alg{\Gamma}(X, \mathcal{A})$
is an \'etale group with finitely generated group of geometric points
extended by an abelian variety over $k$, the $K / k$-trace of $A$.
The part
	\[
			(\pi_{0}
				\alg{\Gamma}(X, \mathcal{A}_{0}^{\vee})
			) / \text{torsion}
		\longleftrightarrow
			(\pi_{0}
				\alg{\Gamma}(X, \mathcal{A})
			) / \text{torsion}
	\]
is given by the height pairing.

For higher cohomology,
the duality theorems used so far and finiteness statements therein give some information.
But this information does not exclude the possibility that
$\alg{H}^{1}(X, \mathcal{A})$ has connected divisible unipotent ind-algebraic part
(such as the direct limit of the groups of Witt vectors $\dirlim_{n} W_{n}$)
since $\alg{H}^{2}_{x}(\Hat{\Order}_{x}, \mathcal{A})$ does have such part.
We need here a global finiteness result to eliminate this possibility.
That is, we compare $\alg{H}^{1}(X, \mathcal{A})$
with the Brauer group of a proper smooth surface over $k$
up to some discrepancy of bounded torsion
(i.e.\ discrepancy killed by multiplication by some positive integer)
and hence with the part $F = p$ of the second rational crystalline cohomology of the surface.
This part does not have connected part
by finiteness of crystalline cohomology,
which proves the desired property of $\alg{H}^{1}(X, \mathcal{A})$.
Dualizing, this in turn shows that $\alg{H}^{2}(X, \mathcal{A})$ has no connected part.

Then we can pass to the rational statement
	\[
			R \invlim_{n} R \alg{\Gamma}(X, \mathcal{A}_{0}^{\vee})
		\longleftrightarrow
			R \alg{\Gamma}(X, \mathcal{A}) \tensor \Q,
	\]
where the derived inverse limit is for multiplication by positive integers $n$.
The non-degeneracy of the height pairing shows that
this rational duality pairing is perfect
up to $V \alg{H}^{1}(X, \mathcal{A}_{0}^{\vee})_{\divis}$.
Then we can finally pass to the desired integral statement, proving the main theorem.

Two remarks are in order.
The reader may have an impression from the above outline that
the duality theorem of this paper is closely related to
duality for crystalline cohomology and Milne's flat duality for surfaces \cite{Mil76}.
Probably it will likely be possible to reduce our theorem to those duality theories
when $A$ is the Jacobian of a proper smooth curve over $K$ with a $K$-rational point
or has semistable reduction everywhere.
But there are differences between these special cases and the general case,
coming from isogenies of $p$-power degrees and wild ramification.
Without our formulation, it could be difficult to formulate a duality statement
that is invariant under isogenies and/or admits Galois descent.

The local duality \cite{Suz14} used above is
where the site $\Spec k^{\ind\rat}_{\pro\et}$ plays a crucial role
(see \cite[\S 1.2]{Suz14} for a little more details).
This should be compared with Milne's approach \cite[III, \S 11]{Mil06}
that uses the \'etale site of all perfect $k$-schemes.


\subsection{Organization}
\label{sec: Organization}

Up to the point \eqref{eq: Q mod Z duality for Neron} above, we will know that
$\alg{H}^{n}(X, \mathcal{A})$ for $n \ge 1$
is an ind-algebraic group of so-called ``cofinite type''.
This is a basic finiteness condition for ind-algebraic groups.
\S \ref{sec: Ind-objects of cofinite type} develops the notion of cofinite type objects
in the setting of ind-categories of general artinian abelian categories.
\S \ref{sec: Ind-algebraic groups of cofinite type} studies cofinite type objects
in the special case of the category of ind-algebraic groups.
\S \ref{sec: The ind-rational pro-etale site and some derived limits} briefly reviews
the ind-rational pro-\'etale site and study some derived limits
such as the one $R \invlim_{n}$ mentioned above.

In \S \ref{sec: Local duality without relative sites}
(\S \ref{sec: Premorphisms of sites} to be mentioned soon),
we basically review the local duality \cite{Suz14}.
Additionally, we improve and simplify the formulation.
We eliminate the relative fppf site
($\Spec K_{\fppf} / k^{\ind\rat}_{\et}$ in the notation of \cite[\S 3.3]{Suz14})
and instead formulate the same duality result using the usual fppf site for local fields.
We can still define the key notion of the structure morphism of a local field,
which targets the ind-rational \'etale site of the residue field.
This newer version is not, however, a morphism of sites, but only a ``premorphism of sites'',
which does not necessarily have an exact pullback functor.
A premorphism of sites is not really new and nothing but a morphism of topologies
in the terminology of \cite[Def.\ 2.4.2]{Art62}.
(We change the terminology since the direction of morphisms is somewhat confusing.)
Its definition is recalled and some basic site-theoretic propositions
are given in the preceding \S \ref{sec: Premorphisms of sites}.
The comparison between the two formulations is given in
Appendix \ref{sec: Comparison of local duality with and without relative sites}.
We do use the local duality results of \cite{Suz14} and do not reprove them.

As explained above, we will need fppf cohomology of curves with compact support
$R \alg{\Gamma}_{c}(U, \var)$
as a complex of sheaves on $\Spec k^{\ind\rat}_{\pro\et}$.
We will develop this machinery in
\S \ref{sec: The fppf site of a curve over the rational etale site of the base}.
There are actually two versions of compact support fppf cohomology currently known,
as explained in \cite[III, Rmk.\ 0.6 (b)]{Mil06}.
Their difference is whether we use henselian local rings $\Order_{x}^{h}$
or their completions $\Hat{\Order}_{x}$ for the local components.
The latter version is what we need in the duality theory of this paper.
But this latter version has a problem about covariant functoriality in $U$
coming from the difference between two versions of local cohomology with support
$R \alg{\Gamma}_{x}(\Order_{x}^{h}, \var)$ and $R \alg{\Gamma}_{x}(\Hat{\Order}_{x}, \var)$.
We will see this problem in Rmk.\ \ref{rmk: henselian and completed versions}.
The difference between the two versions of local cohomology with support vanishes
if the coefficient sheaf is representable by a smooth or finite flat group scheme,
by a Greenberg approximation argument.
However, we need various mapping cone constructions
in order to establish basics about compact support cohomology including covariant functoriality,
which forces us to work on the level of complexes and, consequently,
to explicitly take injective resolutions everywhere.
For non-representable sheaves,
the difference between the two versions of local cohomology with support is unavoidable.
Hence we need to keep all track of this difference from the beginning to build the theory.
Some part of this consideration is purely local,
which is the content of the preceding \S \ref{sec: Henselizations and completions}.
After these subsections, we can and will ignore the difference
since we are interested in smooth or finite flat group schemes only.

Let us mention here that the same problem of compact support fppf cohomology is realized and solved
by Demarche-Harari \cite{DH18} independently at almost the same time.
Their setting is over a finite base field $k = \F_{q}$,
cohomology is viewed as usual complexes of abelian groups, not sheaves,
and coefficient sheaves are affine group schemes for the most part.
Their method to establish covariant functoriality of compact support fppf cohomology
is very similar to our method in
\S \ref{sec: Henselizations and completions}--\ref{sec: The fppf site of a curve over the rational etale site of the base}.

Back to this paper,
\S \ref{sec: Duality for finite flat coefficients over open curves}
proves \eqref{eq: finite flat duality over open curve}.
\S \ref{sec: Mod n duality for Neron models and preliminary calculations} begins
by proving \eqref{eq: Q mod Z duality for Neron} (or the statement before limit in $n$).
Various consequences and further structural results on $\alg{H}^{n}(X, \mathcal{A})$ are given:
the structure of $\alg{\Gamma}(X, \mathcal{A})$
(Prop.\ \ref{prop: structure of cohomology of Neron models});
the structure of $\alg{H}^{1}(X, \mathcal{A})$ and its relation to the Tate-Shafarevich group
(Prop.\ \ref{prop: structure of cohomology of Neron models},
\ref{prop: first cohom is Tate Shafarevich},
\ref{prop: H1 is locally algebraic});
how the duality pairing interacts with N\'eron component groups at the closed points of $X$
(Prop.\ \ref{prop: component group compatibitily morphism});
and the relation to the height pairing
(Prop.\ \ref{prop: our duality contains height pairing}).

The properties of
$R \alg{\Gamma}(X, \mathcal{A})$, $R \alg{\Gamma}(X, \mathcal{A}_{0}^{\vee})$
and the pairing between them proven up to this point,
together with the preliminaries in \S \ref{sec: Site theoretic foundations and local duality},
allow us to deduce the main theorem as a result of formal calculations.
This is done in \S \ref{sec: Formal steps towards duality for Neron models}.
The outline of proof above basically explains that
any choice of the mapping cone of the duality morphism \eqref{eq: duality morphism}
is concentrated in degree zero with cohomology isomorphic to
$V \alg{H}^{1}(X, \mathcal{A}_{0}^{\vee})_{\divis}$.
In addition, we heavily use derived limit arguments, the structure of $\alg{H}^{n}(X, \mathcal{A})$
and the non-degeneracy of the height pairing,
to show that the mapping cone can be taken \emph{canonically}.

\S \ref{sec: Main theorem} just collects all the results obtained so far
to state them as a single theorem.
This finishes the proof of Thm.\ \ref{mainthm}.
A relation to Milne's result \cite[Thm.\ 11.6]{Mil06} is also explained.
The intersection of the connected part and the divisible part
of the Tate-Shafarevich scheme $\alg{H}^{1}(X, \mathcal{A})$
is an interesting but difficult finite \'etale $p$-group.
We briefly explain this in Rmk.\ \ref{rmk: connected divisible part}.

\S \ref{sec: Duality for Neron models over open curves} gives a duality for open $U \subset X$:
	\[
			R \alg{\Gamma}(U, \mathcal{A}_{0}^{\vee})
		\longleftrightarrow
			R \alg{\Gamma}_{c}(U, \mathcal{A})
	\]
up to the same obstruction $V \alg{H}^{1}(X, \mathcal{A}_{0}^{\vee})_{\divis}$.
Again, it requires some work to ensure that
the mapping cone of this duality morphism can be taken canonically.

Applying the derived global section $R \Gamma(k, \var)$
to the complex of sheaves $R \alg{\Gamma}(X, \mathcal{A})$
recovers the usual complex of abelian groups $R \Gamma(X, \mathcal{A})$.
In a non-derived categorical language, this means that there is a spectral sequence
	\[
			E_{2}^{i j}
		=
			H^{i}(k, \alg{H}^{j}(X, \mathcal{A}))
		\Longrightarrow
			H^{i + j}(X, \mathcal{A}).
	\]
In \S \ref{sec: Link to the finite base field case},
when the base field $k$ is finite,
we apply $R \Gamma(k, \var)$ to translate our duality theorem
into the classical duality theorems.
We explain the relations to the finiteness of Tate-Shafarevich groups,
Kato-Trihan's arithmetic cohomology \cite{KT03},
the Cassels-Tate pairing
and the Weil-\'etale cohomology $R \Gamma(X_{W}, \mathcal{A})$.

\begin{Ack}
	The author expresses his deep gratitude to Kazuya Kato.
	It was his suggestion to globalize the results of \cite{Suz14}
	and study Tate-Shafarevich groups as sheaves.
	The author thanks Tetsuji Shioda for his encouragement
	to include explicit calculations,
	though the author's calculations in Rmk.\ \ref{rmk: connected divisible part}
	are essentially no more than quotes of other researchers' calculations.
	Also, an earlier version of this paper proved the finite-dimensionality of $\alg{H}^{1}(X, \mathcal{A})$
	by reducing it to the semistable case,
	using K\"unnemann's projective regular models \cite{Kun98}
	and relating $\alg{H}^{1}(X, \mathcal{A})$ to the Brauer group of a higher-dimensional variety.
	Michael Larsen's suggestion to use Jacobians here
	made it possible to use surfaces only and reduce the length of the proof,
	which the author is grateful.
	Matthew Emerton's idea (personal communication)
	to relate our duality with Brauer-Manin obstructions
	is regrettably not treated in this paper to save pages.
	This will be discussed elsewhere.
	The author is grateful to K\c{e}stutis \v{C}esnavi\v{c}ius
	for his information of the work \cite{DH18},
	to Cyril Demarche and David Harari for their encouragement to submit this paper independently,
	to Fabien Trihan for his information of the relation
	between Prop.\ \ref{prop: duality for arithmetic cohom} of this paper
	and \cite[Cor.\ 4.11, Rmk.\ 4.13]{TV17},
	to Thomas Geisser for discussions about finiteness of Tate-Shafarevich groups
	modulo divisible subgroups when $k$ is finite,
	and to the referees for their careful comments.
	This work was supported by JSPS KAKENHI Grant Number JP18J00415.
\end{Ack}

\begin{Notation}
	We fix two universes $\mathcal{U}_{0}$ and $\mathcal{U}$
	such that $\N \in \mathcal{U}_{0} \in \mathcal{U}$.
	The categories of $\mathcal{U}$-small sets and $\mathcal{U}$-small abelian groups
	are denoted by $\Set$ and $\Ab$, respectively.
	For an abelian category $\mathcal{A}$,
	the category of complexes are denoted by $\Ch(\mathcal{A})$.
	Its homotopy category is $K(\mathcal{A})$
	and derived category $D(\mathcal{A})$.
	We denote the full subcategory of bounded, bounded below and bounded above complexes
	by $D^{b}(\mathcal{A})$, $D^{+}(\mathcal{A})$ and $D^{-}(\mathcal{A})$, respectively.
	Similar notation applies to $\Ch(\mathcal{A})$ and $K(\mathcal{A})$.
	The mapping cone of a morphism $A \to B$ in $\Ch(\mathcal{A})$ is denoted by $[A \to B]$.
	If $A \to B$ is a morphism in a triangulated category
	together with a certain canonical choice of a mapping cone,
	then this mapping cone is also denoted by $[A \to B]$ by abuse of notation.
	If we say $A \to B \to C$ is a distinguished triangle in a triangulated category,
	we implicitly assume that a morphism $C \to A[1]$ to the shift of $A$ is given,
	and the triangle $A \to B \to C \to A[1]$ is distinguished.
\end{Notation}


\section{Site-theoretic foundations and local duality}
\label{sec: Site theoretic foundations and local duality}

\subsection{Ind-objects of cofinite type}
\label{sec: Ind-objects of cofinite type}
Let $p$ be a prime number.
Let $\mathcal{A}$ be an artinian ($\mathcal{U}_{0}$-)small abelian category
such that any object of $\mathcal{A}$ is killed by some power of $p$.
Note that no non-zero object of $\mathcal{A}$ is divisible,
which will be used frequently later on.
Denote its indcategory by $\Ind \mathcal{A}$
(or more precisely, the $\mathcal{U}_{0}$-indcategory,
where index sets are $\mathcal{U}_{0}$-small).
The category $\Ind \mathcal{A}$ is abelian by
\cite[Thm.\ 8.6.5 (i)]{KS06}.
Note that objects of $\mathcal{A}$ need not be artinian in $\Ind \mathcal{A}$.
For instance, the additive algebraic group scheme $\Ga$ over a field of characteristic $p$
contains the ind-finite-\'etale subgroup $\closure{\F_{p}}$.
We say that an object $A \in \Ind \mathcal{A}$ is of \emph{cofinite type}
if the $p^{n}$-torsion part $A[p^{n}]$ is in $\mathcal{A}$ for all $n \ge 0$.
In this case, the equality $A = \dirlim A[p^{n}]$ gives a presentation of $A$
as a filtered direct limit of objects of $\mathcal{A}$.
Denote by $\Ind^{f} \mathcal{A}$
the full subcategory of $\Ind \mathcal{A}$ consisting of objects of cofinite type.

\begin{Prop} \label{prop: cofinite type mod n is finite type}
	For any $A \in \Ind^{f} \mathcal{A}$,
	we have $A / p^{n} A \in \mathcal{A}$ for all $n \ge 0$.
\end{Prop}

\begin{proof}
	Let $A \in \Ind^{f} \mathcal{A}$.
	To show $A / p^{n} A \in \mathcal{A}$, it is enough to show this for $n = 1$.
	The object $A / p A$ is the union of the increasing sequence of subobjects
	$A[p^{m}] / p (A[p^{m + 1}])$ each in $\mathcal{A}$.
	It is enough to show that this sequence stabilizes,
	or $A[p^{m - 1}] + p (A[p^{m + 1}]) = A[p^{m}]$ for large $m$.
	This is equivalent that the decreasing sequence
		\[
				\cdots
			\overset{p}{\into}
				A[p^{m + 1}] / A[p^{m}]
			\overset{p}{\into}
				A[p^{m}] / A[p^{m - 1}]
			\overset{p}{\into}
				\cdots
		\]
	stabilizes.
	It indeed does since each object in the sequence belongs to the artinian category $\mathcal{A}$.
\end{proof}

\begin{Prop}
	The category $\Ind^{f} \mathcal{A}$ is an abelian subcategory of $\Ind \mathcal{A}$.
\end{Prop}

\begin{proof}
	Let $f \colon A \to B$ be a morphism in $\Ind^{f} \mathcal{A}$.
	Then $\Ker(f)[p^{n}]$ for any $n$ is the kernel of the restriction $A[p^{n}] \to B[p^{n}]$ of $f$.
	Hence $\Ker(f)[p^{n}] \in \mathcal{A}$, thus $\Ker(f) \in \Ind^{f} \mathcal{A}$.
	Therefore in the exact sequence
	$0 \to \Ker(f) \to A \to \Im(f) \to 0$,
	the kernel and cokernel of $p^{n}$ on the first two terms are in $\mathcal{A}$.
	Hence $\Im(f) \in \Ind^{f} \mathcal{A}$.
	The same argument for $0 \to \Im(f) \to B \to \Coker(f) \to 0$ implies that
	$\Coker(f) \in \Ind^{f} \mathcal{A}$.
	Hence $\Ind^{f} \mathcal{A}$ is an abelian subcategory of $\Ind \mathcal{A}$.
\end{proof}

\begin{Prop}
	For any $A \in \Ind^{f} \mathcal{A}$,
	the decreasing sequence
		\[
				\cdots
			\into
				p^{2} A
			\into
				p A
			\into
				A
		\]
	stabilizes.
\end{Prop}

\begin{proof}
	For any $n$, we have $p^{n} \colon A / (A[p^{n}] + p A) \isomto p^{n} A / p^{n + 1} A$.
	The isomorphism
		\[
				A / p A
			\cong
				\dirlim_{n}
					(A[p^{n}] + p A) / p A
		\]
	in $\Ind \mathcal{A}$ factors through $(A[p^{n}] + p A) / p A$ for some $n$
	since $A / p A \in \mathcal{A}$.
	Hence $A / p A \cong (A[p^{n}] + p A) / p A$ for all large $n$.
	For such $n$, we have $A = A[p^{n}] + p A$ and $p^{n + 1} A = p^{n} A$.
	Hence the sequence stabilizes.
\end{proof}

For $A \in \Ind^{f} \mathcal{A}$, we define
	\begin{gather*}
				A_{\divis}
			=
				\bigcap_{n} p^{n} A
			\quad
				(= p^{n} A \text{ for some $n$}),
		\\
				A_{/ \divis}
			=
				A / A_{\divis}.
	\end{gather*}

\begin{Prop} \label{prop: divisible subgroup}
	If $A \in \Ind^{f} \mathcal{A}$,
	then $A_{\divis} \in \Ind^{f} \mathcal{A}$ is divisible
	and $A_{/ \divis} \in \mathcal{A}$.
	The sequence
		\[
			0 \to A_{\divis} \to A \to A_{/ \divis} \to 0
		\]
	is exact.
	The object $A_{\divis}$ is the largest divisible sub-object of $A$.
	The object $A_{/ \divis}$ is the largest quotient of $A$ that belongs to $\mathcal{A}$.
\end{Prop}

\begin{proof}
	Obvious.
\end{proof}

\begin{Prop}
	The category $\Ind^{f} \mathcal{A}$ is artinian.
\end{Prop}

\begin{proof}
	By the previous proposition,
	an object of $\mathcal{A}$ is artinian in $\Ind^{f} \mathcal{A}$
	since an object of $\Ind^{f} \mathcal{A}$ embeddable into an object of $\mathcal{A}$
	does not have divisible part and hence itself is in $\mathcal{A}$.
	It is enough to show that a decreasing sequence
	$A_{0} \supset A_{1} \supset \cdots$
	of divisible objects in $\Ind^{f} \mathcal{A}$ stabilizes.
	Since $A_{n}[p] \in \mathcal{A}$ is artinian for any $n$,
	there is some $m$ such that $A_{m}[p] = A_{m + 1}[p] = \cdots$.
	From the exact sequence
		\[
			0 \to A_{n}[p] \to A_{n}[p^{2}] \overset{p}{\to} A_{n}[p] \to 0,
		\]
	we know that $A_{m}[p^{2}] = A_{m + 1}[p^{2}] = \cdots$ for the same $m$.
	Inductively, we have $A_{m} = A_{m + 1} = \cdots$.
\end{proof}

Denote the pro-category of $\mathcal{A}$ by $\Pro \mathcal{A}$
and the ind-category of $\Pro \mathcal{A}$ by $\Ind \Pro \mathcal{A}$.
For $A \in \Ind^{f} \mathcal{A}$,
we define
	\begin{gather*}
				T A
			=
				\invlim_{n} A[p^{n}]
			\in
				\Pro \mathcal{A},
			\quad
		\\
				V A
			=
				\dirlim_{m} T A
			\in
				\Ind \Pro \mathcal{A},
	\end{gather*}
where the direct limit in the second definition is for multiplication by $p^{m}$.
We call $T A$ the Tate module of $A$ and $V A$ the rational Tate module of $A$.
For example, if $A$ is $\Q_{p} / \Z_{p}$ in the category of torsion abelian groups,
then $T A = \Z_{p}$ in the category of profinite abelian groups
and $V A = \Q_{p}$ as an ind-object of profinite abelian groups.
For general $\mathcal{A}$, if $A \in \mathcal{A}$,
then the system $\{A[p^{n}]\}$ that defines $TA$ has essentially zero transition morphisms,
hence $TA = 0$.
For each $m$, we consider the morphism
$TA \to A$ in $\Ind \Pro \mathcal{A}$ given by $(a_{n})_{n} \mapsto a_{m}$.
They form a morphism $VA \to A$ in $\Ind \Pro \mathcal{A}$.

\begin{Prop}
	The functor $A \mapsto VA$ from $\Ind^{f} \mathcal{A}$ to $\Ind \Pro \mathcal{A}$ is exact.
\end{Prop}

\begin{proof}
	Let $0 \to A \to B \to C \to 0$ be exact in $\Ind^{f} \mathcal{A}$.
	For any $n$, we have an exact sequence
		\[
			0 \to A[p^{n}] \to B[p^{n}] \to C[p^{n}]
			\to A / p^{n} A \to B / p^{n} B \to C / p^{n} C \to 0
		\]
	in $\mathcal{A}$.
	The inverse limit in $n$ in $\Pro \mathcal{A}$ gives an exact sequence
		\[
			0 \to TA \to TB \to TC
			\to A_{/ \divis} \to B_{/ \divis} \to C_{/ \divis} \to 0
		\]
	in $\Pro \mathcal{A}$.
	(Note that Mittag-Leffler conditions are not relevant here
	since filtered inverse limits in a pro-category are exact by definition.)
	Since each $(\var)_{/ \divis} \in \mathcal{A}$ is killed by multiplication by some power of $p$,
	the direct limit in multiplication by $m$ gives the desired exact sequence
	$0 \to VA \to VB \to VC \to 0$ in $\Ind \Pro \mathcal{A}$.
\end{proof}

\begin{Prop} \label{prop: Tate module exact sequence}
	Let $A \in \Ind^{f} \mathcal{A}$.
	Then we have an exact sequence
		\[
			0 \to TA \overset{p^{n}}{\to} TA \to A_{\divis} \cap A[p^{n}] \to 0
		\]
	in $\Pro \mathcal{A}$ for any $n \ge 0$ and an exact sequence
		\[
			0 \to TA \to VA \to A_{\divis} \to 0
		\]
	in $\Ind \Pro \mathcal{A}$.
	We have $TA = T(A_{\divis})$ and $VA = V(A_{\divis})$.
\end{Prop}

\begin{proof}
	The multiplication-by-$p^{n}$ map gives an exact sequence
	$0 \to A[p^{n}] \to A \to p^{n} A \to 0$.
	We have $T(A[p^{n}]) = 0$ and $A[p^{n}]_{/ \divis} = A[p^{n}]$
	since $A[p^{n}] \in \mathcal{A}$ has no divisible part.
	Hence the long exact sequence in the proof of the previous proposition for this sequence is
		\[
					0
				\to
					TA
				\to
					T(p^{n} A)
				\to
					A[p^{n}]
				\to
					A_{/ \divis}
				\to
					(p^{n} A)_{/ \divis}
				\to
					0.
		\]
	in $\Pro \mathcal{A}$.
	Since $A / p^{n} A \in \mathcal{A}$, we have
	$T(p^{n} A) = TA$.
	From this, we get the first exact sequence.
	The direct limit in $n$ gives the second exact sequence.
	We have $T(A_{/ \divis}) = 0$ since $A_{/ \divis} \in \mathcal{A}$.
	It follows that $TA = T(A_{\divis})$ and $VA = V(A_{\divis})$.
\end{proof}

Next, let $\mathcal{A}$ be an artinian abelian category such that
any object is killed by multiplication by some positive integer (not necessarily a prime power).
For each prime $p$, denote by $\mathcal{A}_{p}$
the full subcategory of $\mathcal{A}$ of objects killed by a power of $p$.
This is an artinian abelian category.
Any object of $\mathcal{A}$ can be canonically decomposed as $A = \bigoplus_{p} A_{p}$,
where each summand $A_{p}$ belongs to $\mathcal{A}_{p}$
and $A_{p} = 0$ for almost all $p$.
We call $A_{p}$ the $p$-primary part of $A$.
Consider the ind-category $\Ind \mathcal{A}$.
Any object of $\Ind \mathcal{A}$ can be canonically decomposed as
$A = \bigoplus_{p} A_{p}$ (a filtered direct limit of finite partial sums),
where each summand $A_{p}$ belongs to $\Ind \mathcal{A}_{p} := \Ind(\mathcal{A}_{p})$.

We say that an object $A \in \Ind \mathcal{A}$ is of \emph{cofinite type}
if the $n$-torsion part $A[n]$ is in $\mathcal{A}$ for all $n \ge 1$.
In this case, the equality $A = \dirlim A[n]$ gives a presentation of $A$
as a filtered direct limit of objects of $\mathcal{A}$.
Denote by $\Ind^{f} \mathcal{A}$
the full subcategory of $\Ind \mathcal{A}$ consisting of objects of cofinite type.
An object $A \in \Ind \mathcal{A}$ is of cofinite type if and only if
$A_{p}$ is of cofinite type for all $p$.
Note that $\Ind^{f} \mathcal{A}$ is not necessarily artinian.
For instance, the object $\bigoplus_{p} \Z / p \Z$
in the category of torsion abelian groups of cofinite type is not artinian.
For $A \in \Ind^{f} \mathcal{A}$, we define
	\[
			A_{\divis}
		=
			\bigoplus_{p} (A_{p})_{\divis}
		\in
			\Ind \mathcal{A}
		\quad
			A_{/ \divis}
		=
			A / A_{\divis}
		\in
			\Ind \mathcal{A},
	\]
where $(A_{p})_{\divis} \in \Ind \mathcal{A}_{p}$ is previously defined.
We also define
	\[
			T A
		=
			\invlim_{n} A[n]
		\in
			\Pro \mathcal{A},
		\quad
			V A
		=
			\dirlim_{m} T A
		\in
			\Ind \Pro \mathcal{A},
	\]
where the inverse limit in the first definition is over multiplication by $n \ge 1$
and the direct limit in the second is over multiplication by $m \ge 1$.
Also define $T_{p} A = T(A_{p})$, $V_{p} A = V(A_{p})$.

\begin{Prop} \BetweenThmAndList
	\begin{enumerate}
		\item
			For any $A \in \Ind^{f} \mathcal{A}$ and $n \ge 1$,
			we have $A / n A \in \mathcal{A}$.
		\item
			The category $\Ind^{f} \mathcal{A}$ is an abelian subcategory of $\Ind \mathcal{A}$.
		\item
			If $A \in \Ind^{f} \mathcal{A}$,
			then $A_{\divis}$ is the largest divisible sub-object of $A$,
			and $A_{/ \divis}$ is the largest quotient of $A$
			such that $(A_{/ \divis})_{p}$ belongs to $\mathcal{A}_{p}$ for all $p$.
			The sequence
				\[
					0 \to A_{\divis} \to A \to A_{/ \divis} \to 0
				\]
			is an exact sequence in $\Ind^{f} \mathcal{A}$.
		\item
			Let $A \in \Ind^{f} \mathcal{A}$.
			Then we have exact sequences
				\[
					0 \to TA \overset{n}{\to} TA \to A_{\divis} \cap A[n] \to 0
				\]
			in $\Pro \mathcal{A}$ for any $n \ge 1$ and
				\[
					0 \to TA \to VA \to A_{\divis} \to 0
				\]
			in $\Ind \Pro \mathcal{A}$.
			We have $TA = T(A_{\divis}) = \prod_{p} T_{p} A$ and $VA = V(A_{\divis})$.
	\end{enumerate}
\end{Prop}

\begin{proof}
	All the statements are reduced to the $p$-primary parts for primes $p$,
	which have already been proved.
\end{proof}

Contrary to the $p$-primary case,
$A_{/ \divis}$ might not be in $\mathcal{A}$ for general $A \in \Ind^{f} \mathcal{A}$.
An example is $\bigoplus_{p} \Z / p \Z$.
This is related to the Mittag-Leffler condition.
Here is a general fact.

\begin{Prop} \label{prop: divisibly ML objects}
	Let $B$ be an object of any abelian category $\mathcal{B}$.
	Consider the inverse system in $\mathcal{B}$
	given by multiplication maps by positive integers on $B$.
	This system is Mittag-Leffler if and only if
	$B$ has a divisible subobject $B'$ such that
	$B / B'$ is killed by multiplication by some positive integer.
	In this case, $B'$ is the maximal divisible subobject of $B$,
	and we say that $B$ is \emph{divisibly ML}.
\end{Prop}

\begin{proof}
	Elementary.
\end{proof}

\begin{Prop}
	Let $A \in \Ind^{f} \mathcal{A}$.
	Then $A$ is divisibly ML if and only if
	$A_{/ \divis} \in \mathcal{A}$ if and only if
	$A_{p}$ is divisible for almost all $p$.
\end{Prop}

\begin{proof}
	We have
		\[
				A_{/ \divis}
			=
					\bigoplus_{p} A_{p}
				\Big/
					\bigoplus_{p} (A_{p})_{\divis}
			=
				\bigoplus_{p} (A_{p})_{/ \divis}.
		\]
	The result follows from this.
\end{proof}

Note that divisibly ML objects in $\Ind^{f} \mathcal{A}$ do not form an abelian subcategory.
For instance, the direct sum of the exact sequences
$0 \to \Z / p \Z \to \Q_{p} / \Z_{p} \to \Q_{p} / \Z_{p} \to 0$
for primes $p$ gives an exact sequence
$0 \to \bigoplus_{p} \Z / p \Z \to \Q / \Z \to \Q / \Z \to 0$,
where the last two terms are divisible but the first term is not divisibly ML.


\subsection{Ind-algebraic groups of cofinite type}
\label{sec: Ind-algebraic groups of cofinite type}
Let $k$ ($\in \mathcal{U}_{0}$) be a perfect field of characteristic $p > 0$.
We quickly recall some notation about perfections of algebraic groups from \cite[\S 2.1]{Suz14}.
Let $\Alg / k$ be the category of quasi-algebraic groups over $k$
(commutative, as assumed throughout the paper)
in the sense of Serre \cite{Ser60}.
Recall that a quasi-algebraic group is the perfection (inverse limit along Frobenii)
of a (not necessarily connected) algebraic group \cite[\S 1.2, D\'ef.\ 2; \S 1.4, Prop.\ 10]{Ser60}.
For example, the perfection of the additive group $\Spec k[x]$
is $\Spec k[x, x^{1 / p}, x^{1 / p^{2}}, \dots]$,
which we simply call the additive group and denote by $\Ga$ by abuse of notation.
We say that a quasi-algebraic group is a unipotent group,
a torus or a (semi-)abelian variety if it is the perfection of such a group.
We call an object of the pro-category $\Pro \Alg / k$ a pro-algebraic group,
an object of the ind-category $\Ind \Alg / k$ an ind-algebraic group
and an object of the ind-pro-category $\Ind \Pro \Alg / k = \Ind (\Pro \Alg / k)$
an ind-pro-algebraic group.
Let $\Alg_{\uc} / k$ be the full subcategory of $\Alg / k$
consisting of groups whose identity component is unipotent.
Let $\Loc \Alg / k$ be the category of perfections of smooth group schemes over $k$.
A finitely generated \'etale group is an \'etale group
with a finitely generated group of geometric points.
A lattice is a finitely generated \'etale group with no torsion.
Let $\FEt / k \subset \FGEt / k \subset \Et / k$ be
the categories of finite \'etale groups,
finitely generated \'etale groups,
\'etale groups, respectively, over $k$.
For any prime $l$ (possibly equal to $p$),
we denote the full subcategory of $\FEt / k$ of groups of $l$-power order by $\FEt_{l} / k$.
For $A \in \Loc \Alg / k$, we denote its identity component by $A_{0} \in \Alg / k$
and set $\pi_{0}(A) = A / A_{0} \in \Et / k$
(\cite[II, \S 5, Prop.\ 1.8]{DG70b} plus perfection,
noting that perfection does not change the underlying topological space).
The endofunctor $A \mapsto A_{0}$ on $\Alg / k$
extends to endofunctors on $\Pro \Alg / k$, $\Ind \Alg / k$ and $\Ind \Pro \Alg / k$,
still denoted by $A \mapsto A_{0}$.
We say that $A \in \Ind \Pro \Alg / k$ is connected if $A = A_{0}$.
For $A \in \Ind \Pro \Alg / k$, we define $\pi_{0}(A) = A / A_{0} \in \Ind \Pro \FEt / k$.

The category $\Alg_{\uc} / k$ is an artinian abelian category
such that any object is killed by multiplication by some positive integer.
Hence we can apply the results and notation in the previous subsection to $\mathcal{A} = \Alg_{\uc} / k$.

\begin{Prop}
	For any $A \in \Ind^{f} \Alg_{\uc} / k$,
	we have $A_{0}, \pi_{0}(A) \in \Ind^{f} \Alg_{\uc} / k$.
\end{Prop}

\begin{proof}
	We may assume that $A$ is unipotent.
	Consider the sequence
		\begin{align*}
					0
			&	\to
					A_{0}[p^{n}]
				\to
					A[p^{n}]
				\to
					\pi_{0}(A)[p^{n}]
			\\
			&	\to
					A_{0} / p^{n} A_{0}
				\to
					A / p^{n} A
				\to
					\pi_{0}(A) / p^{n} \pi_{0}(A)
				\to
					0
		\end{align*}
	for any $n \ge 0$.
	Since $A[p^{n}]$ is quasi-algebraic and $\pi_{0}(A)[p^{n}]$ is \'etale,
	the image of $A[p^{n}] \to \pi_{0}(A)[p^{n}]$ is finite
	and hence $A_{0}[p^{n}]$ is quasi-algebraic.
	Therefore $A_{0} \in \Ind^{f} \Alg_{\uc} / k$
	and so $A_{0} / p^{n} A_{0}$ is quasi-algebraic
	by Prop.\ \ref{prop: cofinite type mod n is finite type}.
	This implies that $\pi_{0}(A)[p^{n}]$ is quasi-algebraic (i.e.\ finite \'etale)
	and $\pi_{0}(A) \in \Ind^{f} \Alg_{\uc} / k$.
\end{proof}

We define $\Loc^{f} \Alg_{\uc} / k$ to be the full subcategory of $\Ind^{f} \Alg_{\uc} / k$
of objects whose identity component is quasi-algebraic (i.e.\ belongs to $\Alg_{\uc} / k$).
Equivalently, an object $A \in \Loc^{f} \Alg_{\uc} / k$ is
an extension of a torsion \'etale group of cofinite type ($= \pi_{0}(A)$)
by a connected unipotent quasi-algebraic group ($= A_{0}$).
Such objects can naturally be viewed as objects of $\Loc \Alg / k$;
see \cite[\S 2.1, Footnote 3]{Suz13}.

\begin{Prop}
	Let $A \in \Ind^{f} \Alg_{\uc} / k$.
	Then $(A_{0})_{\divis} = (A_{\divis})_{0}$.
	We denote these isomorphic objects by $A_{0 \divis}$.
\end{Prop}

\begin{proof}
	We may assume that $A$ is unipotent.
	It is enough to show that $(A_{0})_{\divis}$ is connected and $(A_{\divis})_{0}$ is divisible.
	Take a power $p^{n}$ of $p$ that kills $(A_{0})_{/ \divis} \in \Alg_{\uc} / k$.
	Then multiplication by $p^{n}$ gives a surjection
	$A_{0} \onto (A_{0})_{\divis}$.
	Hence $(A_{0})_{\divis}$ is connected.
	On the other hand, the exact sequence
	$0 \to (A_{\divis})_{0} \to A_{\divis} \to \pi_{0}(A_{\divis}) \to 0$
	and the snake lemma gives a surjection
	$\pi_{0}(A_{\divis})[p] \onto (A_{\divis})_{0} / p ((A_{\divis})_{0})$.
	This implies that $(A_{\divis})_{0} / p ((A_{\divis})_{0})$ is \'etale and connected,
	hence zero.
	Thus $(A_{\divis})_{0}$ is divisible.
\end{proof}

An example of a connected divisible group in $\Ind^{f} \Alg_{\uc} / k$
is the direct limit $\dirlim_{n} W_{n}$ of (perfections of)
groups of $p$-typical Witt vectors of finite length.
We have $T \dirlim_{n} W_{n} = \invlim_{n} W_{n}$,
which is the group $W$ of $p$-typical Witt vectors of infinite length.

\begin{Prop} \label{prop: criteiron for algebraicity of identity component}
	Let $A \in \Ind^{f} \Alg_{\uc} / k$.
	Then the following are equivalent.
	\begin{enumerate}
		\item \label{item: A is in Loc f}
			$A \in \Loc^{f} \Alg_{\uc} / k$.
		\item \label{item: A has trivial div 0 part}
			$A_{0 \divis} = 0$.
		\item \label{item: A has ind pro etale p adic Tate module}
			$V_{p} A \in \Ind \Pro \FEt / k$.
			(Recall that $V_{p} A$ is defined in \S \ref{sec: Ind-objects of cofinite type}.)
		\item \label{item: A has constant p adic Tate module on alg cl fields}
			$V_{p} A(k')$ as a functor on algebraically closed fields $k'$ over $k$ is constant.
	\end{enumerate}
\end{Prop}

\begin{proof}
	We may assume that $A$ is unipotent.
	The equivalence between \eqref{item: A is in Loc f} and \eqref{item: A has trivial div 0 part}
	follows from the previous proposition.
	We know that $V_{p} A = V_{p} (A_{\divis})$ surjects onto $A_{\divis}$.
	Hence $(V_{p} A)_{0} = (V_{p}(A_{\divis}))_{0}$ surjects onto $A_{0 \divis}$.
	This shows that \eqref{item: A has ind pro etale p adic Tate module} is equivalent to
	\eqref{item: A has trivial div 0 part}.
	Obviously \eqref{item: A has ind pro etale p adic Tate module} implies
	\eqref{item: A has constant p adic Tate module on alg cl fields}.
	
	We show that \eqref{item: A has constant p adic Tate module on alg cl fields} implies
	\eqref{item: A has trivial div 0 part}.
	We may assume that $k = \closure{k}$ and $A$ is divisible.
	Then the assumption implies that
	$V_{p} A(k')$ as a functor on arbitrary perfect fields $k'$ over $k$ is constant.
	The surjection $V_{p} A \onto A$ implies that
	$A(k')$ as a functor on perfect fields $k'$ over $k$ is constant.
	This implies that $A_{0}(k') $ as a functor on perfect fields $k'$ over $k$ is constant.
	Hence we may further assume that $A$ is connected.
	We then want to show that $A = 0$
	if $A(k')$ as a functor on perfect fields $k'$ over $k$ is constant.
	Let $A_{n} = (A[p^{n}])_{0} \in \Alg_{\uc} / k$.
	Then $A = \dirlim_{n} A_{n}$.
	It follows that the generic point of $A_{n}$ for any $n$ maps to a $k$-value point of $A$
	and hence to a $k$-valued point of $A_{m}$ for some $m \ge n$.
	This means that the injective morphism $A_{n} \into A_{m}$ of quasi-algebraic groups is generically constant.
	Therefore $A_{n} = 0$ for any $n$ and $A = 0$.
\end{proof}

\begin{Prop}
	The category $\Loc^{f} \Alg_{\uc} / k$ is an abelian subcategory of $\Ind^{f} \Alg_{\uc} / k$.
\end{Prop}

\begin{proof}
	This follows from the equivalence between \eqref{item: A is in Loc f} and
	\eqref{item: A has ind pro etale p adic Tate module}
	(or \eqref{item: A has constant p adic Tate module on alg cl fields})
	of the previous proposition.
\end{proof}

\begin{Prop} \label{prop: connected divisible part}
	Let $A \in \Ind^{f} \Alg_{\uc} / k$
	and define $A_{0 \cap \divis} := A_{0} \cap A_{\divis}$.
	Then $A_{0 \cap \divis} / A_{0 \divis} \in \FEt_{p} / k$.
	We have exact sequences
		\begin{gather*}
				0 \to A_{0 \cap \divis} \to A_{0} \to (A_{/ \divis})_{0} \to 0,
			\\
				0 \to A_{0 \cap \divis} \to A_{\divis} \to (\pi_{0} A)_{\divis} \to 0.
		\end{gather*}
\end{Prop}

\begin{proof}
	Since $A_{0} / A_{0 \divis} = (A_{0})_{/ \divis}$ is quasi-algebraic unipotent,
	it is killed by some power $p^{n}$ of $p$.
	Hence $A_{0 \cap \divis} / A_{0 \divis}$ is a subgroup of
	the $p^{n}$-torsion part of $A_{\divis} / A_{0 \divis} = \pi_{0}(A_{\divis})$.
	But $\pi_{0}(A_{\divis})$ is a torsion \'etale group of cofinite type.
	Hence its $p^{n}$-torsion part is finite.
	Therefore $A_{0 \cap \divis} / A_{0 \divis}$ is finite.
	The exactness of the sequences is clear.
\end{proof}

The group $A_{0 \cap \divis}$ in the proposition is in general non-zero
even if $A \in \Loc^{f} \Alg_{\uc} / k$.
It is $\Z / p \Z$ if $A$ is the cokernel of the diagonal embedding of
$\Z / p \Z$ into $\Ga \oplus (\Q_{p} / \Z_{p})$.
In particular, $A_{0 \cap \divis}$ can be neither connected nor divisible.
On the other hand, $A_{0 \divis}$ is always connected and divisible.

\begin{Prop}
	For any $A \in \Ind^{f} \Alg_{\uc} / k$,
	we have $\pi_{0}(A)_{/ \divis} = \pi_{0}(A_{/ \divis})$.
	Denote these isomorphic objects by $\pi_{0} A_{/ \divis}$.
	The kernel of $A \onto \pi_{0} A_{/ \divis}$ is
	$A_{0 + \divis} := A_{0} + A_{\divis}$.
\end{Prop}

\begin{proof}
	Obvious.
\end{proof}

For example, if $A = \Ga \oplus \Q / \Z \oplus \Z / p \Z$,
then $A_{0 + \divis} = \Ga \oplus \Q / \Z$ and $\pi_{0} A_{\divis} = \Z / p \Z$.
To see a more non-trivial example,
one may work out the subgroup $A$ of $\dirlim_{n} W_{n} = \{(\dots, x_{-2}, x_{-1}, x_{0}) \mid x_{n} = 0 \text{ for almost all } n \}$
defined by $(F - 1) x_{n} = (F - 1)^{2} x_{-1} = 0$ for all $n \le -2$,
where $F$ is the Frobenius.
Then $A_{0} \cong \Ga$ is defined by $x_{n} = 0$ for all $n \le -1$,
$A_{\divis} \cong \Q_{p} / \Z_{p}$ by $(F - 1) x_{n} = 0$ for all $n \le 0$,
$A_{0 + \divis}$ by $(F - 1) x_{n} = 0$ for all $n \le -1$,
and $\pi_{0} A_{/ \divis} \cong \Z / p \Z$.

Here is a connection diagram between the named subgroups of $A \in \Ind^{f} \Alg_{\uc} / k$:
	\[
		\begin{CD}
				0
			@> \text{conn div} >>
				A_{0 \divis}
			@> \FEt_{p} >>
				A_{0 \cap \divis}
			@> \text{\'et div} >>
				A_{\divis}
			@.
			\\
			@.
			@.
			@V \text{conn alg} VV
			@VV \text{conn alg} V
			@.
			\\
			@.
			@.
				A_{0}
			@>> \text{\'et div} >
				A_{0 + \divis}
			@>> \bigoplus_{l} \FEt_{l} >
				A.
		\end{CD}
	\]
All the arrows are inclusions.
The label of an arrow means that the subquotient there is of that type.
For example, $A_{0 + \divis} / A_{\divis}$ is connected quasi-algebraic.
The symbol $\bigoplus_{l} \FEt_{l}$ means that
the (sub)quotient there is a direct sum of
finite \'etale $l$-primary groups for primes $l$ (with $l = p$ allowed),
which is finite if and only if $A$ is divisibly ML.
The upper and lower sides of the square have the same subquotient,
and the left and right sides of the square have the same subquotient.
The two step subquotient $A_{0} / A_{0 \divis}$ is connected quasi-algebraic and,
similarly, $A_{\divis} / A_{0 \divis}$ is \'etale divisible.
In particular, we have $\pi_{0}(A_{0 + \divis}) = \pi_{0}(A)_{\divis}$
and $(A_{0 + \divis})_{/ \divis} = (A_{/ \divis})_{0}$.
Later we will see in Prop.\ \ref{prop: structure of cohomology of Neron models}
that the cohomology object $G$ in any positive degree
of the complex $R \alg{\Gamma}(X, \mathcal{A})$ mentioned in Introduction
belongs to $\Ind^{f} \Alg_{\uc} / k$.
With some more effort, we will see in Thm.\ \ref{thm: duality for Neron} that
$G_{0 \divis} = 0$ (i.e.\ $G \in \Loc^{f} \Alg_{\uc} / k$)
and $G / G_{0 + \divis}$ is finite (i.e.\ $G$ is divisibly ML).


\subsection{The ind-rational pro-\'etale site and some derived limits}
\label{sec: The ind-rational pro-etale site and some derived limits}
We quickly recall some definitions and notation
about the ind-rational (pro-)\'etale site from \cite[\S 2.1]{Suz14}.
We say that a $k$-algebra is \emph{rational}
if it is a finite direct product of perfections (direct limit along Frobenii)
of finitely generated fields over $k$,
and \emph{ind-rational} if it is a filtered union of rational $k$-subalgebras.
We denote the category of rational (resp.\ ind-rational) $k$-algebras with $k$-algebra homomorphisms
by $k^{\rat}$ (resp.\ $k^{\ind\rat}$).
(More precisely, we only consider $\mathcal{U}_{0}$-small ind-rational $k$-algebras.
Hence $k^{\ind\rat}$ is a $\mathcal{U}$-small $\mathcal{U}_{0}$-category.)
We can endow $k^{\ind\rat}$ with the \'etale topology.
That is, an \'etale covering of $k' \in k^{\ind\rat}$ is a finite family $\{k'_{i}\}$ of \'etale $k'$-algebras
such that the product $\prod k'_{i}$ is faithfully flat over $k'$.
The resulting site is the \emph{ind-rational \'etale site} $\Spec k^{\ind\rat}_{\et}$.
We can also endow $k^{\ind\rat}$ with the pro-\'etale topology of Bhatt-Scholze \cite{BS15}.
That is, a pro-\'etale covering of $k' \in k^{\ind\rat}$ is a finite family $\{k'_{i}\}$ of $k'$-algebras
such that each $k'_{i}$ is a filtered direct limit of \'etale $k'$-algebras
and the product $\prod k'_{i}$ is faithfully flat over $k'$.
The resulting site is the \emph{ind-rational pro-\'etale site} $\Spec k^{\ind\rat}_{\pro\et}$.
The category of sheaves of sets (resp.\ abelian groups) on $\Spec k^{\ind\rat}_{\pro\et}$ is
denoted by $\Set(k^{\ind\rat}_{\pro\et})$ (resp.\ $\Ab(k^{\ind\rat}_{\pro\et})$).
(Here, the target categories $\Set$ and $\Ab$ for sheaves are
the categories of $\mathcal{U}$-small sets and abelian groups.)
The category of complexes in $\Ab(k^{\ind\rat}_{\pro\et})$ is denoted by $\Ch(k^{\ind\rat}_{\pro\et})$,
its homotopy category by $K(k^{\ind\rat}_{\pro\et})$ and
its derived category by $D(k^{\ind\rat}_{\pro\et})$.
See \cite[Chap.\ 18]{KS06} for details about unbounded derived categories of sheaves on sites.
The cohomology of $k' \in k^{\ind\rat}$ with coefficients in $G \in D(k^{\ind\rat}_{\pro\et})$
is denoted by $R \Gamma(k'_{\pro\et}, G)$,
with $n$-th cohomology $H^{n}(k'_{\pro\et}, G)$ when $G \in \Ab(k^{\ind\rat}_{\pro\et})$.
The sheaf-Hom, $n$-th sheaf-Ext and derived sheaf-Hom functors are denoted by
$\sheafhom_{k^{\ind\rat}_{\pro\et}}$, $\sheafext_{k^{\ind\rat}_{\pro\et}}^{n}$
and $R \sheafhom_{k^{\ind\rat}_{\pro\et}}$, respectively.
Similar notation applies to $\Spec k^{\ind\rat}_{\et}$.
For $G \in D(k^{\ind\rat}_{\pro\et})$, we define its \emph{Serre dual} \cite[\S 2.4]{Suz14} to be
	\[
			G^{\SDual}
		=
			R \sheafhom_{k^{\ind\rat}_{\pro\et}}(G, \Z).
	\]
See the list of examples of Serre duals
in the paragraph after the proof of \cite[Prop.\ (2.4.1)]{Suz14}.
A key example there is $\Ga^{\SDual} \cong \Ga[-2]$, which comes from
$\Ga \isomto \sheafext_{k^{\ind\rat}_{\pro\et}}^{1}(\Ga, \Z / p \Z)$
defined by sending $c \in \Ga$ to the pullback of the Artin-Schreier extension class
$0 \to \Z / p \Z \to \Ga \to \Ga \to 0$
by the multiplication-by-$c$ map $\Ga \to \Ga$ to the final term.
In particular, $\Ga$ is Serre reflexive.
We say that $G$ is \emph{Serre reflexive}
if the canonical morphism $G \to G^{\SDual \SDual}$ is an isomorphism.
For $G \in \Ab(k^{\ind\rat}_{\pro\et})$, we denote its torsion part by $G_{\tor}$
and set $G_{/ \tor} = G / G_{\tor}$.
We say that a sheaf $F \in \Set(k^{\ind\rat}_{\et})$ is locally of finite presentation
if it commutes with filtered direct limits as a functor $k^{\ind\rat} \to \Set$ (\cite[\S 2.4]{Suz14}).
In this case, $F$ is automatically a sheaf for the pro-\'etale topology.

The natural Yoneda functor $\Alg / k \to \Ab(k^{\ind\rat}_{\pro\et})$
extends to an additive functor
$\Ind \Pro \Alg / k \to \Ab(k^{\ind\rat}_{\pro\et})$.
(Here, as before, we consider
the $\mathcal{U}_{0}$-procategory of $\Alg / k$ and
the $\mathcal{U}_{0}$-indcategory of the $\mathcal{U}_{0}$-category $\Pro \Alg / k$.)
This functor is exact by \cite[Prop.\ (2.1.2) (e)]{Suz14},
which is fully faithful and induces a fully faithful embedding
$D^{b}(\Ind \Pro \Alg / k) \into D(k^{\ind\rat}_{\pro\et})$
by \cite[Prop.\ (2.3.4)]{Suz14}.
If $G \in \Ind \Pro \Alg / k$ is an extension of an object of $\Ind \Alg_{\uc} / k$ by an object of $\Pro \Alg / k$,
then we denote
	\[
			G^{\SDual'}
		=
			\sheafext_{k^{\ind\rat}_{\pro\et}}^{1}(G, \Q / \Z).
	\]
By \cite[Prop.\ (2.4.1) (b)]{Suz14}, we have $G^{\SDual} = G^{\SDual'}[-2]$.
The same proposition shows that
if $G$ is connected unipotent proalgebraic or ind-algebraic,
then $G^{\SDual'}$ is connected unipotent ind-algebraic or proalgebraic, respectively,
and $G \isomto G^{\SDual' \SDual'}$.
In particular, $G$ is Serre reflexive.
By \cite[Prop.\ (2.4.1) (b)]{Suz14},
if $G$ is semi-abelian, then $G^{\SDual'}$ is the Pontryagin dual of the Tate module $T G$.
If $G$ is not necessarily connected,
then we denote $G_{0}^{\SDual'} = (G_{0})^{\SDual'}$.
If $G$ is an extension of a torsion \'etale group by a pro-finite-\'etale group,
then its Pontryagin dual is denoted by $G^{\PDual}$,
which can be given by $\sheafhom_{k^{\ind\rat}_{\pro\et}}(G, \Q / \Z)$,
or by $G^{\SDual}[1]$ by \cite[Prop.\ (2.4.1) (b)]{Suz14}.
For example, if $G = \Q_{p}$, then $G^{\PDual} \cong \Q_{p}$.
If the torsion part of $G \in \Ab(k^{\ind\rat}_{\pro\et})$ is such a group,
then we denote $G_{\tor}^{\PDual} = (G_{\tor})^{\PDual}$.
If $G$ is a lattice over $k$, then its dual lattice is denoted by $G^{\LDual}$,
which can be given by $\sheafhom_{k^{\ind\rat}_{\pro\et}}(G, \Z)$.
If the torsion-free quotient of $G \in \Ab(k^{\ind\rat}_{\pro\et})$ is such a group,
then we denote $G_{/ \tor}^{\LDual} = (G_{/ \tor})^{\LDual}$.

For an abelian ($\mathcal{U}$-)category,
we say that it has exact products
if the product $\prod_{\lambda \in \Lambda} A_{\lambda}$ of
any family of objects $\{A_{\lambda}\}$ (with $\Lambda \in \mathcal{U}$) exists
and for any family of surjections $\{A_{\lambda} \onto B_{\lambda}\}_{\lambda \in \Lambda}$
(with $\Lambda \in \mathcal{U}$),
the morphism $\prod A_{\lambda} \to \prod B_{\lambda}$ is surjective.

\begin{Prop} \label{prop: proetale topos has enough projectives}
	The category $\Ab(k^{\ind\rat}_{\pro\et})$ has enough projectives and exact products.
	For every sequence
		\[
				\cdots
			\overset{\varphi_{3}}{\to}
				A_{3}
			\overset{\varphi_{2}}{\to}
				A_{2}
			\overset{\varphi_{1}}{\to}
				A_{1}
		\]
	in $\Ab(k^{\ind\rat}_{\pro\et})$,
	its derived inverse limit is represented by the complex
	$\prod_{n} A_{n} \to \prod_{n} A_{n}$ in degrees $0$ and $1$,
	where the morphism sends $(a_{n})_{n}$ to $(a_{n} - \varphi_{n}(a_{n + 1}))_{n}$.
	If the sequence is Mittag-Leffler,
	then its derived inverse limit is zero in positive degrees.
\end{Prop}

\begin{proof}
	For any $k' \in k^{\ind\rat}$,
	there exists an ind-\'etale faithfully flat homomorphism $k' \to k''$
	to a w-contractible ring $k''$ \cite[Def.\ 2.4.1]{BS15} by \cite[Lem.\ 2.4.9]{BS15}.
	An ind-\'etale algebra over an ind-rational $k$-algebra is ind-rational over $k$
	by \cite[Prop.\ 2.1.2]{Suz13},
	so $k'' \in k^{\ind\rat}$.
	By the definition of w-contractibility,
	we know that the sheaf of free abelian groups $\Z[\Spec k'']$
	generated by the representable sheaf of sets $\Spec k''$
	is a projective object of $\Ab(k^{\ind\rat}_{\pro\et})$.
	The natural morphism $\Z[\Spec k''] \to \Z[\Spec k']$ is surjective.
	Since any object of $\Ab(k^{\ind\rat}_{\pro\et})$ is
	a quotient of a direct sum (with index set $\in \mathcal{U}$)
	of objects of the form $\Z[\Spec k']$ with $k' \in k^{\ind\rat}$,
	it follows that $\Ab(k^{\ind\rat}_{\pro\et})$ has enough projectives.
	
	The category $\Ab(k^{\ind\rat}_{\pro\et})$ has products.
	Having enough projectives, it has exact products.
	The stated complex calculates the derived limit
	by \cite[Rmk.\ A.3.6]{Nee01} or \cite[Lem.\ 2.2]{Roo06}.
	The statement about Mittag-Leffler sequences follows from
	\cite{Roo06} or \cite[Lem.\ A.3.15]{Nee01}.
\end{proof}

Let $\Pro \Ab(k^{\ind\rat}_{\pro\et})$ be the procategory of $\Ab(k^{\ind\rat}_{\pro\et})$.
(Here we consider the $\mathcal{U}$-procategory
since $\Ab(k^{\ind\rat}_{\pro\et})$ is a $\mathcal{U}$-category.)
A filtered inverse system $\{A_{\lambda}\}_{\lambda}$ in $\Ab(k^{\ind\rat}_{\pro\et})$
as an object of $\Pro \Ab(k^{\ind\rat}_{\pro\et})$
is denoted by $\finvlim[\lambda] A_{\lambda}$.
The category $\Ab(k^{\ind\rat}_{\pro\et})$ is a Grothendieck category
and hence has enough injectives.
Therefore the procategory $\Pro \Ab(k^{\ind\rat}_{\pro\et})$ also has enough injectives
by \cite[Exercise 13.6 (ii)]{KS06}.
As $\Pro \Ab(k^{\ind\rat}_{\pro\et})$ has exact products (and exact filtered inverse limits),
it follows from \cite[Thm.\ 14.4.4]{KS06} that
any additive functor from $\Pro \Ab(k^{\ind\rat}_{\pro\et})$ to an abelian category
admits a right derived functor on (unbounded) derived categories,
which can be calculated by K-injective (or homotopically injective) complexes
in $\Pro \Ab(k^{\ind\rat}_{\pro\et})$.
In particular, the inverse limit functor
$\invlim \colon \Pro \Ab(k^{\ind\rat}_{\pro\et}) \to \Ab(k^{\ind\rat}_{\pro\et})$
admits a right derived functor
$R \invlim \colon D(\Pro \Ab(k^{\ind\rat}_{\pro\et})) \to D(k^{\ind\rat}_{\pro\et})$.
For a more detailed treatment of $R \invlim$,
see \cite[Prop.\ 13.3.15, Cor.\ 13.3.16, Example 13.3.17 (i)]{KS06}%
\footnote{
	For comparison of derived functors of inverse limits
	defined for pro-objects and for inverse systems,
	see \cite[Cor.\ 7.3.7]{Pro99} for example.
	This reference assumes that the abelian category in question has exact products.
	This assumption is satisfied for our category $\Ab(k^{\ind\rat}_{\pro\et})$
	by the proposition above.
}

For a sheaf $A \in \Ab(k^{\ind\rat}_{\pro\et})$,
the multiplication maps by integers $n \ge 1$ yield a pro-object $\finvlim[n] A$
with index set $\Z_{\ge 1}$ ordered by divisibility.
Note that there is a cofinal map $\sigma$ from the set of positive integers with the usual ordering $\le$
to the set of positive integers with the divisibility ordering.
For example, we can take $\sigma(n)$ to be the $n$-th power of the product of the first $n$ primes.
Hence the pro-object $\finvlim[n] A$ is isomorphic to the sequence
$\cdots \to A_{\sigma(2)} \to A_{\sigma(1)}$.
The assignment $A \mapsto \finvlim[n] A$ defines an exact functor
$\Ab(k^{\ind\rat}_{\pro\et}) \to \Pro \Ab(k^{\ind\rat}_{\pro\et})$
and hence a triangulated functor
$D(k^{\ind\rat}_{\pro\et}) \to D(\Pro \Ab(k^{\ind\rat}_{\pro\et}))$.
By composing it with $R \invlim$,
we have a triangulated endofunctor on $D(k^{\ind\rat}_{\pro\et})$,
which we denote by $A \mapsto R \invlim_{n} A$.
The objects $\finvlim[n] A$ and $R \invlim_{n} A$ are uniquely divisible.
We have a projection $\finvlim[n] A \to A$ to the $n = 1$ term in $D(\Pro \Ab(k^{\ind\rat}_{\pro\et}))$.
This induces a morphism $R \invlim_{n} A \to A$ in $D(k^{\ind\rat}_{\pro\et})$.
It is an isomorphism if and only if $A$ is uniquely divisible.
We have $R \invlim_{n} A = 0$ if $A$ is killed by multiplication by some positive integer.

For $A \in \Ab(k^{\ind\rat}_{\pro\et})$,
we define $A_{\divis}$ to be the image of the natural morphism $\invlim_{n} A \to A$.
It is the maximal divisible subsheaf of $A$.
We denote $A_{/ \divis} = A / A_{\divis}$.
We define
	\[
			TA
		=
			\invlim_{n} A[n],
		\quad
			VA
		=
			(TA) \tensor \Q,
		\quad
			\Hat{A}
		=
			\invlim_{n}(A \tensor \Z / n \Z),
	\]
where $A[n]$ is the $n$-torsion part of $A$.
If $A \in \Ind^{f} \Alg_{\uc} / k$,
then \S \ref{sec: Ind-objects of cofinite type} gives an object
$TA  \in \Pro \Alg_{\uc} / k$,
while the above definition gives an object
$TA \in \Ab(k^{\ind\rat}_{\pro\et})$.
They are compatible under the Yoneda functor
$\Pro \Alg_{\uc} / k \into \Ab(k^{\ind\rat}_{\pro\et})$.
A similar relation holds for $VA$.
We denote $TA_{\divis} = T(A_{\divis})$, $VA_{\divis} = V(A_{\divis})$.
If $A_{\divis}$ is a torsion \'etale group,
we denote $A_{\divis}^{\PDual} = (A_{\divis})^{\PDual}$.
The notation $A_{/ \divis}^{\wedge}$ means $(A_{/ \divis})^{\wedge}$.
For $A \in \FGEt / k$, we denote $A_{/ \tor}^{\LDual \wedge} = (A_{/ \tor}^{\LDual})^{\wedge}$.
(The general rule here is that
we read the subscript first and then the superscript,
from the inside to the outside.
Sticking to one such convention reduces excessive usage of parentheses.
But we are not strictly consistent with this rule,
as the ind-rational pro-\'etale site $\Spec k^{\ind\rat}_{\pro\et}$ or
the open subscheme $\mathcal{A}_{0}^{\vee} \subset \mathcal{A}^{\vee}$ of the N\'eron model shows.)

\begin{Prop} \label{prop: derived limit for multiplication}
	For any $A \in \Ab(k^{\ind\rat}_{\pro\et})$,
	we have $R^{m} \invlim_{n} A = 0$ for $m \ge 2$.
	If $A$ is divisibly ML (Prop.\ \ref{prop: divisibly ML objects}),
	then $R \invlim_{n} A = \invlim_{n} A$
	and $A_{/ \divis} \isomto \Hat{A}$.
	If $A_{\divis}$ is torsion, then $VA \isomto \invlim_{n} A$.
\end{Prop}

\begin{proof}
	As above, represent the pro-object $\finvlim[n] A$
	as a sequence $\cdots \to A_{\sigma(2)} \to A_{\sigma(1)}$.
	Then the vanishing of the derived limit in degree $\ge 2$ follows from the previous proposition.
	If $A$ is divisibly ML, then $R \invlim_{n} A = \invlim_{n} A$
	by the previous proposition.
	We then have $\Hat{A} = (A_{/ \divis})^{\wedge} = A_{/ \divis}$.
	For general $A$, we have an exact sequence
	$0 \to TA \to \invlim_{n} A \to A_{\divis} \to 0$.
	Tensoring by $\Q$, we have an exact sequence
	$0 \to VA \to \invlim_{n} A \to A_{\divis} \tensor \Q \to 0$.
	Hence $VA = \invlim_{n} A$ if $A_{\divis} \tensor \Q = 0$.
\end{proof}

\begin{Prop} \label{prop: inverse limit of multiplication for divisibly ML}
	Let $A \in \Ind^{f} \Alg_{\uc} / k$ be divisibly ML.
	Then the morphism $VA \to A$ induces an isomorphism
	$R \invlim_{n} A = V A$ in $D(k^{\ind\rat}_{\pro\et})$.
\end{Prop}

\begin{proof}
	Recall that any object of $\Alg_{\uc} / k$ is killed by multiplication by some positive integer.
	Hence $A_{\divis}$ is torsion as a sheaf over $\Spec k^{\ind\rat}_{\pro\et}$.
	(Recall that this means $A(k')$ is a torsion abelian group for any $k' \in k^{\ind\rat}$.
	See \cite[Exp.\ IX, D\'ef.\ 1.1, Prop.\ 1.2 (iii)]{AGV73}.)
	Now the result follows from the previous proposition.
\end{proof}

We denote $\Adele^{\infty} = \Hat{\Z} \tensor_{\Z} \Q \in \Ind \Pro \FEt / k$
(which is the adele ring of $\Q$ without components of infinite places).
We have $\Adele^{\infty} = V(\Q / \Z)$.
The assumption that $A$ is divisibly ML in the above proposition cannot be dropped.
For instance, applying $R \invlim_{n}$ to the exact sequence
$0 \to \bigoplus_{p} \Z / p \Z \to \Q / \Z \to \Q / \Z \to 0$,
we know that $R \invlim_{n} \bigoplus_{p} \Z / p \Z$ is isomorphic to
$[\Adele^{\infty} \to \Adele^{\infty}][-1]$,
or $[\Hat{\Z} \to \Hat{\Z}][-1] \tensor \Q$, or
$(\prod_{p} \Z / p \Z) / (\bigoplus_{p} \Z / p \Z) [-1]$.

\begin{Prop} \label{prop: inverse limit of multiplication for finite gen etale}
	Let $A \in \FGEt / k$.
	Then $R \invlim_{n} A = A \tensor_{\Z} \Adele^{\infty} / \Q[-1]$
	in $D(k^{\ind\rat}_{\pro\et})$.
\end{Prop}

\begin{proof}
	We may assume that $A$ is torsion-free.
	We have an exact sequence
	$0 \to A \to A \tensor \Q \to A \tensor \Q / \Z \to 0$.
	Applying $R \invlim_{n}$ and using the previous proposition,
	we obtain a distinguished triangle
	$R \invlim_{n} A \to A \tensor \Q \to A \tensor \Adele^{\infty}$.
	From this, the result follows.
\end{proof}

For $A \in D(k^{\ind\rat}_{\pro\et})$,
there is a canonical choice of a mapping cone of the above defined morphism
$\finvlim[n] A \to A$.
We denote this cone by $\finvlim[n](A \tensor^{L} \Z / n \Z)$.
It is an object of $D(\Pro \Ab(k^{\ind\rat}_{\pro\et}))$.
(This construction can also be understood more conceptually
using the derived functor of the procategory extension of the two variable functor $\tensor$,
at least when $A$ is assumed to be bounded above.)

\begin{Prop} \label{prop: cohomology of formal limit of A tensor Z mod n}
	Let $A \in D(k^{\ind\rat}_{\pro\et})$.
	The distinguished triangle
		\[
				\finvlim[n] A
			\to
				A
			\to
				\finvlim[n](A \tensor^{L} \Z / n \Z)
		\]
	induces an exact sequence
		\[
				0
			\to
				\finvlim[n]
					(H^{i}(A) \tensor \Z / n \Z)
			\to
				H^{i}
				\bigl(
					\finvlim[n](A \tensor^{L} \Z / n \Z)
				\bigr)
			\to
				\finvlim[n]
					H^{i + 1}(A)[n]
			\to
				0
		\]
	in $\Pro \Ab(k^{\ind\rat}_{\pro\et})$ for any $i \in \Z$.
\end{Prop}

\begin{proof}
	Obvious.
\end{proof}

\begin{Prop} \label{prop: all mod n isom imply completion isom}
	If a morphism $A \to B$ in $D(k^{\ind\rat}_{\pro\et})$ induces an isomorphism
		\[
			A \tensor^{L} \Z / n \Z \isomto B \tensor^{L} \Z / n \Z
		\]
	for any $n \ge 1$, then it also induces an isomorphism
		\[
				\finvlim[n](A \tensor^{L} \Z / n \Z)
			\isomto
				\finvlim[n](B \tensor^{L} \Z / n \Z).
		\]
\end{Prop}

\begin{proof}
	This follows from the previous proposition.
\end{proof}

If $A \in D(k^{\ind\rat}_{\pro\et})$,
we can apply $R \invlim$ to $\finvlim[n] (A \tensor^{L} \Z / n \Z)$.
We denote the result by
	\[
			R \invlim_{n}(A \tensor^{L} \Z / n \Z)
		\in
			D(k^{\ind\rat}_{\pro\et}).
	\]

\begin{Prop} \label{prop: all mod n isom imply derived completion isom}
	For $A \in D(k^{\ind\rat}_{\pro\et})$,
	we have a distinguished triangle
		\[
				R \invlim_{n} A
			\to
				A
			\to
				R \invlim_{n}(A \tensor^{L} \Z / n \Z).
		\]
	If a morphism $A \to B$ in $D(k^{\ind\rat}_{\pro\et})$ induces an isomorphism
		\[
			A \tensor^{L} \Z / n \Z \isomto B \tensor^{L} \Z / n \Z
		\]
	for any $n \ge 1$, then it also induces an isomorphism
		\[
				R \invlim_{n}(A \tensor^{L} \Z / n \Z)
			\isomto
				R \invlim_{n}(B \tensor^{L} \Z / n \Z).
		\]
\end{Prop}

\begin{proof}
	This follows from the previous proposition.
\end{proof}

\begin{Prop} \label{prop: cohomologies of derived completion}
	Let $A \in D(k^{\ind\rat}_{\pro\et})$.
	Assume that
		\[
				R^{i} \invlim_{n}(H^{j}(A) \tensor \Z / n \Z)
			=
				R^{i} \invlim_{n}(H^{j}(A)[n])
			=
				0
		\]
	for any $i \ge 1$ and $j \in \Z$.
	Then we have an exact sequence
		\[
				0
			\to
				H^{i}(A)^{\wedge}
			\to
				H^{i} R \invlim_{n}(A \tensor^{L} \Z / n \Z)
			\to
				T H^{i + 1}(A)
			\to
				0.
		\]
	for all $i \in \Z$.
\end{Prop}

\begin{proof}
	By Prop.\ \ref{prop: cohomology of formal limit of A tensor Z mod n},
	we know that applying $R^{i} \invlim$ to
	$H^{j} \finvlim[n](A \tensor^{L} \Z / n \Z)$
	results zero for any $i \ge 1$ and $j \in \Z$.
	Hence the $i$-th cohomology of $R \invlim_{n}(A \tensor^{L} \Z / n \Z)$
	is given by applying $\invlim$ to the $i$-th cohomology of $\finvlim[n](A \tensor^{L} \Z / n \Z)$
	for any $i \in \Z$.
	Applying $\invlim$ to the exact sequence in
	Prop.\ \ref{prop: cohomology of formal limit of A tensor Z mod n},
	we get the required exact sequence.
\end{proof}

\begin{Prop} \label{prop: direct inverse adelic rationalization triangle}
	For any $A \in D(k^{\ind\rat}_{\pro\et})$,
	the natural morphism
		\[
				\Bigl(
					R \invlim_{n}
						(A \tensor^{L} \Z / n \Z)
				\Bigr) \tensor \Q
			\to
				R \invlim_{n}(A \tensor^{L} \Q / \Z)
		\]
	is an isomorphism.
	Denote these isomorphic objects by
		$
			A \ctensor \Adele^{\infty}
		$
	We have a natural distinguished triangle
		\[
				R \invlim_{n} A
			\to
				A \tensor \Q
			\to
				A \ctensor \Adele^{\infty}.
		\]
	The assignment $A \mapsto A \ctensor \Adele^{\infty}$ is a triangulated endofunctor
	on $D(k^{\ind\rat}_{\pro\et})$
	that sends $\Z$ to $\Adele^{\infty}$.
\end{Prop}

\begin{proof}
	We have a distinguished triangle
		\[
				R \invlim_{n} A
			\to
				A
			\to
				R \invlim_{n}(A \tensor^{L} \Z / n \Z)
		\]
	and hence a distinguished triangle
		\[
				R \invlim_{n} A
			\to
				A \tensor \Q
			\to
				\Bigl(
					R \invlim_{n}
						(A \tensor^{L} \Z / n \Z)
				\Bigr) \tensor \Q.
		\]
	Also we have a distinguished triangle
		\[
				A
			\to
				A \tensor \Q
			\to
				A \tensor^{L} \Q / \Z
		\]
	and hence a distinguished triangle
		\[
				R \invlim_{n} A
			\to
				A \tensor \Q
			\to
				R \invlim_{n}(A \tensor^{L} \Q / \Z).
		\]
	Comparing the two triangles, we get the result.
\end{proof}

The boundedness conditions in the following two propositions might be unnecessary.

\begin{Prop}
	For $A \in D^{+}(k^{\ind\rat}_{\pro\et})$, there is a canonical isomorphism
		\[
				R \invlim_{n}(A \tensor^{L} \Z / n \Z)
			=
				R \sheafhom_{k}(\Q / \Z, A)[1].
		\]
\end{Prop}

\begin{proof}
	We have
		\[
				R \sheafhom_{k}(\Q, A)
			=
				R \invlim_{n} R \sheafhom_{k}(\Z, A)
			=
				R \invlim_{n} A
		\]
	by the proof of \cite[Prop.\ (2.2.3)]{Suz14}
	(which needs the bounded below condition).
	Comparing the two distinguished triangles
		\begin{gather*}
					R \invlim_{n} A
				\to
					A
				\to
					R \invlim_{n}(A \tensor^{L} \Z / n \Z),
			\\
					R \sheafhom_{k}(\Q, A)
				\to
					R \sheafhom_{k}(\Z, A)
				\to
					R \sheafhom_{k}(\Q / \Z, A)[1],
		\end{gather*}
	we get the result.
\end{proof}

\begin{Prop} \label{prop: derived completion is RHom from Q mod Z}
	For $A \in D^{-}(k^{\ind\rat}_{\pro\et})$
	and $B \in D^{+}(k^{\ind\rat}_{\pro\et})$,
	there is a canonical isomorphism
		\[
				R \invlim_{n} \bigl(
					R \sheafhom_{k}(A, B) \tensor^{L} \Z / n \Z
				\bigr)
			=
				R \sheafhom_{k}(A \tensor^{L} \Q / \Z, B)[1].
		\]
\end{Prop}

\begin{proof}
	We know that $R \sheafhom_{k}(A, B)$ is bounded below.
	By the previous proposition, the left-hand side can be written as
		\[
			R \sheafhom_{k} \bigl(
				\Q / \Z,
				R \sheafhom_{k}(A, B)
			\bigr)[1].
		\]
	This is also the right-hand side by the derived tensor-Hom adjunction.
\end{proof}


\subsection{Premorphisms of sites}
\label{sec: Premorphisms of sites}

For a site $S$, we denote its category of sheaves of sets (resp.\ abelian groups)
by $\Set(S)$ (resp.\ $\Ab(S)$).
The category of complexes in $\Ab(S)$ is $\Ch(S)$,
its homotopy category $K(S)$ and derived category $D(S)$.
We say $A \in \Ch(S)$ is \emph{K-limp}
if $R \Gamma(X, A) = \Gamma(X, A)$ for any object $X \in S$,
where $\Gamma$ in the right-hand side is applied term-wise.
See also \cite[Cor.\ 5.17]{Spa88} and \cite[Appendix A]{Sch17}.
This is not a ``K''-version of limp sheaf as defined in \cite[Tag 072Y]{Sta18}
since $X$ is not allowed to be an arbitrary sheaf of sets but only a representable one.
K-injectives are K-limp.

The following three propositions are essentially well-known,
at least for bounded below complexes of sheaves or just sheaves;
see \cite[III, \S 1, \S 2]{Mil80}.

\begin{Prop} \label{prop: sheaf-Hom to K-injective is K-limp}
	Let $S$ be a site and $A, J \in \Ch(S)$ complexes such that $J$ is K-injective.
	Then the total complex of the sheaf-Hom double complex
	$\sheafhom_{S}(A, J)$ is K-limp.
\end{Prop}

\begin{proof}
	Note that
	$\sheafhom_{S}(A, J) = R \sheafhom_{S}(A, J)$.
	By \cite[Prop.\ 18.6.6]{KS06}, we have
	$R \Gamma(X, R \sheafhom_{S}(A, J)) = R \Hom_{S / X}(A, J)$
	for any $X \in S$,
	where $S / X$ is the localization of $S$ at $X$.
	Hence
		\begin{align*}
			&
					R \Gamma(X, \sheafhom_{S}(A, J))
				=
					R \Hom_{S / X}(A, J)
			\\
			&	=
					\Hom_{S / X}(A, J)
				=
					\Gamma(X, \sheafhom_{S}(A, J)).
		\end{align*}
\end{proof}

Let $f^{-1}$ be a functor from the underlying category of a site $S$
to the underlying category of another site $S'$.
The right composition with $f^{-1}$ defines a functor
from the category of presheaves of sets on $S'$ to
the category of presheaves of sets on $S$,
which is the pushforward functor for presheaves.
If this functor sends sheaves to sheaves,
$f^{-1}$ is called a continuous functor in the terminology of \cite[Exp.\ III, Def.\ 1.1]{AGV72a}.
We say that $f^{-1}$ then defines a continuous map of sites $f \colon S' \to S$.
The pushforward functor $\Set(S') \to \Set(S)$ or $\Ab(S') \to \Ab(S)$ is denoted by $f_{\ast}$.
The left adjoint to $f_{\ast} \colon \Set(S') \to \Set(S)$ exists by \cite[Exp.\ III, Prop.\ 1.2]{AGV72a},
which we denote by $f^{\ast\set} \colon \Set(S) \to \Set(S')$.
The left adjoint to $f_{\ast} \colon \Ab(S') \to \Ab(S)$ exists by \cite[Exp.\ III, Prop.\ 1.7]{AGV72a},
which we denote by $f^{\ast} \colon \Set(S) \to \Set(S')$.
They are called pullback functors.
The continuous map $f \colon S' \to S$ is called a morphism of sites
if the pullback $f^{\ast\set} \colon \Set(S) \to \Set(S')$ for sheaves of sets is exact.
In this case, the pullback for sheaves of abelian groups $f^{\ast} \colon \Ab(S) \to \Ab(S')$ and
$f^{\ast\set}$ are compatible with forgetting group structures
by \cite[III, Prop.\ 1.7, 4)]{AGV72a},
so we do not have to distinguish $f^{\ast}$ and $f^{\ast\set}$.
In general, we use $f^{\ast}$ to mean the pullback functor for sheaves of abelian groups.

\begin{Prop} \label{prop: K-limp is push injective}
	Let $f \colon S' \to S$ be a continuous map of sites.
	For any K-limp complex $A'$ in $\Ab(S')$,
	we have $R f_{\ast} A' = f_{\ast} A'$ in $D(S)$.
\end{Prop}

\begin{proof}
	Let $A' \isomto I'$ be a K-injective replacement.
	Let $B'$ be the mapping cone of $A' \to I'$,
	which is an exact complex and hence $R f_{\ast} B' = 0$ in $D(S)$.
	For any $X' \in S'$, consider the distinguished triangle
	$\Gamma(X', A') \to \Gamma(X', I') \to \Gamma(X', B')$.
	In $D(\Ab)$, we have $\Gamma(X', A') = R \Gamma(X', A')$ by assumption
	and $\Gamma(X', I') = R \Gamma(X', I')$ by the K-injectivity of $I'$.
	Hence $\Gamma(X', B') = R \Gamma(X', B')$,
	which is zero in $D(\Ab)$ since $B'$ is exact.
	This means that $B'$ is also exact as a complex of presheaves.
	Hence $f_{\ast} B'$ is exact.
	Consider the distinguished triangle
	$f_{\ast} A' \to f_{\ast} I' \to f_{\ast} B'$.
	We have $f_{\ast} I' = R f_{\ast} I'$ and $f_{\ast} B' = R f_{\ast} B' = 0$,
	so $f_{\ast} A' = R f_{\ast} A'$.
\end{proof}

If $S$ and $S'$ are sites defined by pretopologies,
and if $f^{-1}$ is a functor from the underlying category of $S$ to that of $S'$
that sends coverings to coverings and
$f^{-1}(Y \times_{X} Z) = f^{-1}(Y) \times_{f^{-1}(X)} f^{-1}(Z)$
whenever $Y \to X$ appears in a covering family,
then $f^{-1}$ is called a morphism of topologies in the terminology of \cite[Def.\ 2.4.2]{Art62}
and defines a continuous map $f \colon S' \to S$.
We call such a continuous map a \emph{premorphism of sites}.
In this case, $f_{\ast}$ sends acyclic sheaves
(i.e.\ those $A'$ with $H^{n}(X', A') = 0$ for any object $X'$ of $S'$ and $n \ge 1$) to acyclic sheaves
and hence induces the Leray spectral sequence
$R \Gamma(X, R f_{\ast} A') = R \Gamma(f^{-1}(X), A')$
for any $X \in S$ and $A' \in \Ab(S')$ \cite[\S II.4]{Art62}.
Here is a slight generalization.

\begin{Prop} \label{prop: push sends K limp to K limp}
	Let $f \colon S' \to S$ be a premorphism of sites defined by pretopologies.
	Then $f_{\ast}$ sends K-limp complexes to K-limp complexes.
	For any $X \in S$ and $A' \in D(S')$, we have an isomorphism
		\[
				R \Gamma(X, R f_{\ast} A')
			=
				R \Gamma(f^{-1} X, A'),
		\]
	where $f^{-1} \colon S \to S'$ is the underlying functor of $f$.
\end{Prop}

\begin{proof}
	By \cite[Lem.\ 3.7.2]{Suz13}, we know that
	$f^{\ast} \colon \Ab(S) \to \Ab(S')$ admits a left derived functor
	$L f^{\ast} \colon D(S) \to D(S')$,
	which is left adjoint to $R f_{\ast} \colon D(S') \to D(S)$ and
	sends the sheaf $\Z[X]$ of free abelian groups generated by any $X \in S$ to $\Z[f^{-1} X]$.
	For any $A' \in \Ch(S')$ and $X \in S$, we have
		\begin{align*}
				R \Gamma(X, R f_{\ast} A')
		&	=
				R \Hom_{S}(\Z[X], R f_{\ast} A')
			=
				R \Hom_{S'}(L f^{\ast} \Z[X], A')
		\\
		&	=
				R \Hom_{S'}(\Z[f^{-1} X], A')
			=
				R \Gamma(f^{-1} X, A').
		\end{align*}
	If $A'$ is K-limp, then $R f_{\ast} A' = f_{\ast} A'$ by Prop.\ \ref{prop: K-limp is push injective}
	and $R \Gamma(f^{-1} X, A') = \Gamma(f^{-1} X, A')$.
	Hence $R \Gamma(X, f_{\ast} A') = \Gamma(f^{-1} X, A') = \Gamma(X, f_{\ast} A')$,
	and so $f_{\ast} A'$ is K-limp.
\end{proof}


\subsection{Local duality without relative sites}
\label{sec: Local duality without relative sites}

Let $\Hat{K}_{x}$ be a complete discrete valuation field
with perfect residue field $k_{x}$ ($\in \mathcal{U}_{0}$) of characteristic $p > 0$.
The ring of integers is denoted by $\Hat{\Order}_{x}$ with maximal ideal $\Hat{\ideal{p}}_{x}$.
By the fppf site $\Spec \Hat{\Order}_{x, \fppf}$ of $\Hat{\Order}_{x}$,
we mean the category of ($\mathcal{U}_{0}$-small) $\Hat{\Order}_{x}$-algebras
endowed with the fppf topology.
The same applies to the fppf site $\Spec \Hat{K}_{x, \fppf}$ of $\Hat{K}_{x}$.
Sheaves take values in $\Set$ or $\Ab$ (of $\mathcal{U}$-small sets or abelian groups).
The Hom functor and the sheaf-Hom functor for $\Spec \Hat{\Order}_{x, \fppf}$ is denoted by
$\Hom_{\Hat{\Order}_{x}}$ and $\sheafhom_{\Hat{\Order}_{x}}$, respectively.
Their derived functors are denoted by
$\Ext_{\Hat{\Order}_{x}}^{n}$, $\sheafext_{\Hat{\Order}_{x}}^{n}$,
$R \Hom_{\Hat{\Order}_{x}}$, $R \sheafhom_{\Hat{\Order}_{x}}$.
Similar notation applies to $\Spec \Hat{K}_{x, \fppf}$.
We denote $k_{x}^{\ind\rat} = (k_{x})^{\ind\rat}$,
$\Spec k^{\ind\rat}_{x, \et} = \Spec (k_{x})^{\ind\rat}_{\et}$ and
$\Spec k^{\ind\rat}_{x, \pro\et} = \Spec (k_{x})^{\ind\rat}_{\pro\et}$.

The ring $\Hat{\Order}_{x}$ has a natural structure of a $W(k_{x})$-algebra of pro-finite-length,
where $W$ denotes the ring of $p$-typical Witt vectors of infinite length.
In equal characteristic, this structure factors through $k_{x}$
so that $\Hat{\Order}_{x}$ is pro-finite-length over $k_{x}$.
In mixed characteristic, it is finite free over $W(k_{x})$.
For $k'_{x} \in k^{\ind\rat}_{x}$, we define
	\begin{gather*}
				\Hat{\alg{O}}_{x}(k'_{x})
			=
				W(k'_{x}) \Hat{\tensor}_{W(k_{x})} \Hat{\Order}_{x}
			=
				\invlim_{n} \bigl(
					W_{n}(k'_{x}) \tensor_{W_{n}(k_{x})} \Hat{\Order}_{x} / \Hat{\ideal{p}}_{x}^{n}
				\bigr),
		\\
				\Hat{\alg{K}}_{x}(k'_{x})
			=
					\Hat{\alg{O}}_{x}(k'_{x})
				\tensor_{\Hat{\Order}_{x}}
					\Hat{K}_{x}.
	\end{gather*}

\begin{Prop}
	The functors $k'_{x} \mapsto \Hat{\alg{O}}_{x}(k'_{x}), \Hat{\alg{K}}_{x}(k'_{x})$ define premorphisms of sites
		\begin{gather*}
					\pi_{\Hat{\Order}_{x}}
				\colon
					\Spec \Hat{\Order}_{x, \fppf}
				\to
					\Spec k^{\ind\rat}_{x, \et},
			\\
					\pi_{\Hat{K}_{x}}
				\colon
					\Spec \Hat{K}_{x, \fppf}
				\to
					\Spec k^{\ind\rat}_{x, \et}.
		\end{gather*}
\end{Prop}

\begin{proof}
	From the proof of \cite[Prop.\ 2.3.1]{Suz13}, we know that
	if $k'_{x} \to k''_{x}$ is a (faithfully flat) \'etale homomorphism in $k_{x}^{\ind\rat}$,
	then $\Hat{\alg{O}}_{x}(k'_{x}) \to \Hat{\alg{O}}_{x}(k''_{x})$ is (faithfully flat) \'etale,
	and if $k'''_{x} \in k_{x}^{\ind\rat}$ is another object,
	then $\Hat{\alg{O}}_{x}(k''_{x}) \tensor_{\Hat{\alg{O}}_{x}(k'_{x})} \Hat{\alg{O}}_{x}(k'''_{x})$
	is isomorphic to $\Hat{\alg{O}}_{x}(k''_{x} \tensor_{k'_{x}} k'''_{x})$.
	Hence $\pi_{\Hat{\Order}_{x}}$ is a premorphism of sites.
	So is $\pi_{\Hat{K}_{x}}$.
\end{proof}

The scheme morphism $j \colon \Spec \Hat{K}_{x} \into \Spec \Hat{\Order}_{x}$
defines a morphism of sites
	\[
			j
		\colon
			\Spec \Hat{K}_{x, \fppf}
		\to
			\Spec \Hat{\Order}_{x, \fppf}.
	\]
Its pullback functor $j^{\ast}$
is the restriction functor for sheaves from $\Hat{\Order}_{x}$ to $\Hat{K}_{x}$.
This restriction functor is frequently omitted by abuse of notation,
so a sheaf (or a complex of sheaves) $F$ on $\Spec \Hat{\Order}_{x, \fppf}$
restricted to $\Spec \Hat{K}_{x, \fppf}$
will frequently be written by just $F$.
More generally, if $f \colon Z \to Y$ is a morphism of schemes,
then the restriction $f^{\ast} F$ of a sheaf (or a complex of sheaves) $F$ on $Y_{\fppf}$
will frequently be written by just $F$.
But the important point is that $f$ is a localization morphism
\cite[III, \S 5]{AGV72a}.
Recall from loc.\ cit.\ that for a site $S$ and its object $V$,
the category $S / V$ of pairs $(W, g)$ ($W$ an object of $S$ and $g \colon W \to V$ a morphism in $S$)
is equipped with the natural induced topology.
The functor $j_{V} \colon (W, g) \mapsto W$ defines a continuous map of sites $S \to S / V$
whose pushforward functor $\Set(S) \to \Set(S / V)$ is the restriction functor.
In our situation, by $f$ being a localization morphism, we mean that
$f \colon Z \to Y$ belongs to the underlying category of $Y_{\fppf}$
and $Z_{\fppf}$ can be identified with $Y_{\fppf} / Z$.
Hence $f^{\ast}$ admits an exact left adjoint $f_{!} \colon \Ab(Z_{\fppf}) \to \Ab(Y_{\fppf})$
by \cite[IV, Prop.\ 11.3.1]{AGV72a}.
Hence $f^{\ast}$ send (K-)injectives to (K-)injectives.
This will be needed to pass from sheaf or complex level statements to derived categorical statements.

We have the pushforward functors
	\begin{gather*}
				(\pi_{\Hat{\Order}_{x}})_{\ast}
			\colon
				\Ab(\Hat{\Order}_{x, \fppf})
			\to
				\Ab(k^{\ind\rat}_{x, \et}),
		\\
				(\pi_{\Hat{K}_{x}})_{\ast}
			\colon
				\Ab(\Hat{K}_{x, \fppf})
			\to
				\Ab(k^{\ind\rat}_{x, \et})
	\end{gather*}
by $\pi_{\Hat{\Order}_{x}}$ and $\pi_{\Hat{K}_{x}}$.
Let $\Ab(k^{\ind\rat}_{x, \et}) \to \Ab(k^{\ind\rat}_{x, \pro\et})$
be the pro-\'etale sheafification functor.
The composite functors of these functors are denoted by
	\begin{gather*}
				\alg{\Gamma}(\Hat{\Order}_{x}, \var)
			\colon
				\Ab(\Hat{\Order}_{x, \fppf})
			\to
				\Ab(k^{\ind\rat}_{x, \pro\et}),
		\\
				\alg{\Gamma}(\Hat{K}_{x}, \var)
			\colon
				\Ab(\Hat{K}_{x, \fppf})
			\to
				\Ab(k^{\ind\rat}_{x, \pro\et}).
	\end{gather*}
We have the right derived functors
	\begin{gather*}
				R \alg{\Gamma}(\Hat{\Order}_{x}, \var)
			\colon
				D(\Hat{\Order}_{x, \fppf})
			\to
				D(k^{\ind\rat}_{x, \pro\et}),
		\\
				R \alg{\Gamma}(\Hat{K}_{x}, \var)
			\colon
				D(\Hat{K}_{x, \fppf})
			\to
				D(k^{\ind\rat}_{x, \pro\et}).
	\end{gather*}
Since sheafification is exact,
$R \alg{\Gamma}(\Hat{\Order}_{x}, \var)$ is the composite of
$R (\pi_{\Hat{\Order}_{x}})_{\ast}$
and the pro-\'etale sheafification functor
$D(k^{\ind\rat}_{x, \et}) \to D(k^{\ind\rat}_{x, \pro\et})$.
The same is true for $R \alg{\Gamma}(\Hat{K}_{x}, \var)$.
We denote $\alg{H}^{n}(\Hat{\Order}_{x}, \var) = H^{n} R \alg{\Gamma}(\Hat{\Order}_{x}, \var)$
and similarly $\alg{H}^{n}(\Hat{K}_{x}, \var)$.
By Prop.\ \ref{prop: comparison with relative site formalism},
these functors are compatible with the functors
$R \Tilde{\alg{\Gamma}}(\Hat{\Order}_{x}, \var)$ and
$R \Tilde{\alg{\Gamma}}(\Hat{K}_{x}, \var)$
in the notation of \cite[paragraph before Prop.\ (3.3.8)]{Suz14}.

The pro-\'etale sheafification in general makes it hard to calculate the derived global section
$R \Gamma(k'_{x, \pro\et}, \var)$ of the objects
$R \alg{\Gamma}(\Hat{\Order}_{x}, G)$ and $R \alg{\Gamma}(\Hat{K}_{x}, G)$
at each $k'_{x} \in k_{x}^{\ind\rat}$ for a general complex $G$.
In general, it is not clear whether the natural morphisms
	\begin{gather*}
				R \Gamma(\Hat{\Order}_{x}, G)
			\to
				R \Gamma \bigl(
					k_{x, \pro\et},
					R \alg{\Gamma}(\Hat{\Order}_{x}, G)
				\bigr)
		\\
				R \Gamma(\Hat{K}_{x}, G)
			\to
				R \Gamma \bigl(
					k_{x, \pro\et},
					R \alg{\Gamma}(\Hat{K}_{x}, G)
				\bigr)
	\end{gather*}
are isomorphisms or not.
The first (resp.\ second) morphism is an isomorphisms
if $R \alg{\Gamma}(\Hat{\Order}_{x}, G)$ (resp.\ $R \alg{\Gamma}(\Hat{K}_{x}, G)$) is P-acyclic
in the sense of \cite[\S 2.4]{Suz14}.
This condition is satisfied for $G$ a smooth group scheme or a finite flat group scheme
by \cite[Prop.\ (3.4.2), (3.4.3)]{Suz14}.
This is why the notion of P-acyclicity is introduced in \cite[\S 2.4]{Suz14}.
The following proposition shows
we can eliminate P-acyclicity from \cite{Suz14}
to a certain extent.

\begin{Prop}
	Let $G \in D(\Hat{\Order}_{x, \fppf})$.
	For any w-contractible $k'_{x} \in k_{x}^{\ind\rat}$
	(\cite[Def. 2.4.1]{BS15}, which includes any algebraically closed field over $k_{x}$),
	we have
		\[
				R \Gamma \bigl(
					k'_{x, \pro\et},
					R \alg{\Gamma}(\Hat{\Order}_{x}, G)
				\bigr)
			=
				R \Gamma(\Hat{\alg{O}}_{x}(k'_{x}), G).
		\]
	A similar relation holds with $\Hat{\Order}_{x}$ replaced by $\Hat{K}_{x}$.
\end{Prop}

\begin{proof}
	We have
		\[
				R \Gamma \bigl(
					k'_{x, \et},
					R (\pi_{\Hat{\Order}_{x}})_{\ast} G
				\bigr)
			=
				R \Gamma(\Hat{\alg{O}}_{x}(k'_{x}), G)
		\]
	by Prop.\ \ref{prop: push sends K limp to K limp}.
	Hence it is enough to show that
	$R \Gamma(k'_{x, \pro\et}, \Tilde{F}) = R \Gamma(k'_{x, \et}, F)$
	for any $F \in D(k^{\ind\rat}_{\et})$,
	where $\Tilde{F}$ is the pro-\'etale sheafification of $F$.
	The section functor $\Gamma(k'_{x}, \var)$ is exact
	on $\Ab(k^{\ind\rat}_{x, \pro\et})$ and $\Ab(k^{\ind\rat}_{x, \et})$
	since $k'_{x}$ is w-contractible.
	Hence it is enough to show that
	$\Gamma(k'_{x}, \Tilde{F}) = \Gamma(k'_{x}, F)$
	for any $F \in \Ab(k^{\ind\rat}_{x, \et})$.
	Sheafification is given by applying
	the zeroth \v{C}ech cohomology presheaf functor $\underline{\Check{H}}^{0}$ twice
	(\cite[III, Rmk.\ 2.2 (c)]{Mil80}).
	Any pro-\'etale covering of a w-contractible scheme has a section by definition
	and the \v{C}ech complex for a covering with section is null-homotopic.
	Hence for any presheaf $F$ on $\Spec k'_{x, \pro\et}$, we have
		\[
				\Gamma(k'_{x}, \underline{\Check{H}}^{0} \underline{\Check{H}}^{0}(F))
			=
				\Gamma(k'_{x}, \underline{\Check{H}}^{0}(F))
			=
				\Gamma(k'_{x}, F).
		\]
\end{proof}

We define a functor by the mapping fiber construction
	\[
			\alg{\Gamma}_{x}(\Hat{\Order}_{x}, \var)
		:=
			\bigl[
					\alg{\Gamma}(\Hat{\Order}_{x}, \var)
				\to
					\alg{\Gamma}(\Hat{K}_{x}, j^{\ast} \var)
			\bigr][-1]
		\colon
			\Ch(\Hat{\Order}_{x, \fppf})
		\to
			\Ch(k^{\ind\rat}_{x, \pro\et})
	\]
on the category of complexes of sheaves.
This is an additive functor that commutes with the translation functors,
i.e.\ a functor of additive categories with translation
in the terminology of \cite[Def.\ 10.1.1]{KS06}.
It induces a functor on the homotopy categories by \cite[Prop.\ 11.2.9]{KS06}.
We have its right derived functor
	\[
			R \alg{\Gamma}_{x}(\Hat{\Order}_{x}, \var)
		=
			\bigl[
					R \alg{\Gamma}(\Hat{\Order}_{x}, \var)
				\to
					R \alg{\Gamma}(\Hat{K}_{x}, j^{\ast} \var)
			\bigr][-1]
		\colon
			D(\Hat{\Order}_{x, \fppf})
		\to
			D(k^{\ind\rat}_{x, \pro\et})
	\]
by \cite[Thm.\ 14.3.1 (vi)]{KS06}.
We set $\alg{H}_{x}^{n}(\Hat{\Order}_{x}, \var) = H^{n} R \alg{\Gamma}_{x}(\Hat{\Order}_{x}, \var)$.
Note that $\alg{H}_{x}^{0} \ne \alg{\Gamma}_{x}$ in this definition.
(It can be shown that they define the same right derived functors
by the same method as \cite[Prop.\ (3.3.3)]{Suz14}.)
By Prop.\ \ref{prop: comparison with relative site support cohmology},
this is compatible with the functor denoted by
$R \Tilde{\alg{\Gamma}}_{x}(\Hat{\Order}_{x}, \var)$
in \cite[paragraph before Prop.\ (3.3.8)]{Suz14}.
We frequently omit the restriction functor $j^{\ast}$ from this type of formulas
by abuse of notation.

Let $A, B \in \Ch(\Hat{\Order}_{x, \fppf})$.
For any $k'_{x} \in k^{\ind\rat}_{x}$,
the functoriality of $\alg{\Gamma}_{x}(\Hat{\Order}_{x}, \var)$ gives a morphism
	\[
			\Hom_{\Hat{\alg{O}}_{x}(k'_{x})}(A, B)
		\to
			\Hom_{k^{\ind\rat}_{x, \pro\et} / k'_{x}} \bigl(
				\alg{\Gamma}_{x}(\Hat{\Order}_{x}, A),
				\alg{\Gamma}_{x}(\Hat{\Order}_{x}, B)
			\bigr)
	\]
in $\Ch(\Ab)$,
where $\Hom_{k^{\ind\rat}_{x, \pro\et} / k'_{x}}$ is the Hom functor
for the category of sheaves on the localization of $\Spec k^{\ind\rat}_{x, \pro\et}$ at $k'_{x}$.
This is functorial on  $k'_{x}$, so we have a morphism
	\begin{equation} \label{eq: functoriality morphism over integers}
			\alg{\Gamma} \bigl(
				\Hat{\Order}_{x},
				\sheafhom_{\Hat{\Order}_{x}}(A, B)
			\bigr)
		\to
			\sheafhom_{k^{\ind\rat}_{x, \pro\et}} \bigl(
				\alg{\Gamma}_{x}(\Hat{\Order}_{x}, A),
				\alg{\Gamma}_{x}(\Hat{\Order}_{x}, B)
			\bigr)
	\end{equation}
in $\Ch(k^{\ind\rat}_{x, \pro\et})$.
The right derived functor of the left-hand side (as a functor on $A, B$) is
	\[
		R \alg{\Gamma} \bigl(
			\Hat{\Order}_{x},
			R \sheafhom_{\Hat{\Order}_{x}}(A, B)
		\bigr)
	\]
by Prop.\ \ref{prop: sheaf-Hom to K-injective is K-limp},
\ref{prop: K-limp is push injective}
and the theorem on derived functors of composition \cite[Prop.\ 10.3.5 (ii)]{KS06}.
Hence we have a morphism
	\begin{equation} \label{eq: derived functoriality morphism over integers}
			R \alg{\Gamma} \bigl(
				\Hat{\Order}_{x}, R \sheafhom_{\Hat{\Order}_{x}}(A, B)
			\bigr)
		\to
			R \sheafhom_{k^{\ind\rat}_{x, \pro\et}} \bigl(
				R \alg{\Gamma}_{x}(\Hat{\Order}_{x}, A),
				R \alg{\Gamma}_{x}(\Hat{\Order}_{x}, B)
			\bigr)
	\end{equation}
in $D(k^{\ind\rat}_{x, \pro\et})$,
functorial on $A, B \in D(\Hat{\Order}_{x, \fppf})$,
by the universal property of the right derived functor.
(Note that the right-hand side of this is not the right derived functor of
the right-hand side of \eqref{eq: functoriality morphism over integers}.
The problem is that $\pi_{\Hat{\Order}_{x}}$ is only a premorphism of sites
and hence its pushforward functor might not send (K-)injectives to (K-)injectives.
Sheafification also might not send (K-)injectives to (K-)injectives.)
Similarly, we have a morphism
	\begin{equation} \label{eq: functoriality morphism over integers, variant}
			\alg{\Gamma}_{x} \bigl(
				\Hat{\Order}_{x},
				\sheafhom_{\Hat{\Order}_{x}}(A, B)
			\bigr)
		\to
			\sheafhom_{k^{\ind\rat}_{x, \pro\et}} \bigl(
				\alg{\Gamma}(\Hat{\Order}_{x}, A),
				\alg{\Gamma}_{x}(\Hat{\Order}_{x}, B)
			\bigr)
	\end{equation}
in $\Ch(k^{\ind\rat}_{x, \pro\et})$, functorial on
$A, B \in \Ch(\Hat{\Order}_{x, \fppf})$.
Deriving, we have a morphism
	\begin{equation} \label{eq: derived functoriality morphism over integers, variant}
			R \alg{\Gamma}_{x} \bigl(
				\Hat{\Order}_{x}, R \sheafhom_{\Hat{\Order}_{x}}(A, B)
			\bigr)
		\to
			R \sheafhom_{k^{\ind\rat}_{x, \pro\et}} \bigl(
				R \alg{\Gamma}(\Hat{\Order}_{x}, A),
				R \alg{\Gamma}_{x}(\Hat{\Order}_{x}, B)
			\bigr)
	\end{equation}
in $D(k^{\ind\rat}_{x, \pro\et})$,
functorial on $A, B \in D(\Hat{\Order}_{x, \fppf})$.
Also we have a morphism
	\begin{equation} \label{eq: functoriality morphism over local field}
			\alg{\Gamma} \bigl(
				\Hat{K}_{x},
				\sheafhom_{\Hat{K}_{x}}(A, B)
			\bigr)
		\to
			\sheafhom_{k^{\ind\rat}_{x, \pro\et}} \bigl(
				\alg{\Gamma}(\Hat{K}_{x}, A),
				\alg{\Gamma}(\Hat{K}_{x}, B)
			\bigr)
	\end{equation}
in $\Ch(k^{\ind\rat}_{x, \pro\et})$, functorial on
$A, B \in \Ch(\Hat{K}_{x, \fppf})$.
Deriving, we have a morphism
	\begin{equation} \label{eq: derived functoriality morphism over local field}
			R \alg{\Gamma} \bigl(
				\Hat{K}_{x}, R \sheafhom_{\Hat{K}_{x}}(A, B)
			\bigr)
		\to
			R \sheafhom_{k^{\ind\rat}_{x, \pro\et}} \bigl(
				R \alg{\Gamma}(\Hat{K}_{x}, A),
				R \alg{\Gamma}(\Hat{K}_{x}, B)
			\bigr)
	\end{equation}
in $D(k^{\ind\rat}_{x, \pro\et})$,
functorial on $A, B \in D(\Hat{K}_{x, \fppf})$.
These morphisms are compatible with the morphisms in \cite[Prop.\ (3.3.8)]{Suz14}
by Prop.\ \ref{prop: comparison with relative site functoriality morphisms}.

With the comparison results in
Appendix \ref{sec: Comparison of local duality with and without relative sites},
we can translate the results of \cite{Suz14} to our setting.
In particular, by \cite[Prop.\ (3.4.2) (a), (3.4.3) (e); \S 4.1]{Suz14},
any term in the localization triangle
	\[
			R \alg{\Gamma}(\Hat{\Order}_{x}, \Gm)
		\to
			R \alg{\Gamma}(\Hat{K}_{x}, \Gm)
		\to
			R \alg{\Gamma}_{x}(\Hat{\Order}_{x}, \Gm)[1]
	\]
is concentrated in degree zero, where we have an exact sequence
	\[
			0
		\to
			\Hat{\alg{O}}_{x}^{\times}
		\to
			\Hat{\alg{K}}_{x}^{\times}
		\to
			\Z
		\to
			0.
	\]
We call the (iso)morphisms
	\begin{equation} \label{eq: local trace morphism}
			R \alg{\Gamma}(\Hat{K}_{x}, \Gm)
		\to
			R \alg{\Gamma}_{x}(\Hat{\Order}_{x}, \Gm)[1]
		=
			\Z
	\end{equation}
the \emph{trace (iso)morphisms} (at $x$).

Let $A$ be an abelian variety over $\Hat{K}_{x}$ with dual $A^{\vee}$.
By the Barsotti-Weil formula \cite[Chap.\ III, Thm.\ (18.1)]{Oor66},
we have a canonical isomorphism
$A^{\vee} \cong \sheafext_{\Hat{K}_{x}}^{1}(A, \Gm)$ defined by the Poincar\'e bundle.
With $\sheafhom_{\Hat{K}_{x}}(A, \Gm) = 0$,
we have a canonical morphism
$A^{\vee} \to R \sheafhom_{\Hat{K}_{x}}(A, \Gm)[1]$ in $D(\Hat{K}_{x, \fppf})$.
Hence we have a morphism
	\begin{align*}
				R \alg{\Gamma}(\Hat{K}_{x}, A^{\vee})
		&	\to
				R \alg{\Gamma} \bigl(
					\Hat{K}_{x},
					R \sheafhom_{\Hat{K}_{x}}(A, \Gm)
				\bigr)[1]
		\\
		&	\to
				R \sheafhom_{k^{\ind\rat}_{x, \pro\et}}(
					R \alg{\Gamma}(\Hat{K}_{x}, A), R \alg{\Gamma}(\Hat{K}_{x}, \Gm)
				)[1]
		\\
		&	\to
				R \sheafhom_{k^{\ind\rat}_{x, \pro\et}}(
					R \alg{\Gamma}(\Hat{K}_{x}, A), \Z
				)[1]
			=
				R \alg{\Gamma}(\Hat{K}_{x}, A)^{\SDual_{x}}[1]
	\end{align*}
by \eqref{eq: derived functoriality morphism over local field}
and \eqref{eq: local trace morphism},
where $\SDual_{x} = R \sheafhom_{k^{\ind\rat}_{x, \pro\et}}(\var, \Z)$.
Its Serre dual, when $A$ is replaced by $A^{\vee}$, is denoted by
	\[
			\theta_{A}
		\colon
			R \alg{\Gamma}(\Hat{K}_{x}, A^{\vee})^{\SDual_{x} \SDual_{x}}
		\to
			R \alg{\Gamma}(\Hat{K}_{x}, A)^{\SDual_{x}}[1].
	\]
These morphisms agree with the morphisms defined in \cite[\S 4.1]{Suz14}.
Hence \cite[Thm.\ (4.1.2)]{Suz14} implies that $\theta_{A}$ is an isomorphism.

\begin{Prop} \label{prop: structure of local cohom of abelian varieties}
	Let $A$ be an abelian variety over $\Hat{K}_{x}$ with N\'eron model $\mathcal{A}$.
	Let $\mathcal{A}_{0}$ be the maximal open subgroup scheme of $\mathcal{A}$
	with connected fibers
	and $\mathcal{A}_{x}$ the special fiber of $\mathcal{A}$ over $x = \Spec k_{x}$.
	Then $R \alg{\Gamma}(\Hat{\Order}_{x}, \mathcal{A})$ is concentrated in degree $0$;
	$R \alg{\Gamma}(\Hat{\Order}_{x}, \mathcal{A}_{0})$ in degree $0$;
	$R \alg{\Gamma}_{x}(\Hat{\Order}_{x}, \mathcal{A})$ in degree $2$;
	and $R \alg{\Gamma}_{x}(\Hat{\Order}_{x}, \mathcal{A}_{0})$ in degrees $1, 2$.
	We have
		\begin{gather*}
					\alg{\Gamma}(\Hat{\Order}_{x}, \mathcal{A})
				=
					\alg{\Gamma}(\Hat{K}_{x}, A)
				\in
					\Pro \Alg / k_{x},
				\quad
					\alg{\Gamma}(\Hat{\Order}_{x}, \mathcal{A}_{0})
				=
					\alg{\Gamma}(\Hat{K}_{x}, A)_{0},
			\\
					\alg{H}_{x}^{1}(\Hat{\Order}_{x}, \mathcal{A}_{0})
				=
					\pi_{0}(\mathcal{A}_{x})
				\in
					\FEt / k_{x},
			\\
					\alg{H}_{x}^{2}(\Hat{\Order}_{x}, \mathcal{A})
				=
					\alg{H}_{x}^{2}(\Hat{\Order}_{x}, \mathcal{A}_{0})
				=
					\alg{H}^{1}(\Hat{K}_{x}, A)
				\in
					\Ind \Alg_{\uc} / k_{x}.
		\end{gather*}
	The isomorphic groups in the third line are divisible.
\end{Prop}

\begin{proof}
	The group $\alg{H}^{1}(\Hat{K}_{x}, A)$ is divisible
	since $\alg{H}^{2}(\Hat{K}_{x}, A[n]) = 0$ for any $n \ge 1$
	by \cite[Prop.\ (3.4.3) (b)]{Suz14}.
	The statements about $R \alg{\Gamma}(\Hat{\Order}_{x}, \mathcal{A})$,
	$R \alg{\Gamma}(\Hat{\Order}_{x}, \mathcal{A}_{0})$ follow from
	\cite[Prop.\ (3.4.2) (a), (3.4.3) (d)]{Suz14}.
	These propositions, the localization triangle
		\[
				R \alg{\Gamma}_{x}(\Hat{\Order}_{x}, \mathcal{A})
			\to
				R \alg{\Gamma}(\Hat{\Order}_{x}, \mathcal{A})
			\to
				R \alg{\Gamma}(\Hat{K}_{x}, A)
		\]
	and the similar triangle for $\mathcal{A}_{0}$ imply the rest of the statements.
\end{proof}

We recall \cite[Rmk.\ (4.2.10)]{Suz14}.
Let $\mathcal{A}$ and $\mathcal{A}^{\vee}$ be the N\'eron models of $A$ and $A^{\vee}$, respectively.
Let $\mathcal{A}_{0}^{\vee}$ be the maximal open subgroup scheme of $\mathcal{A}^{\vee}$
with connected fibers.
The Poincar\'e biextension $A^{\vee} \tensor^{L} A \to \Gm[1]$
as a morphism in $D(\Hat{K}_{x, \fppf})$
canonically extends to a biextension
$\mathcal{A}_{0}^{\vee} \tensor^{L} \mathcal{A} \to \Gm[1]$
as a morphism in $D(\Hat{\Order}_{x, \fppf})$
by \cite[IX, 1.4.3]{Gro72}.
Hence we have a morphism
$\mathcal{A}_{0}^{\vee} \to R \sheafhom_{\Hat{\Order}_{x}}(\mathcal{A}, \Gm[1])$.
With the functoriality morphism \eqref{eq: functoriality morphism over integers, variant}
and the trace isomorphism \eqref{eq: local trace morphism},
we have morphisms
	\begin{align*}
				R \alg{\Gamma}_{x}(\Hat{\Order}_{x}, \mathcal{A}_{0}^{\vee})
		&	\to
				R \sheafhom_{k^{\ind\rat}_{x, \pro\et}} \bigl(
					R \alg{\Gamma}(\Hat{\Order}_{x}, \mathcal{A}),
					R \alg{\Gamma}_{x}(\Hat{\Order}_{x}, \Gm[1])
				\bigr)
		\\
		&	=
				R \sheafhom_{k^{\ind\rat}_{x, \pro\et}} \bigl(
					R \alg{\Gamma}(\Hat{\Order}_{x}, \mathcal{A}),
					\Z
				\bigr)
			=
				R \alg{\Gamma}(\Hat{\Order}_{x}, \mathcal{A})^{\SDual_{x}}.
	\end{align*}

\begin{Prop} \label{prop: local duality}
	The diagram
		\[
			\begin{CD}
					R \alg{\Gamma}(\Hat{\Order}_{x}, \mathcal{A}_{0}^{\vee})
				@>>>
					R \alg{\Gamma}(\Hat{K}_{x}, A^{\vee})
				@>>>
					R \alg{\Gamma}_{x}(\Hat{\Order}_{x}, \mathcal{A}_{0}^{\vee})[1]
				\\
				@VVV
				@VVV
				@VVV
				\\
					R \alg{\Gamma}_{x}(\Hat{\Order}_{x}, \mathcal{A})^{\SDual_{x}}
				@>>>
					R \alg{\Gamma}(\Hat{K}_{x}, A)^{\SDual_{x}}[1]
				@>>>
					R \alg{\Gamma}(\Hat{\Order}_{x}, \mathcal{A})^{\SDual_{x}}[1]
			\end{CD}
		\]
	is a morphism of distinguished triangles.
	Applying $\SDual_{x} \SDual_{x}$, the induced diagram
		\[
			\begin{CD}
					R \alg{\Gamma}(\Hat{\Order}_{x}, \mathcal{A}_{0}^{\vee})^{\SDual_{x} \SDual_{x}}
				@>>>
					R \alg{\Gamma}(\Hat{K}_{x}, A^{\vee})^{\SDual_{x} \SDual_{x}}
				@>>>
					R \alg{\Gamma}_{x}(\Hat{\Order}_{x}, \mathcal{A}_{0}^{\vee})[1]
				\\
				@VVV
				@VV \theta_{A} V
				@VVV
				\\
					R \alg{\Gamma}_{x}(\Hat{\Order}_{x}, \mathcal{A})^{\SDual_{x}}
				@>>>
					R \alg{\Gamma}(\Hat{K}_{x}, A)^{\SDual_{x}}[1]
				@>>>
					R \alg{\Gamma}(\Hat{\Order}_{x}, \mathcal{A})^{\SDual_{x}}[1]
			\end{CD}
		\]
	is an isomorphism of distinguished triangles.
\end{Prop}

\begin{proof}
	By \cite[Prop.\ (3.3.8)]{Suz14}, we know that
	the first diagram is a morphism of distinguished triangles.
	The terms in the first diagram can be identified with
	the terms in the first diagram of \cite[Prop.\ (4.2.3)]{Suz14}.
	The uniqueness stated in \cite[Prop.\ (4.2.3)]{Suz14} shows that
	the morphisms in the first diagram here and
	the morphisms in the first diagram of \cite[Prop.\ (4.2.3)]{Suz14} are equal.
	Hence the second diagram is an isomorphism of distinguished triangles
	by \cite[Prop.\ (4.2.7), Thm.\ (4.1.2)]{Suz14}.
\end{proof}

Assume that $k_{x}$ is a finite extension of another perfect field $k$.
We have a finite \'etale morphism $f_{x} \colon \Spec k_{x} \to \Spec k$.
This induces a morphism of sites
	\[
			f_{x}
		\colon
			\Spec k^{\ind\rat}_{x, \pro\et}
		\to
			\Spec k^{\ind\rat}_{\pro\et}
		\quad
	\]
Its pushforward functor is the Weil restriction functor $\Res_{k_{x} / k}$,
which is exact (\cite[Lem.\ 6.1.17]{BS15}).
We denote the composites of $\alg{\Gamma}(\Hat{\Order}_{x}, \var)$,
$R \alg{\Gamma}(\Hat{\Order}_{x}, \var)$ and $(f_{x})_{\ast} = \Res_{k_{x} / k}$ by
	\begin{gather*}
				\alg{\Gamma}(\Hat{\Order}_{x} / k, \var)
			\colon
				\Ab(\Hat{\Order}_{x, \fppf})
			\to
				\Ab(k^{\ind\rat}_{\pro\et}),
		\\
				R \alg{\Gamma}(\Hat{\Order}_{x} / k, \var)
			\colon
				D(\Hat{\Order}_{x, \fppf})
			\to
				D(k^{\ind\rat}_{\pro\et}),
	\end{gather*}
respectively, with cohomology
$\alg{H}^{n}(\Hat{\Order}_{x} / k, \var) = H^{n} R \alg{\Gamma}(\Hat{\Order}_{x} / k, \var)$.
Similar functors
	\begin{gather*}
				\alg{\Gamma}(\Hat{K}_{x} / k, \var)
			\colon
				\Ab(\Hat{K}_{x, \fppf})
			\to
				\Ab(k^{\ind\rat}_{\pro\et}),
		\\
				\alg{\Gamma}_{x}(\Hat{\Order}_{x} / k, \var)
			\colon
				\Ch(\Hat{\Order}_{x, \fppf})
			\to
				\Ch(k^{\ind\rat}_{\pro\et})
	\end{gather*}
and their derived functor are defined.
We have
	\[
			(f_{x})_{\ast}
			R \sheafhom_{k^{\ind\rat}_{x, \pro\et}}(\var, \Z)
		=
			R \sheafhom_{k^{\ind\rat}_{\pro\et}} \bigl(
				(f_{x})_{\ast}(\var), \Z
			\bigr)
	\]
by the duality for finite \'etale morphisms \cite[V, Prop.\ 1.13]{Mil80}
(which is for the \'etale topology; the same proof works for the pro-\'etale topology).
In other words, we have
	\[
			\Res_{k_{x} / k} \compose \SDual_{x}
		=
			\SDual \compose \Res_{k_{x} / k},
	\]
where $\SDual = R \sheafhom_{k^{\ind\rat}_{\pro\et}}(\var, \Z)$ as before.
Hence applying $\Res_{k_{x} / k}$ to the duality statements above over $k_{x}$ defines
some new duality statements over $k$.
Doing this for $\theta_{A}$ for instance defines a new isomorphism
	\[
			R \alg{\Gamma}(\Hat{K}_{x} / k, A^{\vee})^{\SDual \SDual}
		\isomto
			R \alg{\Gamma}(\Hat{K}_{x} / k, A)^{\SDual}
	\]
in $D(k^{\ind\rat}_{\pro\et})$.
(Here we are not using any specific property of our duality isomorphism.
We are using nothing but the obvious fact that
if a morphism $C \to D$ in $D(k^{\ind\rat}_{x, \pro\et})$ is an isomorphism,
then the induced morphism $\Res_{k_{x} / k} C \to \Res_{k_{x} / k} D$ is an isomorphism.)
More explicitly, this morphism comes from the morphisms
	\begin{align*}
				R \alg{\Gamma}(\Hat{K}_{x} / k, A^{\vee})
		&	\to
				R \alg{\Gamma} \bigl(
					\Hat{K}_{x} / k,
					R \sheafhom_{\Hat{K}_{x}}(A, \Gm)
				\bigr)[1]
		\\
		&	\to
				R \sheafhom_{k^{\ind\rat}_{\pro\et}}(
					R \alg{\Gamma}(\Hat{K}_{x} / k, A),
					R \alg{\Gamma}(\Hat{K}_{x} / k, \Gm)
				)[1]
		\\
		&	\to
				R \sheafhom_{k^{\ind\rat}_{\pro\et}}(
					R \alg{\Gamma}(\Hat{K}_{x} / k, A), \Z
				)[1]
			=
				R \alg{\Gamma}(\Hat{K}_{x} / k, A)^{\SDual}[1]
	\end{align*}
The trace morphism in this situation used here is
	\begin{equation} \label{eq: local trace morphism, Weil restricted}
			R \alg{\Gamma}(\Hat{K}_{x} / k, \Gm)
		\to
			R \alg{\Gamma}_{x}(\Hat{\Order}_{x} / k, \Gm)[1]
		=
			\Res_{k_{x} / k} \Hat{\alg{K}}_{x}^{\times} / \Hat{\alg{O}}_{x}^{\times}
		=
			\Res_{k_{x} / k} \Z
		\to
			\Z,
	\end{equation}
where the last morphism is
	\[
			\Z[\Hom_{k}(k_{x}, \closure{k})]
		\cong
			\Z^{[k_{x} : k]}
		\overset{\text{sum}}{\onto}
			\Z
	\]
on geometric points.


\subsection{Henselizations and completions}
\label{sec: Henselizations and completions}
Let $\Order_{x}^{h}$ be an excellent henselian discrete valuation ring
of equal characteristic $p > 0$
with perfect residue field $k_{x}$ ($\in \mathcal{U}_{0}$).
We denote the maximal ideal by $\ideal{p}_{x}^{h}$ and the fraction field by $K_{x}^{h}$.
The corresponding objects after completion are denoted by
$\Hat{\Order}_{x}$, $\Hat{\ideal{p}}_{x}$ and $\Hat{K}_{x}$.

We make a variant of the constructions in the previous subsection
using henselizations instead of completions.
For $k'_{x} \in k_{x}^{\ind\rat}$,
the henselization of the (non-local) ring $k'_{x} \tensor_{k_{x}} \Order_{x}^{h}$
at the ideal $k'_{x} \tensor_{k_{x}} \ideal{p}_{x}^{h}$ (\cite[Chap.\ XI, \S 2]{Ray70b})
is denoted by $(k'_{x} \tensor_{k_{x}} \Order_{x}^{h})^{h}$.
We define
	\[
			\alg{O}_{x}^{h}(k'_{x})
		=
			(k'_{x} \tensor_{k_{x}} \Order_{x}^{h})^{h},
		\quad
			\alg{K}_{x}^{h}(k'_{x})
		=
			\alg{O}_{x}^{h}(k'_{x}) \tensor_{\Order_{x}^{h}} K_{x}^{h}.
	\]
The functors $\alg{O}_{x}^{h}$ and $\alg{K}_{x}^{h}$ define premorphisms of sites
	\[
				\pi_{\Order_{x}^{h}}
			\colon
				\Spec \Order_{x, \fppf}^{h}
			\to
				\Spec k^{\ind\rat}_{x, \et},
		\quad
				\pi_{K_{x}^{h}}
			\colon
				\Spec K_{x, \fppf}^{h}
			\to
				\Spec k^{\ind\rat}_{x, \et}.
	\]
The pro-\'etale sheafifications of their pushforward functors are denoted by
	\begin{gather*}
				\alg{\Gamma}(\Order_{x}^{h}, \var),
			\colon
				\Ab(\Order_{x, \fppf}^{h})
			\to
				\Ab(k^{\ind\rat}_{x, \pro\et}),
		\\
				\alg{\Gamma}(K_{x}^{h}, \var)
			\colon
				\Ab(K_{x, \fppf}^{h})
			\to
				\Ab(k^{\ind\rat}_{x, \pro\et}).
	\end{gather*}
We set
	\begin{gather*}
				\alg{\Gamma}_{x}(\Order_{x}^{h}, \var)
			=
				\bigl[
						\alg{\Gamma}(\Order_{x}^{h}, \var),
					\to
						\alg{\Gamma}(K_{x}^{h}, j^{\ast} \var)
				\bigr][-1]
			\colon
		\\
				\Ch(\Order_{x, \fppf}^{h})
			\to
				\Ch(k^{\ind\rat}_{x, \pro\et}),
	\end{gather*}
where $j \colon \Spec K_{x}^{h} \to \Spec \Order_{x}^{h}$ is the natural morphism
inducing a morphism on the fppf sites.
We have their right derived functors
	\begin{gather*}
				R \alg{\Gamma}(\Order_{x}^{h}, \var),
			\colon
				D(\Order_{x, \fppf}^{h})
			\to
				D(k^{\ind\rat}_{x, \pro\et}),
		\\
				R \alg{\Gamma}(K_{x}^{h}, \var),
			\colon
				D(K_{x, \fppf}^{h})
			\to
				D(k^{\ind\rat}_{x, \pro\et}),
		\\
				R \alg{\Gamma}_{x}(\Order_{x}^{h}, \var)
			=
				\bigl[
						R \alg{\Gamma}(\Order_{x}^{h}, \var)
					\to
						R \alg{\Gamma}(K_{x}^{h}, j^{\ast} \var)
				\bigr][-1]
			\colon
		\\
				D(\Order_{x, \fppf}^{h})
			\to
				D(k^{\ind\rat}_{x, \pro\et}).
	\end{gather*}
Again, we frequently omit the restriction functor $j^{\ast}$ by abuse of notation.

We say that a sheaf $F \in \Set(\Order_{x, \fppf}^{h})$ is locally of finite presentation
if it commutes with filtered direct limits as a functor on the category of $\Order_{x}^{h}$-algebras.

\begin{Prop} \label{prop: cohom of henselization is fin pres}
	For any sheaf $A \in \Ab(\Order_{x, \fppf}^{h})$ locally of finite presentation and $n \ge 0$,
	the sheaf $R^{n} (\pi_{\Order_{x}^{h}})_{\ast} A \in \Ab(k^{\ind\rat}_{x, \et})$
	is locally of finite presentation.
	In particular, we have
	$\alg{H}^{n}(\Order_{x}^{h}, A) = R^{n} (\pi_{\Order_{x}^{h}})_{\ast} A$.
	That is, $\alg{H}^{n}(\Order_{x}^{h}, A)$ is the \'etale sheafification of the presheaf
		\[
				k'_{x}
			\mapsto
				H^{n}(\alg{O}_{x}^{h}(k'_{x}), A).
		\]
	A similar statement holds when $\Order_{x}^{h}$ is replaced by $K_{x}^{h}$.
\end{Prop}

\begin{proof}
	The sheaf $\alg{O}_{x}^{h}$ is locally of finite presentation
	by the construction of henselization.
	Let $A \in \Ab(\Order_{x, \fppf}^{h})$ be locally of finite presentation.
	Let $k'_{x} = \bigcup k'_{x, \lambda} \in k^{\ind\rat}$ with $k'_{x, \lambda} \in k^{\rat}_{x}$.
	Then for any $n \ge 0$, we have
		\[
				H^{n}(\alg{O}_{x}^{h}(k'_{x}), A)
			=
				\dirlim_{\lambda}
					H^{n}(\alg{O}_{x}^{h}(k'_{x, \lambda}), A).
		\]
	By sheafification, we know that $R^{n} (\pi_{\Order_{x}^{h}})_{\ast} A$ is locally of finite presentation.
	The same proof works for $K_{x}^{h}$.
\end{proof}

Define a functor $\Ch(\Order_{x, \fppf}^{h}) \to \Ch(k^{\ind\rat}_{x, \pro\et})$
of additive categories with translation by
	\[
			o_{x}(\Order_{x}^{h}, \var)
		=
			\bigl[
					\alg{\Gamma}_{x}(\Order_{x}^{h}, \var)
				\to
					\alg{\Gamma}_{x}(\Hat{\Order}_{x}, f^{\ast} \var)
			\bigr],
	\]
where $f \colon \Spec \Hat{\Order}_{x} \to \Spec \Order_{x}^{h}$
is the natural scheme morphism inducing a morphism on the fppf sites.
We call its right derived functor
	\[
			R o_{x}(\Order_{x}^{h}, \var)
		=
			\bigl[
					R \alg{\Gamma}_{x}(\Order_{x}^{h}, \var)
				\to
					R \alg{\Gamma}_{x}(\Hat{\Order}_{x}, f^{\ast} \var)
			\bigr]
	\]
the \emph{obstruction for cohomological approximation} (at $x$).
Define $D(\Order_{x, \fppf}^{h})_{\ca}$ to be the kernel of the functor $R o_{x}(\Order_{x}^{h}, \var)$,
i.e.\ the full subcategory of $D(\Order_{x, \fppf}^{h})$
consisting of objects $A$ with $R o_{x}(\Order_{x}^{h}, A) = 0$, or
	\[
			R \alg{\Gamma}_{x}(\Order_{x}^{h}, A)
		=
			R \alg{\Gamma}_{x}(\Hat{\Order}_{x}, f^{\ast} A).
	\]
Such an object $A$ is said to \emph{satisfy cohomological approximation}.

\begin{Prop} \label{prop: cohom with closed support doesnt change under completion}
	Any smooth group scheme or finite flat group scheme over $\Order_{x}^{h}$
	satisfies cohomological approximation.
\end{Prop}

\begin{proof}
	Any finite flat group scheme $N$ is a closed subgroup scheme of
	some smooth affine group scheme $G$ by \cite[Prop.\ 2.2.1]{Beg81}.
	The fppf quotient $H = G / N$ is a smooth affine group scheme by descent.
	We have a distinguished triangle
	$R o_{x}(\Order_{x}^{h}, N) \to R o_{x}(\Order_{x}^{h}, G) \to R o_{x}(\Order_{x}^{h}, H)$.
	If two of the terms are zero,
	then so is the other.
	Hence the finite flat case is reduced to the smooth case.
	
	Assume that $A$ is a smooth group scheme over $\Order_{x}^{h}$.
	Let $N \subset A$ be the schematic closure of the identity section of $A \times_{\Order_{x}^{h}} K_{x}^{h}$.
	Then $N$ is an \'etale group scheme over $\Order_{x}^{h}$ with trivial generic fiber
	and $A / N$ is a separated smooth group scheme over $\Order_{x}^{h}$
	by \cite[Prop.\ 3.3.5]{Ray70a}.
	The both objects $R \alg{\Gamma}_{x}(\Order_{x}^{h}, N)$
	and $R \alg{\Gamma}_{x}(\Hat{\Order}_{x}, N)$ are isomorphic to
	$N \times_{\Order_{x}^{h}} k_{x}$.
	
	Hence we may assume that $A$ is smooth separated.
	Then we have $\alg{H}_{x}^{0}(\Hat{\Order}_{x}, A) = 0$.
	Also $\alg{H}^{n}(\Hat{\Order}_{x}, A) = 0$ for $n \ge 1$
	by \cite[Prop.\ (3.4.2) (a)]{Suz14}.
	Hence
		\begin{gather*}
					\alg{H}_{x}^{1}(\Hat{\Order}_{x}, A)
				=
						\alg{\Gamma}(\Hat{K}_{x}, A)
					/
						\alg{\Gamma}(\Hat{\Order}_{x}, A),
			\\
					\alg{H}_{x}^{n}(\Hat{\Order}_{x}, A)
				=
					\alg{H}^{n - 1}(\Hat{K}_{x}, A)
		\end{gather*}
	for $n \ge 2$.
	These sheaves are locally of finite presentation (even before pro-\'etale sheafification)
	by the Greenberg approximation argument \cite[Prop.\ (3.2.8), (3.2.9)]{Suz14}.
	On the other hand, the sheaves $\alg{H}^{n}(\Order_{x}^{h}, A)$, $\alg{H}^{n}(K_{x}^{h}, A)$
	and thus $\alg{H}_{x}^{n}(\Order_{x}^{h}, A)$ are locally of finite presentation for $n \ge 0$
	by the previous proposition.
	Therefore it is enough to show that the morphism
		\[
				R \Gamma_{x}(\Order_{x}^{h}, A)
			\to
				R \Gamma_{x}(\Hat{\Order}_{x}, A)
		\]
	is an isomorphism when $k$ is algebraically closed
	by \cite[the second paragraph after Prop.\ (2.4.1)]{Suz14}.
	The statements to prove are
		\begin{gather*}
						\Gamma(K_{x}^{h}, A)
					/
						\Gamma(\Order_{x}^{h}, A)
				=
						\Gamma(\Hat{K}_{x}, A)
					/
						\Gamma(\Hat{\Order}_{x}, A),
			\\
					H^{n}(K_{x}^{h}, A)
				=
					H^{n}(\Hat{K}_{x}, A),
				\quad
					n \ge 1.
		\end{gather*}
	They can be proven in the same way as
	the Greenberg approximation argument \cite[Prop.\ (3.2.8), (3.2.9)]{Suz14}
	(or equivalently, by the same argument as \cite[I, Rmk.\ 3.10]{Mil06}
	combined with \cite[Prop.\ 3.5 (b)]{Ces15}).
\end{proof}

Perhaps any group scheme locally of finite type over $\Order_{x}^{h}$
might satisfy cohomological approximation
since the Greenberg approximation itself holds in this generality.
We do not pursue this point.
But see \cite[Lem.\ 2.6, Rmk.\ 2.7]{DH18} for a related result.

If $k_{x}$ is a finite extension of another perfect field $k$
and $f \colon \Spec k_{x} \to \Spec k$ is the natural morphism,
then the composite of $\alg{\Gamma}(\Order_{x}^{h}, \var)$
and $f_{\ast}$ is denoted by
	\[
			\alg{\Gamma}(\Order_{x}^{h} / k, \var)
		\colon
			\Ch(\Order_{x, \fppf}^{h})
		\to
			\Ch(k^{\ind\rat}_{\pro\et})
	\]
with right derived functor
	\[
			R \alg{\Gamma}(\Order_{x}^{h} / k, \var)
		\colon
			D(\Order_{x, \fppf}^{h})
		\to
			D(k^{\ind\rat}_{\pro\et}).
	\]
Similar notation applies to other objects, defining
$\alg{\Gamma}(K_{x}^{h} / k, \var)$,
$\alg{\Gamma}_{x}(\Order_{x}^{h} / k, \var)$,
$o_{x}(\Order_{x}^{h} / k, \var)$
and their derived functors.


\subsection{The fppf site of a curve over the rational \'etale site of the base}
\label{sec: The fppf site of a curve over the rational etale site of the base}

Let $U$ be a smooth geometrically connected curve
over a perfect field $k$ ($\in \mathcal{U}_{0}$) of characteristic $p > 0$,
with smooth compactification $X$ and function field $K$.
By the fppf site $U_{\fppf}$ of $U$,
we mean the category of ($\mathcal{U}_{0}$-small) $U$-schemes
endowed with the fppf topology.
For $k' \in k^{\ind\rat}$, we denote $U_{k'} = U \times_{k} k'$.
The functor sending $k' \in k^{\ind\rat}$ to $U_{k'}$ defines a premorphism of sites
	\[
			\pi_{U}
		\colon
			U_{\fppf}
		\to
			\Spec k^{\ind\rat}_{\et}.
	\]
We define a left exact functor
	\[
			\alg{\Gamma}(U, \var)
		\colon
			\Ab(U_{\fppf})
		\to
			\Ab(k^{\ind\rat}_{\pro\et})
	\]
by the composite of the pushforward functor $\pi_{U \ast}$ and the pro-\'etale sheafification.
We have its right derived functor
	\[
			R \alg{\Gamma}(U, \var)
		\colon
			D(U_{\fppf})
		\to
			D(k^{\ind\rat}_{\pro\et}),
	\]
with cohomologies $\alg{H}^{n} = H^{n} R \alg{\Gamma}$.
For any $A \in \Ab(U_{\fppf})$ and $n \ge 0$,
the sheaf $\alg{H}^{n}(U, A)$ on $\Spec k^{\ind\rat}_{\pro\et}$ is
the pro-\'etale sheafification of the presheaf
	\[
			k'
		\mapsto
			H^{n}(U_{k'}, A).
	\]

Let $\Spec k^{\perf}_{\et}$ be the category of ($\mathcal{U}_{0}$-small) perfect $k$-schemes
endowed with the \'etale topology.
The structure morphism $U \to \Spec k$ induces a morphism of sites
	\[
			\pi_{U}^{\perf} \colon U_{\fppf}
		\to
			\Spec k^{\perf}_{\et},
	\]
which is defined by the functor that sends a perfect $k$-scheme $S$ to $U \times_{k} S$.
In \cite[(3.1)]{AM76}, this is denoted by $\pi^{f, p}$.

\begin{Prop} \label{prop: comparison with Artin Milne}
	Let
		\[
				f
			\colon
				\Spec k^{\perf}_{\et}
			\to
				\Spec k^{\ind\rat}_{\et}
		\]
	be the premorphism of sites defined by the identity functor.
	Then we have $\pi_{U} = f \compose \pi_{U}^{\perf}$.
	The functor $f_{\ast}$ is exact.
	The functors $\alg{\Gamma}(U, \var)$ and $R \alg{\Gamma}(U, \var)$
	are the pro-\'etale sheafifications of $f_{\ast} \pi_{U \ast}^{\perf}$ and $f_{\ast} R \pi_{U \ast}^{\perf}$,
	respectively.
\end{Prop}

\begin{proof}
	Obvious.
\end{proof}

In this sense, our constructions are
pro-\'etale sheafifications of restrictions (from all perfect schemes to only ind-rational algebras)
of the constructions in \cite{AM76}.
In the next section, we will translate results in \cite{AM76} to our setting in this way.

For a closed point $x \in X$,
let $k_{x}$, $\Order_{x}^{h}$, $\Hat{\Order}_{x}$, $K_{x}^{h}$, $\Hat{K}_{x}$
be the residue field, the henselian local ring, the completed local ring, their fraction fields,
respectively, at $x$.
The results and notation in the previous two subsections apply to $\Hat{\Order}_{x}$ and $\Order_{x}^{h}$.

Assume $x \in U$.
For any $k' \in k^{\ind\rat}$, we have
	$
			\alg{O}_{x}^{h}(k' \tensor_{k} k_{x})
		=
			(k' \tensor_{k} \Order_{x}^{h})^{h}.
	$
Hence the morphism $\Spec \Order_{x}^{h} \to U$ induces a morphism
$\Spec \alg{O}_{x}^{h}(k' \tensor_{k} k_{x}) \to U_{k'}$.
This induces a homomorphism
	\[
			\Gamma(U_{k'}, A)
		\to
			\Gamma(\alg{O}_{x}^{h}(k' \tensor_{k} k_{x}), A)
	\]
for any $A \in \Ab(U_{\fppf})$.
Thus we have a morphism
	\[
			\alg{\Gamma}(U, \var)
		\to
			\alg{\Gamma}(\Order_{x}^{h} / k, \var)
	\]
of left exact functors $\Ab(U_{\fppf}) \to \Ab(k^{\ind\rat}_{\pro\et})$
and a morphism
	\[
			R \alg{\Gamma}(U, \var)
		\to
			R \alg{\Gamma}(\Order_{x}^{h} / k, \var)
	\]
of triangulated functors $D(U_{\fppf}) \to D(k^{\ind\rat}_{\pro\et})$.

Assume $x \not\in U$.
For any $k' \in k^{\ind\rat}$, we similarly have a morphism
$\Spec \alg{K}_{x}^{h}(k' \tensor_{k} k_{x}) \to U_{k'}$.
This induces a homomorphism
	\[
			\Gamma(U_{k'}, A)
		\to
			\Gamma(\alg{K}_{x}^{h}(k' \tensor_{k} k_{x}), A)
	\]
for any $A \in \Ab(U_{\fppf})$.
Thus we have a morphism
	\[
			\alg{\Gamma}(U, \var)
		\to
			\alg{\Gamma}(K_{x}^{h} / k, \var)
	\]
of left exact functors $\Ab(U_{\fppf}) \to \Ab(k^{\ind\rat}_{\pro\et})$
and a morphism
	\[
			R \alg{\Gamma}(U, \var)
		\to
			R \alg{\Gamma}(K_{x}^{h} / k, \var)
	\]
of triangulated functors $D(U_{\fppf}) \to D(k^{\ind\rat}_{\pro\et})$.

For a dense open subscheme $V \subset U$, we have a morphism
	\[
			\alg{\Gamma}(U, \var)
		\to
			\alg{\Gamma}(V, \var)
	\]
of left exact functors $\Ab(U_{\fppf}) \to \Ab(k^{\ind\rat}_{\pro\et})$
and a morphism
	\[
			R \alg{\Gamma}(U, \var)
		\to
			R \alg{\Gamma}(V, \var)
	\]
of triangulated functors $D(U_{\fppf}) \to D(k^{\ind\rat}_{\pro\et})$.
Let $Z = U \setminus V$, which is a finite set of closed points of $X$.
We define a functor of additive categories with translation by
	\begin{gather*}
				\alg{\Gamma}_{Z}(U, \var)
			=
				\bigl[
						\alg{\Gamma}(U, \var)
					\to
						\alg{\Gamma}(V, \var)
				\bigr][-1]
			\colon
		\\
				\Ch(U_{\fppf})
			\to
				\Ch(k^{\ind\rat}_{\pro\et}).
	\end{gather*}
We have its right derived functor
	\begin{gather*}
				R \alg{\Gamma}_{Z}(U, \var)
			=
				\bigl[
						R \alg{\Gamma}(U, \var)
					\to
						R \alg{\Gamma}(V, \var)
				\bigr][-1]
			\colon
		\\
				D(U_{\fppf})
			\to
				D(k^{\ind\rat}_{\pro\et}).
	\end{gather*}
When $Z = \{x\}$, these are also denoted by $\alg{\Gamma}_{x}(U, \var)$
and $R \alg{\Gamma}_{x}(U, \var)$.

\begin{Prop} \label{prop: excision}
	Let $V \subset U$ be a dense open subscheme and set $Z = U \setminus V$.
	The diagram
		\[
			\begin{CD}
					\alg{\Gamma}(U, \var)
				@>>>
					\alg{\Gamma}(V, \var)
				\\
				@VVV
				@VVV
				\\
					\bigoplus_{x \in U \setminus V}
						\alg{\Gamma}(\Order_{x}^{h} / k, \var)
				@>>>
					\bigoplus_{x \in U \setminus V}
						\alg{\Gamma}(K_{x}^{h} / k, \var)
			\end{CD}
		\]
	of left exact functors $\Ab(U_{\fppf}) \to \Ab(k^{\ind\rat}_{\pro\et})$ is commutative.
	The induced morphism
		\[
				R \alg{\Gamma}_{Z}(U, \var)
			\to
				\bigoplus_{x \in U \setminus V}
					R \alg{\Gamma}_{x}(\Order_{x}^{h} / k, \var)
		\]
	of triangulated functors $D(U_{\fppf}) \to D(k^{\ind\rat}_{\pro\et})$
	is an isomorphism.
\end{Prop}

\begin{proof}
	The first statement about the commutativity is obvious.
	For the second statement, it is enough to treat the case that $Z$ is a singleton $\{x\}$.
	We want to show that the morphism
		\[
				R \alg{\Gamma}_{x}(U, \var)
			\to
				R \alg{\Gamma}_{x}(\Order_{x}^{h} / k, \var)
		\]
	is an isomorphism.
	Let $f_{x} \colon \Spec k_{x} \to \Spec k$ be the natural morphism.
	The above morphism is the pro-\'etale sheafification of
		\[
				[R \pi_{U \ast} \to R \pi_{V \ast}][-1]
			\to
				f_{x \ast}
				[R (\pi_{\Order_{x}^{h}})_{\ast} \to R (\pi_{K_{x}^{h}})_{\ast}][-1].
		\]
	Applying $R \Gamma(k', \var)$ for $k' \in k^{\ind\rat}$ to this morphism before sheafification,
	we have a morphism
		\[
				R \Gamma_{x}(U_{k'}, \var)
			\to
				R \Gamma_{x} \bigl(
					(k' \tensor_{k} \Order_{x}^{h})^{h}, \var
				\bigr),
		\]
	where the left-hand (resp.\ right-hand) side is the fppf cohomology of
	$U_{k'} = U \times_{k} k'$ (resp.\ $(k' \tensor_{k} \Order_{x}^{h})^{h}$)
	with support on $x \times_{k} k'$ (resp.\ the ideal generated by $k' \tensor_{k} \ideal{p}_{x}^{h}$)
	(\cite[III, Prop.\ 0.3]{Mil06}).
	It is enough to show that this morphism is an isomorphism for any $k'$.
	We may assume that $k' \in k^{\rat}$
	since cohomology and henselization commute with filtered inverse limits of schemes.
	Then $k'$ is a finite product of perfect fields.
	Hence we may assume that $k'$ is a perfect field.
	Replacing $k$ by $k'$, we may assume that $k' = k$.
	Hence we are reduced to showing that
		\[
				R \Gamma_{x}(U, \var)
			\isomto
				R \Gamma_{x}(\Order_{x}^{h}, \var)
		\]
	on $D(U_{\fppf})$.
	Since $\Spec \Order_{x}^{h}$ is a filtered inverse limit of affine \'etale $U$-schemes
	and cohomology commutes with such limits,
	we can push them forward to the \'etale sites.
	The statement to prove is thus
		\[
				R \Gamma_{x}(U_{\et}, \var)
			\isomto
				R \Gamma_{x}(\Order_{x, \et}^{h}, \var)
		\]
	on $D(U_{\et})$.
	This is the excision isomorphism of \'etale cohomology \cite[III, Cor.\ 1.28]{Mil80}.
\end{proof}

We define a functor of additive categories with translation by
	\begin{gather*}
				\alg{\Gamma}_{c}(U, \var)
			=
				\Bigl[
						\alg{\Gamma}(U, \var)
					\to
						\bigoplus_{x \not\in U}
							\alg{\Gamma}(\Hat{K}_{x} / k, \var)
				\Bigr][-1]
			\colon
		\\
				\Ch(U_{\fppf})
			\to
				\Ch(k^{\ind\rat}_{\pro\et}),
	\end{gather*}
where the sum is over all $x \in X \setminus U$.
We have its right derived functor
	\begin{gather*}
				R \alg{\Gamma}_{c}(U, \var)
			=
				\Bigl[
						R \alg{\Gamma}(U, \var)
					\to
						\bigoplus_{x \not\in U}
							R \alg{\Gamma}(\Hat{K}_{x} / k, \var)
				\Bigr][-1]
			\colon
		\\
				D(U_{\fppf})
			\to
				D(k^{\ind\rat}_{\pro\et}).
	\end{gather*}
Here we are working with unbounded complexes,
making the definition of compact support cohomology more involved than \cite{DH18}.
This is important in view of the definition of the pairing \eqref{eq: global cup product} below
(which uses derived tensor product $\tensor^{L}$).

For another dense open subscheme $V \subset U$,
unfortunately there is no obvious morphism from
$R \alg{\Gamma}_{c}(V, \var)$ to $R \alg{\Gamma}_{c}(U, \var)$,
and no natural distinguished triangle
	\[
			R \alg{\Gamma}_{c}(V, A)
		\to
			R \alg{\Gamma}_{c}(U, A)
		\to
			\bigoplus_{x \in U \setminus V}
				R \alg{\Gamma}(\Hat{\Order}_{x} / k, A),
	\]
unless $A \in D(U_{\fppf})$ satisfies cohomological approximation at all $x \in U \setminus V$
(see Prop.\ \ref{prop: cpt support cohom triangle} and
\ref{prop: cpt supp triangle after cohom approx} below).
We need to define a variant of $R \alg{\Gamma}_{c}(V, \var)$
that does admit a natural morphism to $R \alg{\Gamma}_{c}(U, \var)$.
We define a functor $\Ch(U_{\fppf}) \to \Ch(k^{\ind\rat}_{\pro\et})$
of additive categories with translation by
	\begin{align*}
				\alg{\Gamma}_{c}(V, U, \var)
		&	=
				\Bigl[
						\alg{\Gamma}_{c}(U, \var)
					\to
						\bigoplus_{x \in U \setminus V}
							\alg{\Gamma}(\Hat{\Order}_{x} / k, \var)
				\Bigr][-1]
		\\
		&	=
				\Bigl[
						\alg{\Gamma}(U, \var)
					\to
							\bigoplus_{x \not\in U}
								\alg{\Gamma}(\Hat{K}_{x} / k, \var)
						\oplus
							\bigoplus_{x \in U \setminus V}
								\alg{\Gamma}(\Hat{\Order}_{x} / k, \var)
				\Bigr][-1].
	\end{align*}
Here we used the fact that
for any two morphisms $A, B \to C$ of complexes in an additive category,
we have natural isomorphisms of complexes
	\[
			[A \oplus B \to C]
		\cong
			[A \to [B \to C]]
		\cong
			[B \to [A \to C]],
	\]
or dually, for any two morphisms $A \to B, C$ of complexes in an additive category,
we have natural isomorphisms of (mapping fiber) complexes
	\[
			[A \to B \oplus C][-1]
		\cong
			\bigl[
				[A \to B][-1] \to C
			\bigr][-1]
		\cong
			\bigl[
				[A \to C][-1] \to B
			\bigr][-1].
	\]
We have the right derived functor $D(U_{\fppf}) \to D(k^{\ind\rat}_{\pro\et})$
	\begin{align*}
				R \alg{\Gamma}_{c}(V, U, \var)
		&	=
				\Bigl[
						R \alg{\Gamma}_{c}(U, \var)
					\to
						\bigoplus_{x \in U \setminus V}
							R \alg{\Gamma}(\Hat{\Order}_{x} / k, \var)
				\Bigr][-1]
		\\
		&	=
				\Bigl[
						R \alg{\Gamma}(U, \var)
					\to
							\bigoplus_{x \not\in U}
								R \alg{\Gamma}(\Hat{K}_{x} / k, \var)
						\oplus
							\bigoplus_{x \in U \setminus V}
								R \alg{\Gamma}(\Hat{\Order}_{x} / k, \var)
				\Bigr][-1].
	\end{align*}
By definition, we have a distinguished triangle
	\[
			R \alg{\Gamma}_{c}(V, U, \var)
		\to
			R \alg{\Gamma}_{c}(U, \var)
		\to
			\bigoplus_{x \in U \setminus V}
				R \alg{\Gamma}(\Hat{\Order}_{x} / k, \var).
	\]
By this, we mean the values of these functors at any object form a distinguished triangle.

We also need to define a variant of $R \alg{\Gamma}_{c}(U, \var)$.
Define
	\[
			\alg{\Gamma}_{c}(U, V, \var)
		=
			\Bigl[
					\alg{\Gamma}(V, \var)
				\to
					\bigoplus_{x \notin U}
						\alg{\Gamma}(\Hat{K}_{x} / k, \var)
					\oplus
					\bigoplus_{x \in U \setminus V}
						\alg{\Gamma}_{x}(\Hat{\Order}_{x} / k, \var)[1]
			\Bigr][-1],
	\]
where the morphism from $\alg{\Gamma}(V, \var)$
to $\alg{\Gamma}_{x}(\Hat{\Order}_{x} / k, \var)[1]$
is the composite
	\[
			\alg{\Gamma}(V, \var)
		\to
			\alg{\Gamma}(K_{x}^{h} / k, \var)
		\to
			\alg{\Gamma}(\Hat{K}_{x} / k, \var)
		\to
			\alg{\Gamma}_{x}(\Hat{\Order}_{x} / k, \var)[1].
	\]
Its derived functor is
	\[
			R \alg{\Gamma}_{c}(U, V, \var)
		=
			\Bigl[
					R \alg{\Gamma}(V, \var)
				\to
					\bigoplus_{x \notin U}
						R \alg{\Gamma}(\Hat{K}_{x} / k, \var)
					\oplus
					\bigoplus_{x \in U \setminus V}
						R \alg{\Gamma}_{x}(\Hat{\Order}_{x} / k, \var)[1]
			\Bigr][-1].
	\]
For any $x \in U \setminus V$, the inclusion into the second summand defines morphisms
	\[
			\alg{\Gamma}_{x}(\Hat{\Order}_{x} / k, \var)
		\to
			\alg{\Gamma}_{c}(U, V, \var)
	\]
and
	\begin{equation} \label{eq: local to global support cohomology morphism}
			R \alg{\Gamma}_{x}(\Hat{\Order}_{x} / k, \var)
		\to
			R \alg{\Gamma}_{c}(U, V, \var).
	\end{equation}
The last morphism will be used later to connect trace morphisms in the local and global situations.

Let $Z = U \setminus V$.
Define
	\[
			o_{Z}(U, \var)
		=
			\bigr[
					\alg{\Gamma}_{Z}(U, \var)
				\to
					\bigoplus_{x \in U \setminus V}
						\alg{\Gamma}_{x}(\Hat{\Order}_{x} / k, \var)
			\bigr],
	\]
with derived functor
	\[
			R o_{Z}(U, \var)
		=
			\bigr[
					R \alg{\Gamma}_{Z}(U, \var)
				\to
					\bigoplus_{x \in U \setminus V}
						R \alg{\Gamma}_{x}(\Hat{\Order}_{x} / k, \var)
			\bigr].
	\]

We explain the notation to be used in the next proposition.
A commutative diagram of distinguished triangles in a triangulated category
	\[
		\begin{CD}
			A @>>> B @>>> C @>>> A[1] \\
			@VVV @VVV @VVV @VVV \\
			A' @>>> B' @>>> C' @>>> A'[1] \\
			@VVV @VVV @VVV @VVV \\
			A'' @>>> B'' @>>> C'' @>>> A''[1] \\
			@VVV @VVV @VVV @VVV \\
			A[1] @>>> B[1] @>>> C[1] @>>> A[2]
		\end{CD}
	\]
means a commutative diagram all of whose rows and columns are distinguished triangles
(where the right lower square is actually ``anti-commutative'' \cite[Diagram (10.5.5)]{KS06},
but we largely ignore commutative vs.\ anti-commutative issues,
which is especially harmless if $A$ or $C''$ is zero for example).
As usual, we will hide the shifted terms $B[1], A'' [1], A[2]$ etc.\ for brevity
and mention the remaining $3 \times 3$ diagram
as a commutative diagram of distinguished triangles.
A commutative diagram of distinguished triangles of triangulated functors
	\[
		\begin{CD}
			F @>>> G @>>> H \\
			@VVV @VVV @VVV \\
			F' @>>> G' @>>> H' \\
			@VVV @VVV @VVV \\
			F'' @>>> G'' @>>> H'' \\
		\end{CD}
	\]
means a commutative diagram of morphisms of triangulated functors
whose values at any object form a commutative diagram of distinguished triangles in the above sense.

\begin{Prop} \label{prop: cpt support cohom triangle}
	Let $V \subset U$ be a dense open subscheme and set $Z = U \setminus V$.
	The natural morphisms form a commutative diagram
		\[
			\begin{CD}
					\alg{\Gamma}_{c}(V, U, \var)
				@>>>
					\alg{\Gamma}_{c}(U, \var)
				@>>>
					\bigoplus_{x \in U \setminus V}
						\alg{\Gamma}(\Hat{\Order}_{x} / k, \var)
				\\
				@VVV
				@VVV
				@|
				\\
					\alg{\Gamma}_{c}(V, \var)
				@>>>
					\alg{\Gamma}_{c}(U, V, \var)
				@>>>
					\bigoplus_{x \in U \setminus V}
						\alg{\Gamma}(\Hat{\Order}_{x} / k, \var)
				\\
				@VVV
				@VVV
				@VVV
				\\
					o_{Z}(U, \var)
				@=
					o_{Z}(U, \var)
				@>>>
					0
			\end{CD}
		\]
	of distinguished triangles of triangulated functors
	$K(U_{\fppf}) \to K(k^{\ind\rat}_{\pro\et})$.
	With the isomorphism
		\[
				R o_{Z}(U, \var)
			\isomto
				\bigoplus_{x \in U \setminus V}
					R o_{x}(\Order_{x}^{h} / k, \var)
		\]
	coming from the previous proposition,
	we have a canonical commutative diagram
		\[
			\begin{CD}
					R \alg{\Gamma}_{c}(V, U, \var)
				@>>>
					R \alg{\Gamma}_{c}(U, \var)
				@>>>
					\bigoplus_{x \in U \setminus V}
						R \alg{\Gamma}(\Hat{\Order}_{x} / k, \var)
				\\
				@VVV
				@VVV
				@|
				\\
					R \alg{\Gamma}_{c}(V, \var)
				@>>>
					R \alg{\Gamma}_{c}(U, V, \var)
				@>>>
					\bigoplus_{x \in U \setminus V}
						R \alg{\Gamma}(\Hat{\Order}_{x} / k, \var)
				\\
				@VVV
				@VVV
				@VVV
				\\
					\bigoplus_{x \in U \setminus V}
						R o_{x}(\Order_{x}^{h} / k, \var)
				@=
					\bigoplus_{x \in U \setminus V}
						R o_{x}(\Order_{x}^{h} / k, \var)
				@>>>
					0
			\end{CD}
		\]
	of distinguished triangles of triangulated functors
	$D(U_{\fppf}) \to D(k^{\ind\rat}_{\pro\et})$.
\end{Prop}

\begin{proof}
	For $A \in \Ch(U_{\fppf})$, let
		\begin{gather*}
					C
				=
					\alg{\Gamma}(U, A),
				\quad
					D
				=
					\alg{\Gamma}(V, A),
				\quad
					E
				=
					\bigoplus_{x \not\in U}
						\alg{\Gamma}(\Hat{K}_{x} / k, A),
			\\
					C'
				=
					\bigoplus_{x \in U \setminus V}
						\alg{\Gamma}(\Hat{\Order}_{x} / k, A),
				\quad
					D'
				=
					\bigoplus_{x \in U \setminus V}
						\alg{\Gamma}(\Hat{K}_{x} / k, A).
		\end{gather*}
	These are objects of $\Ch(k^{\ind\rat}_{\pro\et})$.
	We have a natural commutative diagram
		\[
			\begin{CD}
				C @>>> D @>>> E \\
				@VVV @VVV @. \\
				C' @>>> D' @.
			\end{CD}
		\]
	in $\Ch(k^{\ind\rat}_{\pro\et})$.
	The value at $A$ of the first diagram shifted by one can be written as
		\[
			\begin{CD}
					[C \to E \oplus C']
				@>>>
					[C \to E]
				@>>>
					C'[1]
				\\
				@VVV @VVV @|
				\\
					[D \to E \oplus D']
				@>>>
					\bigl[
						D \to E \oplus [C' \to D']
					\bigr]
				@>>>
					C'[1]
				\\
				@VVV @VVV @VVV
				\\
					\bigl[
						[C \to D] \to [C' \to D']
					\bigr]
				@=
					\bigl[
						[C \to D] \to [C' \to D']
					\bigr]
				@>>>
					0
			\end{CD}
		\]
	(which is actually a $4 \times 4$ diagram as we are omitting the shifted terms).
	It is routine to check that this diagram is commutative up to homotopy.
	The second diagram results from the first.
\end{proof}

We say that an object $A \in D(U_{\fppf})$
\emph{satisfies cohomological approximation} if $R o_{x}(\Order_{x}^{h}, A) = 0$ for any $x \in U$.
We denote by $D(U_{\fppf})_{\ca}$
the full subcategory of $D(U_{\fppf})$
consisting of objects satisfying cohomological approximation.
It is a triangulated subcategory.

\begin{Prop} \label{prop: cpt supp triangle after cohom approx}
	On $D(U_{\fppf})_{\ca}$, we have isomorphisms
		\[
				R \alg{\Gamma}_{c}(V, U, \var)
			=
				R \alg{\Gamma}_{c}(V, \var),
			\quad
				R \alg{\Gamma}_{c}(U, V, \var)
			=
				R \alg{\Gamma}_{c}(U, \var),
		\]
	a distinguished triangle
		\[
				R \alg{\Gamma}_{c}(V, \var)
			\to
				R \alg{\Gamma}_{c}(U, \var)
			\to
				\bigoplus_{x \in U \setminus V}
					R \alg{\Gamma}(\Hat{\Order}_{x} / k, \var)
		\]
	and a morphism
		\[
				R \alg{\Gamma}_{x}(\Hat{\Order}_{x} / k, \var)
			\to
				R \alg{\Gamma}_{c}(U, \var)
		\]
	for any $x \in U$ compatible with \eqref{eq: local to global support cohomology morphism}.
\end{Prop}

\begin{proof}
	Obvious from the previous proposition.
\end{proof}

\begin{Prop} \label{prop: smooths and finite flats satisfy cohom approx}
	If $A$ is a smooth group scheme or a finite flat group scheme over $U$,
	then $A \in D(U_{\fppf})_{\ca}$.
\end{Prop}

\begin{proof}
	This follows from Prop.\ \ref{prop: cohom with closed support doesnt change under completion}.
\end{proof}

Let $B, C \in D(U_{\fppf})$.
To simplify the notation, we denote
$R \sheafhom_{\Order_{x}^{h}}$ for any $x \in U$ by
$[\var, \var]_{\Order_{x}^{h}}$.
Denote similarly $R \sheafhom_{U}$ by
$[\var, \var]_{U}$
and $R \sheafhom_{k^{\ind\rat}_{\pro\et}}$ by
$[\var, \var]_{k}$.
A similar construction to \eqref{eq: functoriality morphism over integers}
defines a morphism
	\[
			R \alg{\Gamma}(U, [B, C]_{U})
		\to
			\bigl[
				R \alg{\Gamma}_{c}(U, B),
				R \alg{\Gamma}_{c}(U, C)
			\bigr]_{k}.
	\]
This is equivalent to a morphism
	\begin{equation} \label{eq: global cup product}
			R \alg{\Gamma}(U, A) \tensor^{L} R \alg{\Gamma}_{c}(U, B)
		\to
			R \alg{\Gamma}_{c}(U, A \tensor^{L} B)
	\end{equation}
via the derived tensor-Hom adjunction \cite[Thm.\ 18.6.4 (vii)]{KS06}
and the change of variables
$[B, C]_{U} \leadsto A$ and $A \tensor^{L} B \leadsto C$.

For each $x \in U \setminus V$, we have natural morphisms
	\begin{align*}
		&
				R \alg{\Gamma}_{x}(\Order_{x}^{h} / k, [B, C]_{\Order_{x}^{h}})
			\to
				R \alg{\Gamma}_{x}(\Hat{\Order}_{x} / k, [B, C]_{\Order_{x}^{h}})
		\\
		&	\to
				\bigl[
					R \alg{\Gamma}(\Hat{\Order}_{x} / k, B),
					R \alg{\Gamma}_{x}(\Hat{\Order}_{x} / k, C)
				\bigr]_{k}
			\to
				\bigl[
					R \alg{\Gamma}(\Hat{\Order}_{x} / k, B),
					R \alg{\Gamma}_{c}(U, V, C)
				\bigr]_{k}
	\end{align*}
using the morphisms \eqref{eq: functoriality morphism over integers, variant} and
\eqref{eq: local to global support cohomology morphism}.
Using the morphisms in Prop.\ \ref{prop: cpt support cohom triangle},
we have morphisms
	\begin{gather*}
				R \alg{\Gamma}(U, [B, C]_{U})
			\to
				\bigl[
					R \alg{\Gamma}_{c}(U, B),
					R \alg{\Gamma}_{c}(U, C)
				\bigr]_{k}
			\to
				\bigl[
					R \alg{\Gamma}_{c}(U, B),
					R \alg{\Gamma}_{c}(U, V, C)
				\bigr]_{k},
		\\
				R \alg{\Gamma}(V, [B, C]_{V})
			\to
				\bigl[
					R \alg{\Gamma}_{c}(V, B),
					R \alg{\Gamma}_{c}(V, C)
				\bigr]_{k}
			\to
				\bigl[
					R \alg{\Gamma}_{c}(V, U, B),
					R \alg{\Gamma}_{c}(U, V, C)
				\bigr]_{k}.
	\end{gather*}

\begin{Prop} \label{prop: compatibility of local global pairings without cohom approx}
	The above morphisms give a morphism from the distinguished triangles
		\[
				\bigoplus_{x \in U \setminus V}
					R \alg{\Gamma}_{x}(\Order_{x}^{h} / k, [B, C]_{\Order_{x}^{h}})
			\to
				R \alg{\Gamma}(U, [B, C]_{U})
			\to
				R \alg{\Gamma}(V, [B, C]_{V})
		\]
	to the distinguished triangle
		\begin{align*}
				\bigoplus_{x \in U \setminus V}
					\bigl[
						R \alg{\Gamma}(\Hat{\Order}_{x} / k, B),
						R \alg{\Gamma}_{c}(U, V, C)
					\bigr]_{k}
		&	\to
				\bigl[
					R \alg{\Gamma}_{c}(U, B),
					R \alg{\Gamma}_{c}(U, V, C)
				\bigr]_{k}
		\\
		&	\to
				\bigl[
					R \alg{\Gamma}_{c}(V, U, B),
					R \alg{\Gamma}_{c}(U, V, C)
				\bigr]_{k}.
		\end{align*}
\end{Prop}

\begin{proof}
	We only prove the commutativity of the square
		\[
			\begin{CD}
					R \alg{\Gamma}_{x}(\Order_{x}^{h} / k, [B, C]_{\Order_{x}^{h}})
				@>>>
					R \alg{\Gamma}(U, [B, C]_{U})
				\\
				@VVV
				@VVV
				\\
					\bigl[
						R \alg{\Gamma}(\Hat{\Order}_{x} / k, B),
						R \alg{\Gamma}_{c}(U, V, C)
					\bigr]_{k}
				@>>>
					\bigl[
						R \alg{\Gamma}_{c}(U, B),
						R \alg{\Gamma}_{c}(U, V, C)
					\bigr]_{k}
			\end{CD}
		\]
	for $x \in U \setminus V$.
	There are two more squares whose commutativity has to be proven.
	They can be treated similarly, so we omit their treatment.
	It is enough to show that the diagram
		\[
			\begin{CD}
					\alg{\Gamma}_{x}(\Order_{x}^{h} / k, [B, C]_{\Order_{x}^{h}})
				@>>>
					\alg{\Gamma}(U, [B, C]_{U})
				\\
				@VVV
				@VVV
				\\
					\bigl[
						\alg{\Gamma}(\Hat{\Order}_{x} / k, B),
						\alg{\Gamma}_{c}(U, V, C)
					\bigr]_{k}^{c}
				@>>>
					\bigl[
						\alg{\Gamma}_{c}(U, B),
						\alg{\Gamma}_{c}(U, V, C)
					\bigr]_{k}^{c}
			\end{CD}
		\]
	in $K(k^{\ind\rat}_{\pro\et})$ for $x \in U \setminus V$ and $B, C \in K(U_{\fppf})$ is commutative,
	where $[\var, \var]_{k}^{c}$ is the sheaf-Hom complex functor $\sheafhom_{k^{\ind\rat}_{\pro\et}}$.
	By the tensor-Hom adjunction, it is enough to show that the morphism
		\begin{equation} \label{eq: cup of local coh and cpt supp coh}
			\begin{aligned}
				&
							\alg{\Gamma}_{x}(\Order_{x}^{h} / k, A)
						\tensor
							\alg{\Gamma}_{c}(U, B)
					\to
							\alg{\Gamma}_{x}(\Order_{x}^{h} / k, A)
						\tensor
							\alg{\Gamma}(\Hat{\Order}_{x} / k, B)
				\\
				&	\to
						\alg{\Gamma}_{x}(\Hat{\Order}_{x} / k, A \tensor B)
					\to
						\alg{\Gamma}_{c}(U, V, A \tensor B)
			\end{aligned}
		\end{equation}
	and the morphism
		\[
					\alg{\Gamma}(U, A)
				\tensor
					\alg{\Gamma}_{c}(U, B)
			\to
				\alg{\Gamma}_{c}(U, A \tensor B)
			\to
				\alg{\Gamma}_{c}(U, V, A \tensor B)
		\]
	in $\Ch(k^{\ind\rat}_{\pro\et})$ are compatible up to homotopy via
		\[
				\alg{\Gamma}_{x}(\Order_{x}^{h} / k, A)
			\isomfrom
				\alg{\Gamma}_{x}(U, A)
			\to
				\alg{\Gamma}(U, A)
		\]
	(where $\isomfrom$ is a quasi-isomorphism).
	It is routine to check that the morphism \eqref{eq: cup of local coh and cpt supp coh},
	the morphism
		\begin{align*}
						\alg{\Gamma}_{x}(U, A)
					\tensor
						\alg{\Gamma}_{c}(U, B)
			&	\to
						\alg{\Gamma}_{x}(U, A)
					\tensor
						\alg{\Gamma}(U, B)
				\to
					\alg{\Gamma}_{x}(U, A \tensor B)
			\\
			&	\to
					\alg{\Gamma}_{x}(\Hat{\Order}_{x} / k, A \tensor B)
				\to
					\alg{\Gamma}_{c}(U, V, A \tensor B)
		\end{align*}
	and the morphism
		\begin{align*}
			&
						\alg{\Gamma}_{c}(U, A)
					\tensor
						\alg{\Gamma}_{c}(U, B)
				\to
						\alg{\Gamma}_{c}(U, A)
					\tensor
						\alg{\Gamma}(U, B)
			\\
			&	\to
					\alg{\Gamma}_{c}(U, A \tensor B)
				\to
					\alg{\Gamma}_{c}(U, V, A \tensor B)
		\end{align*}
	are all compatible (without homotopy) via
		\[
				\alg{\Gamma}_{x}(\Order_{x}^{h} / k, A)
			\isomfrom
				\alg{\Gamma}_{x}(U, A)
			\to
				\alg{\Gamma}_{c}(U, A).
		\]
	Hence we need to show that the diagram
		\[
			\begin{CD}
						\alg{\Gamma}_{c}(U, A)
					\tensor
						\alg{\Gamma}_{c}(U, B)
				@>>>
						\alg{\Gamma}(U, A)
					\tensor
						\alg{\Gamma}_{c}(U, B)
				\\
				@VVV
				@VVV
				\\
						\alg{\Gamma}_{c}(U, A)
					\tensor
						\alg{\Gamma}(U, B)
				@>>>
						\alg{\Gamma}_{c}(U, A \tensor B)
			\end{CD}
		\]
	in $\Ch(k^{\ind\rat}_{\pro\et})$ is commutative up to homotopy.
	The left upper term is the total complex of the three term complex in degrees $0, 1, 2$
	of double complexes
		\begin{align*}
						\alg{\Gamma}(U, A)
					\tensor
						\alg{\Gamma}(U, B)
			&	\to
						\alg{\Gamma}(U, A)
					\tensor
						\bigoplus_{x \not\in U}
							\alg{\Gamma}(\Hat{K}_{x} / k, B)
					\oplus
						\bigoplus_{x \not\in U}
							\alg{\Gamma}(\Hat{K}_{x} / k, A)
					\tensor
						\alg{\Gamma}(U, B)
			\\
			&	\to
						\bigoplus_{x \not\in U}
							\alg{\Gamma}(\Hat{K}_{x} / k, A)
					\tensor
						\bigoplus_{x \not\in U}
							\alg{\Gamma}(\Hat{K}_{x} / k, B).
		\end{align*}
	The right lower term is the total complex of the two term complex in degrees $0, 1$
	of double complexes
		\[
				\alg{\Gamma}(U, A \tensor B)
			\to
				\bigoplus_{x \not\in U}
					\alg{\Gamma}(\Hat{K}_{x} / k, A \tensor B).
		\]
	The required homotopy is given by the projection to the diagonal
		\[
						\bigoplus_{x \not\in U}
							\alg{\Gamma}(\Hat{K}_{x} / k, A)
					\tensor
						\bigoplus_{x \not\in U}
							\alg{\Gamma}(\Hat{K}_{x} / k, B)
				\to
						\bigoplus_{x \not\in U}
							\alg{\Gamma}(\Hat{K}_{x} / k, A \tensor B)
		\]
	in degree $2$ and zero in other degrees.
\end{proof}

\begin{Prop} \label{prop: compatibility of local global pairings after cohom approx}
	Let $A, B, C \in D(U_{\fppf})_{\ca}$.
	Let $A \to R \sheafhom_{U}(B, C)$, or equivalently,
	$A \tensor^{L} B \to C$, be a morphism in $D(U_{\fppf})$.
	Then the morphism in Prop.\ \ref{prop: compatibility of local global pairings without cohom approx}
	induces a morphism of distinguished triangles from
		\[
				\bigoplus_{x \in U \setminus V}
					R \alg{\Gamma}_{x}(\Hat{\Order}_{x} / k, A)
			\to
				R \alg{\Gamma}(U, A)
			\to
				R \alg{\Gamma}(V, A)
		\]
	to
		\begin{align*}
					\bigoplus_{x \in U \setminus V}
						\bigl[
							R \alg{\Gamma}(\Hat{\Order}_{x} / k, B),
							R \alg{\Gamma}_{c}(U, C)
						\bigr]_{k}
			&	\to
					\bigl[
						R \alg{\Gamma}_{c}(U, B),
						R \alg{\Gamma}_{c}(U, C)
					\bigr]_{k}
			\\
			&	\to
					\bigl[
						R \alg{\Gamma}_{c}(V, B),
						R \alg{\Gamma}_{c}(U, C)
					\bigr]_{k}.
		\end{align*}
\end{Prop}

\begin{proof}
	This follows from Prop.\ \ref{prop: cpt supp triangle after cohom approx}.
\end{proof}

We say that a sheaf $F \in \Set(U_{\fppf})$ is locally of finite presentation
if $F(\invlim_{\lambda} U_{\lambda}) = \dirlim_{\lambda} F(U_{\lambda})$
for any filtered inverse system $\{U_{\lambda}\}$
of quasi-compact quasi-separated $U$-schemes with affine transition morphisms.

\begin{Prop} \label{prop: cohom of curve is fin pres}
	For any sheaf $A \in \Ab(U_{\fppf})$ locally of finite presentation and $n \ge 0$,
	the sheaf $R^{n} \pi_{U \ast} A \in \Ab(k^{\ind\rat}_{\et})$
	is locally of finite presentation.
	In particular, we have
	$\alg{H}^{n}(U, A) = R^{n} \pi_{U \ast} A$.
	That is, $\alg{H}^{n}(U, A)$ is the \'etale sheafification of the presheaf
		\[
				k'
			\mapsto
				H^{n}(U_{k'}, A).
		\]
\end{Prop}

\begin{proof}
	The functor $k' \mapsto U_{k'}$ from the opposite category of $k^{\ind\rat}$ to the category of $U$-schemes
	commutes with filtered inverse limits.
	The same proof as Prop.\ \ref{prop: cohom of henselization is fin pres} works.
\end{proof}

In particular, in the situation of this proposition, we have isomorphisms
	\[
			R \Gamma(k_{\pro\et}, R \alg{\Gamma}(U, A))
		=
			R \Gamma(k_{\et}, R \pi_{U, \ast} A)
		=
			R \Gamma(U, A).
	\]
and a spectral sequence
	\[
			E_{2}^{i j}
		=
			H^{i}(k_{\pro\et}, \alg{H}^{j}(U, A))
		\Longrightarrow
			H^{i + j}(U, A).
	\]

\begin{Rmk} \label{rmk: henselian and completed versions}
	Prop.\ \ref{prop: cohom with closed support doesnt change under completion} may be false
	if $A$ is replaced by a general fppf sheaf over $\Order_{x}^{h}$.
	In fact, whenever $\Order_{x}^{h}$ is not complete,
	there exists an fppf sheaf $A$ over $\Order_{x}^{h}$ such that the map
		$
				H_{x}^{0}(\Order_{x}^{h}, A)
			\to
				H_{x}^{0}(\Hat{\Order}_{x}, A)
		$
	is not surjective,
	where $H_{x}^{0}(\Order_{x}^{h}, A)$ is the kernel of
	$\Gamma(\Order_{x}^{h}, A) \to \Gamma(K_{x}^{h}, A)$
	and $H_{x}^{0}(\Hat{\Order}_{x}, A)$ is defined similarly.
	This shows that the distinguished triangle in
	Prop.\ \ref{prop: cpt supp triangle after cohom approx}
	and a similar long exact sequence stated in \cite[III, Rmk.\ 0.6 (b)]{Mil06}
	do not exist for general sheaves.
	Coefficients in smooth group schemes and finite flat group schemes as stated in
	Prop.\ \ref{prop: smooths and finite flats satisfy cohom approx}
	are sufficient in this paper and in \cite[III]{Mil06}.
	
	To give an example of such an fppf sheaf,
	consider the two ring homomorphisms
	$\Order_{x}^{h} \into \Order_{x}^{h}[t] \onto k_{x}$,
	where the first one is the inclusion and
	the second is the map $t \mapsto 0$ followed by the reduction map.
	Denote the morphisms
	$\Spec k_{x, \fppf} \to \Spec \Order_{x}^{h}[t]_{\fppf} \to \Spec \Order_{x, \fppf}^{h}$
	induced on the fppf sites by $i$ and $f$, respectively.
	Then a desired counterexample is given by $A = f_{!} i_{\ast} \Z$,
	where $f_{!}$ is the left adjoint of $f^{\ast}$ (\cite[II, Rmk.\ 3.18]{Mil80}).
	
	Indeed, let $f_{!}^{\mathrm{pre}}$ be the left adjoint of the presheaf pullback functor by $f$ and
	set $A^{\mathrm{pre}} = f_{!}^{\mathrm{pre}} i_{\ast} \Z$.
	Denote $\Hat{\Order}_{x} \otimes_{\Order_{x}^{h}} \Hat{\Order}_{x}$
	by $\Hat{\Order}_{x}^{\tensor 2}$.
	Then $\Gamma(\Hat{\Order}_{x}, A^{\mathrm{pre}})$ and
	$\Gamma(\Hat{\Order}_{x}^{\tensor 2}, A^{\mathrm{pre}})$ are
	the free abelian groups generated by the sets
	$\Hat{\ideal{p}}_{x} = \ideal{p}_{x}^{h} \Hat{\Order}_{x}$ and
	$\ideal{p}_{x}^{h} \Hat{\Order}_{x}^{\tensor 2}$, respectively.
	Consider the element of $\Gamma(\Hat{\Order}_{x}, A^{\mathrm{pre}})$
	corresponding to any element $c$ of $\Hat{\ideal{p}}_{x}$ not in $\Order_{x}^{h}$.
	Since $c \otimes 1$ and $1 \otimes c$ are distinct in $\Hat{\Order}_{x}^{\tensor 2}$,
	these elements do not glue in $\Gamma(\Hat{\Order}_{x}^{\tensor 2}, A^{\mathrm{pre}})$.
	They do not even after replacing $\Hat{\Order}_{x}^{\tensor 2}$ by its fppf cover
	since $f_{!}^{\mathrm{pre}}$ sends separated presheaves to separated presheaves
	by construction.
	Hence the corresponding element of
	$\Gamma(\Hat{\Order}_{x}, A) = H_{x}^{0}(\Hat{\Order}_{x}, A)$
	does not come from an element of $\Gamma(\Order_{x}^{h}, A) = H_{x}^{0}(\Order_{x}^{h}, A)$.
\end{Rmk}


\section{Global duality and its proof}
\label{sec: Global duality and its proof}

In the rest of this paper,
let $X$ be a proper smooth geometrically connected curve
over a perfect field $k$ ($\in \mathcal{U}_{0}$) of characteristic $p > 0$ with function field $K$.
We continue the notation in
\S \ref{sec: The fppf site of a curve over the rational etale site of the base}.
We denote $\Ab(k^{\ind\rat}_{\pro\et})$ by $\Ab(k)$
and $D(k^{\ind\rat}_{\pro\et})$ by $D(k)$.
We denote $\sheafhom_{k^{\ind\rat}_{\pro\et}}$ by $\sheafhom_{k}$
and use the notation $\sheafext_{k}^{n}$ and $R \sheafhom_{k}$ similarly.
Exact sequences and distinguished triangles of objects over $k$
are always considered in $\Ab(k)$ or $D(k)$ unless otherwise noted.

Let $A$ be an abelian variety over $K$ with N\'eron model $\mathcal{A}$ over $X$.
(Actually \S \ref{sec: Duality for finite flat coefficients over open curves}
and \ref{sec: Formal steps towards duality for Neron models} do not use $A$.)
The maximal open subgroup scheme of $\mathcal{A}$ with connected fibers is denoted by $\mathcal{A}_{0}$.
The fiber of $\mathcal{A}$ over any closed point $x = \Spec k_{x}$ of $X$
is denoted by $\mathcal{A}_{x}$.
The dual of $A$ is denoted by $A^{\vee}$,
with N\'eron model $\mathcal{A}^{\vee}$ and the corresponding subscheme $\mathcal{A}_{0}^{\vee}$
and the fibers $\mathcal{A}_{x}^{\vee}$.

\subsection{Duality for finite flat coefficients over open curves}
\label{sec: Duality for finite flat coefficients over open curves}

By Prop.\ \ref{prop: cohom of curve is fin pres},
the sheaf $\alg{H}^{n}(X, \Gm) \in \Ab(k^{\ind\rat}_{\pro\et})$ for any $n$ is
locally of finite presentation and
the \'etale sheafification of the presheaf
	\[
			k'
		\mapsto
			H^{n}(X_{k'}, \Gm).
	\]
Let $\Pic_{X} = \Pic_{X / k}$ be the perfection of the Picard scheme of $X$ over $k$.
This represents the sheaf $R^{1} \pi_{X \ast}^{\perf} \Gm$ on $\Spec k^{\perf}_{\et}$
and hence the sheaf $\alg{H}^{1}(X, \Gm)$ on $\Spec k^{\ind\rat}_{\pro\et}$.

\begin{Prop}
	We have
	$\alg{\Gamma}(X, \Gm) = \Gm$,
	$\alg{H}^{1}(X, \Gm) = \Pic_{X}$
	and $\alg{H}^{n}(X, \Gm) = 0$ for $n \ge 2$.
\end{Prop}

\begin{proof}
	The statements for $n = 0, 1$ are obvious.
	For $n \ge 2$, it is a classical fact \cite[III, Example 2.22, Case (d)]{Mil80} that
	$H^{n}(X_{k'}, \Gm) = 0$ for any algebraically closed fields $k'$ over $k$.
	This implies $\alg{H}^{n}(X, \Gm) = 0$
	since this sheaf is locally of finite presentation.
\end{proof}

Let $U \subset X$ be a dense open subscheme.
The smooth group scheme $\Gm \in \Ab(U_{\fppf})$ satisfies cohomological approximation
by Prop.\ \ref{prop: smooths and finite flats satisfy cohom approx}.
Hence by Prop.\ \ref{prop: cpt supp triangle after cohom approx}
and the above proposition, we have morphisms
	\begin{equation} \label{eq: global trace morphism}
			R \alg{\Gamma}_{c}(U, \Gm)
		\to
			R \alg{\Gamma}(X, \Gm)
		\to
			\Pic_{X}[-1]
		\overset{\deg}{\to}
			\Z[-1].
	\end{equation}
We call the composite the (global) \emph{trace morphism}.

\begin{Prop}
	The global trace morphism
	$R \alg{\Gamma}_{c}(U, \Gm) \to \Z[-1]$
	and the local trace morphism
	$R \alg{\Gamma}_{x}(\Hat{\Order}_{x} / k, \Gm) \to \Z[-1]$
	at $x \in U$ in \eqref{eq: local trace morphism, Weil restricted} are compatible under the morphism
	$R \alg{\Gamma}_{x}(\Hat{\Order}_{x} / k, \Gm) \to R \alg{\Gamma}_{c}(U, \Gm)$
	in Prop.\ \ref{prop: cpt supp triangle after cohom approx}.
\end{Prop}

\begin{proof}
	We are comparing the degree morphism
	$\Pic_{X} \to \Z$ and the valuation morphism
	$\Res_{k_{x} / k} \Hat{\alg{K}}_{x}^{\times} / \Hat{\alg{O}}_{x}^{\times} \to \Z$.
	It is enough to compare them on $\closure{k}$-points
	since $\Hat{\alg{K}}_{x}^{\times} / \Hat{\alg{O}}_{x}^{\times} \cong \Z$ is \'etale.
	Then the comparison is between the abstract group homomorphisms
	$\Pic(X_{\closure{k}}) \to \Z$ and
		$
				\bigoplus_{\closure{x} \to x}
					\Hat{K}_{\closure{x}}^{\times} / \Hat{\Order}_{\closure{x}}^{\times}
			\to
				\Z
		$,
	where the sum is over all $\closure{k}$-points of $X_{\closure{k}}$ lying over $x$.
	This is obvious.
\end{proof}

Tensoring $\Z / p \Z[-1]$, the global trace morphism induces a morphism
	\[
			R \alg{\Gamma}(X, \mu_{p})
		\to
			\Z / p \Z[-2].
	\]
Prop.\ \ref{prop: comparison with Artin Milne} and
\ref{prop: cohom of curve is fin pres}
allow us to translate the results of \cite{AM76} to our setting.
The kernel of the morphism
	\[
				\alg{H}^{1}(X, \Omega_{X}^{1})
			\overset{C - 1}{\to}
				\alg{H}^{1}(X, \Omega_{X}^{1}),
		\quad \text{or} \quad
				\Ga
			\overset{F - 1}{\to}
				\Ga
	\]
identifies $\alg{H}^{2}(X, \mu_{p})$ as $\Z / p \Z$
as explained in \cite[Introduction]{AM76},
where $C$ is the Cartier operator and $F$ is the Frobenius.
This gives a morphism $R \alg{\Gamma}(X, \mu_{p}) \to \Z / p \Z$.
This is equal to the above trace morphism since we have a commutative diagram
	\[
		\begin{CD}
				\alg{H}^{1}(X, \Gm)
			@>> p >
				\alg{H}^{1}(X, \Gm)
			@>> \dlog >
				\alg{H}^{1}(X, \Omega_{X}^{1})
			@>> C - 1 >
				\alg{H}^{1}(X, \Omega_{X}^{1})
			\\
			@VV \deg V
			@VV \deg V
			@VV \Res V
			@V \Res VV
			\\
				\Z
			@> p >>
				\Z
			@> \text{can} >>
				\Ga
			@> F^{-1} - 1 >>
				\Ga,
		\end{CD}
	\]
where $\Res$ denotes the residue map.

We need the following result of Milne \cite[III, Thm.\ 11.1]{Mil06},
which is the generalization for open curves
of the corresponding result of Artin-Milne \cite{AM76}.
Since \cite[\S 5.2]{Suz14} replaces Bester's local finite flat duality,
we give a proof based on \cite[Thm.\ (5.2.1.2)]{Suz14} for clarity.

\begin{Thm} \label{thm: duality for finite flat over open curve}
	Let $U \subset X$ be a dense open subscheme
	and $N$ a finite flat group scheme over $U$.
	Then $R \alg{\Gamma}(U, N) \in D^{b}(\Ind \Alg_{\uc})$ and
	$R \alg{\Gamma}_{c}(U, N) \in D^{b}(\Pro \Alg_{\uc})$,
	which are both concentrated in degrees $0, 1, 2$.
	Consider the morphism
		\[
				R \alg{\Gamma}(U, N^{\CDual}) \tensor^{L} R \alg{\Gamma}_{c}(U, N)
			\to
				R \alg{\Gamma}_{c}(U, \Gm)
			\to
				\Z[-1]
		\]
	given by the cup product morphism \eqref{eq: global cup product}
	and the trace morphism \eqref{eq: global trace morphism}.
	The induced morphism
		\[
				R \alg{\Gamma}(U, N^{\CDual})
			\to
				R \alg{\Gamma}_{c}(U, N)^{\SDual}[-1]
		\]
	in $D(k)$ is an isomorphism.
\end{Thm}

\begin{proof}
	The proof proceeds by d\'evissage.
	
	Step 1: If $U = X$, then the theorem in this case is \cite[(0.3), (4.9)]{AM76}.
	
	Step 2: Let $V \subset U$ be a dense open subscheme.
	The theorem is true for $N$ over $U$ if and only if so is for $N$ over $V$.
	Indeed, the groups $N$, $N^{\CDual}$ and $\Gm$ satisfy cohomological approximation by
	Prop.\ \ref{prop: smooths and finite flats satisfy cohom approx}.
	The natural morphism $N^{\CDual} \to R \sheafhom_{U}(N, \Gm)$,
	the trace morphism $R \alg{\Gamma}_{c}(U, \Gm) \to \Z[-1]$
	and the morphism in Prop.\ \ref{prop: compatibility of local global pairings after cohom approx}
	give a morphism of distinguished triangles
		\[
			\begin{CD}
					\displaystyle{\bigoplus_{x \in U \setminus V}}
						R \alg{\Gamma}_{x}(\Hat{\Order}_{x} / k, N^{\CDual})
				@>>>
					R \alg{\Gamma}(U, N^{\CDual})
				@>>>
					R \alg{\Gamma}(V, N^{\CDual})
				\\
				@VVV
				@VVV
				@VVV
				\\
					\displaystyle{\bigoplus_{x \in U \setminus V}}
						R \alg{\Gamma}(\Hat{\Order}_{x} / k, N)^{\SDual}[-1]
				@>>>
					R \alg{\Gamma}_{c}(U, N)^{\SDual}[-1]
				@>>>
					R \alg{\Gamma}_{c}(V, N)^{\SDual}[-1].
			\end{CD}
		\]
	We know that $R \alg{\Gamma}_{x}(\Hat{\Order}_{x} / k, N^{\CDual})$ is in $D^{b}(\Ind \Alg_{\uc} / k)$
	concentrated in degree $2$ and
	and $R \alg{\Gamma}(\Hat{\Order}_{x} / k, N)$ is in $D^{b}(\Pro \Alg_{\uc} / k)$
	concentrated in degrees $0, 1$
	by \cite[Prop.\ (3.4.2) (b), Prop.\ (3.4.6)]{Suz14}.
	The left vertical morphism is an isomorphism by Bester's duality (\cite[Thm.\ (5.2.1.2)]{Suz14}).
	Hence the theorem for $N$ over $U$ and the theorem for $N$ over $V$ are equivalent.
	
	Step 3: If $0 \to N' \to N \to N'' \to 0$ is an exact sequence of finite flat group schemes over $U$
	and the theorem is true for $N'$ and $N''$, then so is for $N$.
	
	Step 4: The general case.
	The group $N_{K} = N \times_{X} K$ over $K$ has a filtration by finite flat subgroup schemes
	whose each successive subquotient or its dual is of height one.
	A finite flat group scheme over $K$ of height one or with dual of height one
	extends to $X$ as a finite flat group scheme by \cite[III, Prop.\ B.4, B.5]{Mil06}.
	By spreading out, we know that $N_{V} = N \times_{X} V$ over some dense open $V \subset U$
	has a filtration by finite flat subgroup schemes
	whose successive subquotients are finite flat and extendable to $X$ as finite flat group schemes.
	Hence the previous three steps imply the theorem.
\end{proof}

For how the above theorem is related to the corresponding duality statement \cite[III, Thm. 8.2]{Mil06}
in the finite base field case,
see Rmk.\ \ref{rmk: duality for finite flat over open curve over finite field} below.


\subsection{Mod $n$ duality for N\'eron models and preliminary calculations}
\label{sec: Mod n duality for Neron models and preliminary calculations}

Let $A$ be as in the beginning of this section.
The Poincar\'e biextension $A^{\vee} \tensor^{L} A \to \Gm[1]$
as a morphism in $D(K_{\fppf})$ canonically extends to
a biextension $\mathcal{A}_{0}^{\vee} \tensor^{L} \mathcal{A} \to \Gm[1]$
as a morphism in $D(X_{\fppf})$
by \cite[IX, 1.4.3]{Gro72}.
With this morphism, the cup product morphism \eqref{eq: global cup product}
and the trace morphism \eqref{eq: global trace morphism},
we have morphisms
	\begin{equation} \label{eq: duality pairing for Neron over X}
				R \alg{\Gamma}(X, \mathcal{A}_{0}^{\vee})
			\tensor^{L}
				R \alg{\Gamma}(X, \mathcal{A})
		\to
			R \alg{\Gamma}(X, \Gm[1])
		\to
			\Z.
	\end{equation}

\begin{Prop} \label{prop: mod n duality for Neron models}
	Let $n \ge 1$.
	Then $R \alg{\Gamma}(X, \mathcal{A}) \tensor^{L} \Z / n \Z$ and
	$R \alg{\Gamma}(X, \mathcal{A}_{0}) \tensor^{L} \Z / n \Z$
	are both in $D^{b}(\Alg_{\uc})$
	concentrated in degrees $-1, 0, 1$.
	Consider the morphism \eqref{eq: duality pairing for Neron over X}
	derived-tensored with $\Z / n \Z$:
		\[
					\bigl(
						R \alg{\Gamma}(X, \mathcal{A}_{0}^{\vee}) \tensor^{L} \Z / n \Z
					\bigr)
				\tensor^{L}
					\bigl(
						R \alg{\Gamma}(X, \mathcal{A}) \tensor^{L} \Z / n \Z
					\bigr)
			\to
				\Z / n \Z
			\to
				\Z[1],
		\]
	where the last morphism is the connecting morphism for the short exact sequence
	$0 \to \Z \to \Z \to \Z / n \Z \to 0$.
	The induced morphism
		\[
				R \alg{\Gamma}(X, \mathcal{A}_{0}^{\vee}) \tensor^{L} \Z / n \Z
			\to
				\bigl(
					R \alg{\Gamma}(X, \mathcal{A}) \tensor^{L} \Z / n \Z
				\bigr)^{\SDual}[1]
		\]
	in $D(k)$ is an isomorphism.
\end{Prop}

\begin{proof}
	We denote $(\var) \tensor^{L} \Z / n \Z$ by $(\var)_{n}$ to simplify the notation.
	Take a dense open subscheme $U \subset X$ over which $\mathcal{A}$ is an abelian scheme.
	We have a morphism of distinguished triangles
		\[
			\begin{CD}
					\displaystyle
					\bigoplus_{x \not\in U}
						R \alg{\Gamma}_{x}(\Hat{\Order}_{x} / k, \mathcal{A}_{0}^{\vee})_{n}
				@>>>
					R \alg{\Gamma}(X, \mathcal{A}_{0}^{\vee})_{n}
				@>>>
					R \alg{\Gamma}(U, \mathcal{A}_{0}^{\vee})_{n}
				\\
				@VVV
				@VVV
				@VVV
				\\
					\displaystyle
					\bigoplus_{x \not\in U}
						\bigl(
							R \alg{\Gamma}(\Hat{\Order}_{x} / k, \mathcal{A})_{n}
						\bigr)^{\SDual}[1]
				@>>>
					\bigl(
						R \alg{\Gamma}(X, \mathcal{A})_{n}
					\bigr)^{\SDual}[1]
				@>>>
					\bigl(
						R \alg{\Gamma}_{c}(U, \mathcal{A})_{n}
					\bigr)^{\SDual}[1]
			\end{CD}
		\]
	by the same method as the proof of the previous theorem.
	The right vertical morphism is an isomorphism by the previous proposition since
	$(\mathcal{A}^{\vee} \tensor^{L} \Z / n \Z)[-1]$ over $U$ is the $n$-torsion part of $\mathcal{A}$,
	which is finite flat over $U$.
	The left vertical morphism is an isomorphism
	by Prop.\ \ref{prop: local duality}.
	Hence so is the middle vertical morphism.
	The object $R \alg{\Gamma}_{x}(\Hat{\Order}_{x} / k, \mathcal{A}_{0}^{\vee})_{n}$
	is in $D^{b}(\Ind \Alg_{\uc} / k)$ concentrated in degrees $0, 1$
	and the object $R \alg{\Gamma}(\Hat{\Order}_{x} / k, \mathcal{A})_{n}$
	is in $D^{b}(\Pro \Alg_{\uc} / k)$ concentrated in degrees $-1, 0$
	by Prop.\ \ref{prop: structure of local cohom of abelian varieties}.
	The objects $R \alg{\Gamma}(U, \mathcal{A}_{0}^{\vee})_{n}$
	and $R \alg{\Gamma}_{c}(U, \mathcal{A})_{n}$ are in $D^{b}(\Ind \Alg_{\uc} / k)$
	and in $D^{b}(\Pro \Alg_{\uc} / k)$, respectively, and concentrated in degrees $-1, 0, 1$
	by the previous proposition.
	Hence the same is true for $R \alg{\Gamma}(X, \mathcal{A}_{0}^{\vee})_{n}$
	and $R \alg{\Gamma}(X, \mathcal{A})_{n}$.
	Therefore they are in both $D^{b}(\Pro \Alg_{\uc} / k)$ and $D^{b}(\Ind \Alg_{\uc} / k)$,
	hence in $D^{b}(\Alg_{\uc} / k)$.
\end{proof}

In the next proposition, we consider not necessarily perfect group schemes over $k$.
But we will soon apply perfection.

\begin{Prop}
	Let $\Res_{X / k} \mathcal{A}$ be the Weil restriction of $\mathcal{A}$
	as a functor on the category of (not necessarily perfect) $k$-schemes.
	Let $\Tr_{K / k} A$ be the $K / k$-trace of $A$
	(\cite[VIII, Thm.\ 8]{Lan83}, \cite[Def.\ 6.1]{Con09}),
	which is an abelian variety over $k$
	with canonical $K$-morphism $(\Tr_{K / k} A)_{K} \to A$.
	Let $\Tr_{K / k} A \to \Res_{X / k} \mathcal{A}$ be the $k$-morphism
	induced by the $X$-morphism $(\Tr_{K / k} A)_{X} \to \mathcal{A}$,
	which itself is induced by $(\Tr_{K / k} A)_{K} \to A$.
	
	Then $\Res_{X / k} \mathcal{A}$ is represented by a group scheme locally of finite type over $k$.
	The morphism $\Tr_{K / k} A \to (\Res_{X / k} \mathcal{A})_{0}$
	to the identity component has finite infinitesimal kernel and cokernel.
	The \'etale $k$-group $\pi_{0} \Res_{X / k} \mathcal{A}$ of connected components has
	finitely generated group of geometric points.
\end{Prop}

\begin{proof}
	Since $X$ is proper over $k$ and $\mathcal{A}$ is quasi-projective over $X$ by \cite[6.4/1]{BLR90},
	the result on existence of Hilbert schemes \cite[\S 4.c]{Gro95} shows that
	$\Res_{X / k} \mathcal{A}$ is a group scheme locally of finite type over $k$.
	The kernel of $(\Tr_{K / k} A)_{K} \to A$ is a finite infinitesimal $K$-group
	by \cite[Thm.\ 6.12]{Con09}.
	It follows that $\Tr_{K / k} A \to \Res_{K / k} \mathcal{A}$ is injective on geometric points.
	Hence its kernel is a finite infinitesimal $k$-group.
	The group of $\closure{k}$-points of the cokernel of $\Tr_{K / k} A \to \Res_{K / k} \mathcal{A}$
	is $\mathcal{A}(X_{\closure{k}}) / (\Tr_{K / k} A)(\closure{k})$,
	which is finitely generated by the Lang-N\'eron theorem (\cite[Thm.\ 7.1]{Con09}).
	This implies that the connected group scheme
	$(\Res_{X / k} \mathcal{A})_{0} / \Tr_{K / k} A$ of finite type over $k$
	has finitely generated group of geometric points.
	It follows that $(\Res_{X / k} \mathcal{A})_{0} / \Tr_{K / k} A$ has trivial reduced part
	and hence is finite infinitesimal.
	Hence $\pi_{0} \Res_{X / k} \mathcal{A}$ has
	finitely generated group of geometric points.
\end{proof}

\begin{Prop} \label{prop: structure of cohomology of Neron models}
	The sheaf $\alg{\Gamma}(X, \mathcal{A})$ is an extension
	of a finitely generated \'etale group
	by an abelian variety over $k$.
	For $n \ge 1$, we have $\alg{H}^{n}(X, \mathcal{A}) \in \Ind^{f} \Alg_{\uc} / k$.
	The group $\alg{H}^{2}(X, \mathcal{A})$ is divisible.
	For $n \ge 3$, we have $\alg{H}^{n}(X, \mathcal{A}) = 0$.
\end{Prop}

\begin{proof}
	The statement about $\alg{\Gamma}(X, \mathcal{A})$ follows from the previous proposition
	since it has the same values as $\Res_{X / k} \mathcal{A}$ on ind-rational $k$-algebras
	and hence is represented by the perfection of $\Res_{X / k} \mathcal{A}$.
	
	We show that the sheaf $\alg{H}^{n}(X, \mathcal{A})$ for any $n \ge 1$ is torsion.
	By Prop.\ \ref{prop: cohom of curve is fin pres},
	it is an \'etale sheafification of the presheaf
	$k' \mapsto H^{n}(X_{k'}, \mathcal{A})$ and
	locally of finite presentation.
	Hence it is enough to show that the abstract group $H^{n}(X_{k'}, \mathcal{A})$ for $n \ge 1$ is torsion
	for any algebraically closed field $k'$ over $k$.
	Assume that $k = \closure{k}$.
	For any dense open $U \subset X$, by applying $R \Gamma(k, \var)$
	to the isomorphism in Prop.\ \ref{prop: excision},
	we have a distinguished triangle
		\[
				\bigoplus_{x \not\in U}
					R \Gamma_{x}(\Hat{\Order}_{x}, \mathcal{A})
			\to
				R \Gamma(X, \mathcal{A})
			\to
				R \Gamma(U, \mathcal{A}).
		\]
	Taking the limit in smaller and smaller $U$,
	we have a distinguished triangle
		\[
				\bigoplus_{x \in X}
					R \Gamma_{x}(\Hat{\Order}_{x}, \mathcal{A})
			\to
				R \Gamma(X, \mathcal{A})
			\to
				R \Gamma(K, A).
		\]
	We have $R \Gamma_{x}(\Hat{\Order}_{x}, \mathcal{A}) = H^{1}(\Hat{K}_{x}, A)[-2]$
	by Prop.\ \ref{prop: structure of local cohom of abelian varieties}.
	Since Galois cohomology groups are torsion in positive degrees,
	it follows that $H^{n}(X, \mathcal{A})$ is torsion for any $n \ge 1$.
	The same is true if $k$ is replaced by any algebraically closed field over $k$.
	Hence $\alg{H}^{n}(X, \mathcal{A})$ for $n \ge 1$ is torsion.
	
	Let $m \ge 1$.
	Denote $C_{m} = R \alg{\Gamma}(X, \mathcal{A}) \tensor^{L} \Z / m \Z$.
	For any $n \in \Z$, we have an exact sequence
		\[
				0
			\to
				\alg{H}^{n - 1}(X, \mathcal{A}) \tensor \Z / m \Z
			\to
				H^{n - 1}(C_{m})
			\to
				\alg{H}^{n}(X, \mathcal{A})[m]
			\to
				0,
		\]
	where $[m]$ denotes the $m$-torsion part.
	We have $H^{n - 1}(C_{m}) \in \Alg_{\uc} / k$
	by Prop.\ \ref{prop: mod n duality for Neron models}.
	We have $\alg{\Gamma}(X, \mathcal{A}) \tensor \Z / m \Z \in \FEt / k$
	by what we have shown above about the structure of $\alg{\Gamma}(X, \mathcal{A})$.
	Therefore $\alg{H}^{1}(X, \mathcal{A})[m] \in \Alg_{\uc} / k$.
	Since $m$ is arbitrary, the torsionness shown above then shows that
	$\alg{H}^{1}(X, \mathcal{A}) \in \Ind^{f} \Alg_{\uc} / k$.
	Hence $\alg{H}^{1}(X, \mathcal{A}) \tensor \Z / m \Z \in \Alg_{\uc} / k$
	by Prop.\ \ref{prop: cofinite type mod n is finite type}.
	Repeating the same argument,
	we have $\alg{H}^{2}(X, \mathcal{A}) \in \Ind^{f} \Alg_{\uc} / k$.
	Since $H^{n - 1}(C_{m}) = 0$ for $n \ge 3$ by Prop.\ \ref{prop: mod n duality for Neron models},
	we know that $\alg{H}^{2}(X, \mathcal{A})$ is divisible and
	$\alg{H}^{n}(X, \mathcal{A}) = 0$ for $n \ge 3$.
\end{proof}

For each closed point $x \in X$,
we regard the component group $\pi_{0}(\mathcal{A}_{x})$ to be an \'etale group over $k_{x}$.
Let $i_{x} \colon x \into X$ be the inclusion, where we identified $x = \Spec k_{x}$.
We have an exact sequence
	\[
			0
		\to
			\mathcal{A}_{0}
		\to
			\mathcal{A}
		\to
			\bigoplus_{x}
				i_{x \ast} \pi_{0}(\mathcal{A}_{x})
		\to
			0
	\]
in $\Ab(X_{\fppf})$.
The sheaf $i_{x \ast} \pi_{0}(\mathcal{A}_{x})$ for any $x$ is an \'etale scheme over $X$
if the \'etale group $\pi_{0}(\mathcal{A}_{x})$ over $k_{x}$ is constant
and an \'etale algebraic space over $X$ in general
by \cite[Prop. (3.3.6.1)]{Ray70a}.

\begin{Prop} \label{prop: cohomology of connected Neron and Neron}
	The above sequence induces a distinguished triangle
		\[
				R \alg{\Gamma}(X, \mathcal{A}_{0})
			\to
				R \alg{\Gamma}(X, \mathcal{A})
			\to
				\bigoplus_{x}
					\Res_{k_{x} / k} \pi_{0}(\mathcal{A}_{x}),
		\]
	where $\Res_{k_{x} / k}$ denotes the Weil restriction functor.
	In particular, we have an exact sequence
		\begin{align*}
					0
			&	\to
					\alg{\Gamma}(X, \mathcal{A}_{0})
				\to
					\alg{\Gamma}(X, \mathcal{A})
				\to
					\bigoplus_{x}
						\Res_{k_{x} / k} \pi_{0}(\mathcal{A}_{x})
			\\
			&	\to
					\alg{H}^{1}(X, \mathcal{A}_{0})
				\to
					\alg{H}^{1}(X, \mathcal{A})
				\to
					0
		\end{align*}
	and an isomorphism
		$
				\alg{H}^{2}(X, \mathcal{A}_{0})
			=
				\alg{H}^{2}(X, \mathcal{A})
		$.
	The group $\bigoplus_{x} \Res_{k_{x} / k} \pi_{0}(\mathcal{A}_{x})$ is finite \'etale over $k$.
\end{Prop}

\begin{proof}
	Obvious.
\end{proof}

Let $\mathcal{G}_{m}$ be the N\'eron (lft) model of $\Gm$ over $X$ (\cite[10.1/5]{BLR90}).
It fits in the canonical exact sequence
	\[
		0 \to \Gm \to \mathcal{G}_{m} \to \bigoplus_{x} i_{x \ast} \Z \to 0
	\]
of smooth group schemes over $X$,
where the sum is over all closed points $x \in X$.
At each $x$, we have Grothendieck's pairing \cite[IX, 1.2.1]{Gro72}
	\[
				\pi_{0}(\mathcal{A}_{x}^{\vee})
			\times
				\pi_{0}(\mathcal{A}_{x})
		\to
			\Q / \Z
	\]
over $k_{x}$.
Combining these two, in $D(X_{\fppf})$, we have morphisms
	\[
				\left(
					\bigoplus_{x}
						i_{x \ast}
						\pi_{0}(\mathcal{A}_{x}^{\vee})
				\right)
			\tensor^{L}
				\left(
					\bigoplus_{x}
						i_{x \ast}
						\pi_{0}(\mathcal{A}_{x})
				\right)
		\to
			\bigoplus_{x}
				i_{x \ast} \Q / \Z
		\to
			\bigoplus_{x}
				i_{x \ast} \Z[1]
		\to
			\Gm[2].
	\]
Denote $\Phi_{A, X} = \bigoplus_{x} i_{x \ast} \pi_{0}(\mathcal{A}_{x})$
and $\Phi_{A^{\vee}, X}$ similarly.
The above defines a morphism
	\[
			\Phi_{A^{\vee}, X} \tensor^{L} \Phi_{A, X}
		\to
			\Gm[2].
	\]
Recall again that we have morphisms
	\[
			\mathcal{A}_{0}^{\vee} \tensor^{L} \mathcal{A}
		\to
			\Gm[1],
		\quad
			\mathcal{A}^{\vee} \tensor^{L} \mathcal{A}_{0}
		\to
			\Gm[1]
	\]
defined by the canonical extensions of the Poincar\'e biextension.

\begin{Prop} \label{prop: Poincare pairing and Grothendieck pairing}
	The above morphisms
		\[
			\begin{CD}
					\mathcal{A}_{0}^{\vee}
				@>>>
					\mathcal{A}^{\vee}
				@>>>
					\Phi_{A^{\vee}, X}
				\\
				@VVV
				@VVV
				@VVV
				\\
					R \sheafhom_{X}(\mathcal{A}, \Gm)[1]
				@>>>
					R \sheafhom_{X}(\mathcal{A}_{0}, \Gm)[1]
				@>>>
					R \sheafhom_{X}(
						\Phi_{A, X},
						\Gm
					)[2]
			\end{CD}
		\]
	form a morphism of distinguished triangles in $D(X_{\fppf})$,
	where the horizontal triangles are the natural ones.
\end{Prop}

To prove this, we need some notation and three lemmas.
Let $X_{\sm}$ be the smooth site of $X$,
i.e.\ the category of smooth schemes over $X$ with $X$-scheme morphisms
endowed with the \'etale (or smooth) topology.
Denote the sheaf-Hom functor for $X_{\sm}$ by $\sheafhom_{X_{\sm}}$
and $\sheafext_{X_{\sm}}^{n}$, $R \sheafhom_{X_{\sm}}$ similarly
(while $\sheafhom_{X}$ is still the sheaf-Hom for $X_{\fppf}$).

\begin{Lem}
	To prove Prop.\ \ref{prop: Poincare pairing and Grothendieck pairing},
	it is enough to show the modified statement in $D(X_{\sm})$
	where $R \sheafhom_{X}$ is replaced by $R \sheafhom_{X_{\sm}}$.
\end{Lem}

\begin{proof}
	Let $f \colon X_{\fppf} \to X_{\sm}$ be the premorphism of sites defined by the identity functor.
	By \cite[Lem.\ 3.7.2]{Suz13}, the pullback functor
	$f^{\ast} \colon \Ab(X_{\sm}) \to \Ab(X_{\fppf})$ admits a left derived functor
	$L f^{\ast} \colon D(X_{\sm}) \to D(X_{\fppf})$,
	which is left adjoint to $R f_{\ast} \colon D(X_{\fppf}) \to D(X_{\sm})$
	and satisfies $L_{n} f^{\ast} \Z[Y] = 0$ for any smooth $X$-scheme $Y$ and $n \ge 1$
	(where $\Z[Y]$ is the sheaf of free abelian groups
	generated by the representable sheaf of sets $Y$).
	By \cite[Prop.\ 4.2]{Suz18}, we have
	$L f^{\ast} G = G$ for any smooth group algebraic space $G$ over $X$.
	Also $R f_{\ast} G = G$ since the fppf cohomology with coefficients in a smooth group algebraic space
	agrees with the \'etale cohomology \cite[III, Rmk.\ 3.11 (b)]{Mil80}.
	Therefore if $H$ is another smooth group algebraic space over $X$, then
		\begin{align*}
			&
					R f_{\ast}
					R \sheafhom_{X}(G, H)
				=
					R f_{\ast}
					R \sheafhom_{X}(L f^{\ast} G, H)
			\\
			&	=
					R \sheafhom_{X_{\sm}}(G, R f_{\ast} H)
				=
					R \sheafhom_{X_{\sm}}(G, H)
		\end{align*}
	by the derived tensor-Hom adjunction \cite[Prop.\ 3.1 (1)]{Suz18},
	which is applicable to our situation
	since the category of smooth schemes over $X$ has finite products.
	Applying $R f_{\ast}$ to the diagram in the statement,
	we know that the modified statement implies the original statement.
\end{proof}

\begin{Lem}
	Let $x \in X$ be a closed point
	and $N$ a finite \'etale group over $k_{x}$ with Pontryagin dual $N^{\PDual}$.
	Let $i_{x} \colon x_{\sm} \to X_{\sm}$ be the premorphism of sites
	defined by the inclusion $i_{x} \colon x \into X$.
	Then the truncation $\tau_{\le 2}$ of $R \sheafhom_{X_{\sm}}(i_{x \ast} N, \Gm)$ in degrees $\le 2$
	is canonically isomorphic to $i_{x \ast} N^{\PDual}[-2]$.
\end{Lem}

\begin{proof}
	Let $j_{x} \colon U = X \setminus \{x\} \into X$
	and denote its extension-by-zero functor by $j_{x !} \colon D(U_{\sm}) \to D(X_{\sm})$.
	We have a distinguished triangle
		\[
				R \sheafhom_{X_{\sm}}(i_{x \ast} N, \Gm)
			\to
				R \sheafhom_{X_{\sm}}(N, \Gm)
			\to
				R \sheafhom_{X_{\sm}}(j_{x !} N, \Gm)
		\]
	in $D(X_{\sm})$
	(where $N$ is base-changed to $X$).
	We have
		\[
				R \sheafhom_{X_{\sm}}(N, \Gm)
			=
				N^{\PDual} \tensor^{L} \Gm[-1],
		\]
		\begin{align*}
			&
					R \sheafhom_{X_{\sm}}(j_{x !} N, \Gm)
				=
					R j_{x \ast}
					R \sheafhom_{U_{\sm}}(N, \Gm)
			\\
			&	=
					R j_{x \ast}
					(N^{\PDual} \tensor^{L} \Gm)[-1]
				=
					N^{\PDual} \tensor^{L} R j_{x \ast} \Gm[-1].
		\end{align*}
	Hence
		\[
				R \sheafhom_{X_{\sm}}(i_{x \ast} N, \Gm)
			=
					N^{\PDual}
				\tensor^{L}
					[\Gm \to R j_{x \ast} \Gm][-2].
		\]
	By definition, $j_{x \ast} \Gm$ is the N\'eron model over $X$ of $\Gm$ over $U$.
	Hence it fits in the exact sequence
		\[
			0 \to \Gm \to j_{x \ast} \Gm \to i_{x \ast} \Z \to 0.
		\]
	We have $R^{1} j_{x \ast} \Gm = 0$
	by the proof of \cite[III, Lem.\ C.10]{Mil06}.
	Hence 
		\[
				\tau_{\le 2}
				R \sheafhom_{X_{\sm}}(i_{x \ast} N, \Gm)
			=
					N^{\PDual}
				\tensor^{L}
					i_{x \ast} \Z[-2]
			=
				i_{x \ast} N^{\PDual}[-2].
		\]
\end{proof}

\begin{Lem}
	We have
		\[
				\sheafhom_{X_{\sm}}(\mathcal{A}, \Gm)
			=
				\sheafhom_{X_{\sm}}(\mathcal{A}_{0}, \Gm)
			=
				0.
		\]
\end{Lem}

\begin{proof}
	Any morphism from $\mathcal{A}$ or $\mathcal{A}_{0}$ to $\Gm$
	over any smooth scheme over $X$ is generically zero and hence zero.
	This implies the result.
\end{proof}

Now we prove Prop.\ \ref{prop: Poincare pairing and Grothendieck pairing}.

\begin{proof}[Proof of Prop.\ \ref{prop: Poincare pairing and Grothendieck pairing}]
	The commutativity of the left square in the diagram in the statement is easy to see.
	To see the commutativity of the right square,
	the above three lemmas show that it is enough to check the commutativity of the diagram
		\[
			\begin{CD}
					\mathcal{A}^{\vee}
				@>>>
					\Phi_{A^{\vee}, X}
				\\
				@VVV
				@VVV
				\\
					\sheafext_{X_{\sm}}^{1}(\mathcal{A}_{0}, \Gm)
				@>>>
					\Phi_{A, X}^{\PDual},
			\end{CD}
		\]
	where we denoted
		\[
				\Phi_{A, X}^{\PDual}
			=
				\bigoplus_{x}
					i_{x \ast} (\pi_{0}(\mathcal{A}_{x})^{\PDual})
		\]
	Any morphism $\mathcal{A}^{\vee} \to \Phi_{A, X}^{\PDual}$ is determined by
	its values at $\Order_{x}^{sh}$ for all closed points $x \in X$,
	where $\Order_{x}^{sh}$ is the strict henselization of $X$ at $x$.
	Hence it is enough to show that the diagram
		\[
			\begin{CD}
					\mathcal{A}^{\vee}(\Order_{x}^{sh})
				@>>>
					\pi_{0}(\mathcal{A}_{x}^{\vee})(\closure{k_{x}})
				\\
				@VVV
				@VVV
				\\
					\Ext_{\Order_{x, \sm}^{sh}}^{1}(\mathcal{A}_{0}, \Gm)
				@>>>
					\Ext_{\Order_{x, \sm}^{sh}}^{2}(
						i_{x \ast} \pi_{0}(\mathcal{A}_{x}),
						\Gm
					)
			\end{CD}
		\]
	is commutative,
	where $\closure{k_{x}}$ is the algebraic closure of $k_{x}$.
	Let $K_{x}^{sh}$ be the fraction field of $\Order_{x}^{sh}$.
	Let $\xi \colon 0 \to \Gm \to H \to A \to 0$ be an extension as an element of
		\[
				\mathcal{A}^{\vee}(\Order_{x}^{sh})
			=
				A^{\vee}(K_{x}^{sh})
			=
				\Ext_{K_{x, \sm}^{sh}}^{1}(A, \Gm),
		\]
	where the second isomorphism is the Barsotti-Weil formula \cite[Chap.\ III, Thm.\ (18.1)]{Oor66}.
	Let $\mathcal{H}$ be the N\'eron (lft) model over $\Order_{x}^{sh}$ of $H$
	and $\mathcal{H}_{0}$ the maximal open subgroup scheme with connected fibers.
	The image of $\xi$ under the left vertical morphism is
	the extension $0 \to \Gm \to \mathcal{H}_{0} \to \mathcal{A}_{0} \to 0$
	(which is exact since $R^{1} j_{x \ast} \Gm = 0$
	and $\mathcal{H}_{0}$ is of finite type).
	Its image under the lower vertical morphism is the extension
		\[
				\eta_{1}
			\colon
				0
			\to
				\Gm
			\to
				\mathcal{H}_{0}
			\to
				\mathcal{A}
			\to
				i_{x \ast} \pi_{0}(\mathcal{A}_{x})
			\to
				0
		\]
	given by composing it with
	$0 \to \mathcal{A}_{0} \to \mathcal{A} \to i_{x \ast} \pi_{0}(\mathcal{A}_{x}) \to 0$.
	On the other hand, we have an exact sequence
		\[
				0
			\to
				i_{x \ast} \Z
			\to
				i_{x \ast} \pi_{0}(\mathcal{H}_{x})
			\to
				i_{x \ast} \pi_{0}(\mathcal{A}_{x})
			\to
				0.
		\]
	The image of the extension $\xi$ under the upper horizontal morphism
	followed by the right vertical morphism is the extension
		\[
				\eta_{2}
			\colon
				0
			\to
				\Gm
			\to
				\mathcal{G}_{m}
			\to
				i_{x \ast} \pi_{0}(\mathcal{H}_{x})
			\to
				i_{x \ast} \pi_{0}(\mathcal{A}_{x})
			\to
				0
		\]
	given by the composite with
	$0 \to \Gm \to \mathcal{G}_{m} \to i_{x \ast} \Z \to 0$.
	We need to show that $\eta_{1}$ and $\eta_{2}$ are equivalent.
	Denote by $\mathcal{H}'$ the inverse image of $\mathcal{A}_{0}$
	by $\mathcal{H} \onto \mathcal{A}$.
	Consider the extension
		\[
				\eta_{3}
			\colon
				0
			\to
				\Gm
			\to
				\mathcal{H}_{0} \times_{\mathcal{A}_{0}} \mathcal{H}'
			\to
				\mathcal{H}
			\to
				i_{x \ast} \pi_{0}(\mathcal{A}_{x})
			\to
				0,
		\]
	where the first morphism (from $\Gm$ to $\mathcal{H}_{0} \times_{\mathcal{A}_{0}} \mathcal{H}'$)
	is the inclusion into the first factor,
	the second the projection onto to the second factor
	and the third the natural morphism.
	We have a morphism $\eta_{3} \to \eta_{1}$ of extensions,
	where $\mathcal{H}_{0} \times_{\mathcal{A}_{0}} \mathcal{H}' \to \mathcal{H}_{0}$
	is the first projection.
	We also have a morphism $\eta_{3} \to \eta_{1}$,
	where $\mathcal{H}_{0} \times_{\mathcal{A}_{0}} \mathcal{H}' \to \mathcal{G}_{m}$
	is the subtraction map $(a, b) \mapsto a - b$.
	Therefore $\eta_{1}$ and $\eta_{2}$ are equivalent.
	This proves that the right square in the statement is commutative.
	
	We finally show that the hidden square
		\[
			\begin{CD}
					\Phi_{A^{\vee}, X}
				@>>>
					\mathcal{A}_{0}^{\vee}[1]
				\\
				@VVV
				@VVV
				\\
					R \sheafhom_{X}(
						\Phi_{A, X},
						\Gm
					)[2]
				@>>>
					R \sheafhom_{X}(\mathcal{A}, \Gm)[2]
			\end{CD}
		\]
	is commutative.
	Interchanging the variables, this is equivalent to showing that the diagram
		\[
			\begin{CD}
					\mathcal{A}
				@>>>
					R \sheafhom_{X}(
						\mathcal{A}_{0}^{\vee},
						\Gm
					)[1]
				\\
				@VVV
				@VVV
				\\
					\Phi_{A, X}
				@>>>
					R \sheafhom_{X}(\Phi_{A^{\vee, X}}, \Gm)[2]
			\end{CD}
		\]
	is commutative.
	This diagram is the same, up to replacing $A$ by $A^{\vee}$,
	as the diagram whose commutativity has just been proved.
\end{proof}

\begin{Prop} \label{prop: component group compatibitily morphism}
	The diagram in the previous proposition,
	after applying $R \alg{\Gamma}(X, \var)$,
	the cup product morphism \eqref{eq: global cup product}
	and the trace morphism \eqref{eq: global trace morphism},
	induce a morphism of distinguished triangles
		\[
			\begin{CD}
					R \alg{\Gamma}(X, \mathcal{A}_{0}^{\vee})
				@>>>
					R \alg{\Gamma}(X, \mathcal{A}^{\vee})
				@>>>
					\bigoplus_{x}
						\Res_{k_{x} / k}
						\pi_{0}(\mathcal{A}_{x}^{\vee})
				\\
				@VVV
				@VVV
				@VVV
				\\
					R \alg{\Gamma}(X, \mathcal{A})^{\SDual}
				@>>>
					R \alg{\Gamma}(X, \mathcal{A}_{0})^{\SDual}
				@>>>
					\bigoplus_{x}
						\Res_{k_{x} / k}
						\pi_{0}(\mathcal{A}_{x})^{\PDual}
			\end{CD}
		\]
	in $D(k)$.
	The right vertical morphism is the sum over $x \in X$ of
	the Weil restrictions of Grothendieck's pairings,
	which is an isomorphism \cite[Thm.\ C]{Suz14}.
\end{Prop}

\begin{proof}
	The existence of the stated morphism of distinguished triangle is self-explanatory.
	To show the description of the right vertical morphism,
	it is enough to show that the morphism
	$\bigoplus_{x} i_{x \ast} \Z \to \Gm[1]$ after applying $R \alg{\Gamma}(X, \var)$
	can be identified with the summation map
	$\bigoplus_{x} \Res_{k_{x} / k} \Z \to \Z$.
	The group $\bigoplus_{x} i_{x \ast} \Z$ can be identified with the sheaf of divisors on $X$
	and the sequence
	$0 \to \Gm \to \mathcal{G}_{m} \to \bigoplus_{x} i_{x \ast} \Z \to 0$
	can be identified with the divisor exact sequence.
	Hence the composite
		\[
				\alg{\Gamma} \left(
					X, \bigoplus_{x} \Res_{k_{x} / k} \Z
				\right)
			\to
				\alg{H}^{1}(X, \Gm)
			\to
				\Z
		\]
	of the connecting morphism of the divisor exact sequence and the degree map
	is the summation map.
\end{proof}

If $k$ is algebraically closed or finite,
we denote the Tate-Shafarevich group of $A$ over $K$ by $\Sha(A / K)$,
which is the kernel of the natural homomorphism
from $H^{1}(K, A)$ to the direct sum of $H^{1}(\Hat{K}_{x}, A)$ over the closed points $x \in X$.
If $k$ is algebraically closed,
we also call $\Sha(A / K)$ the Tate-Shafarevich group of $\mathcal{A}$ over $X$
and denote it by $\Sha(\mathcal{A} / X)$.
If $k$ is a general perfect field with algebraic closure $\closure {k}$,
then the group $\Sha(\mathcal{A}_{\closure{k}} / X_{\closure{k}})$ has
a natural action of $G_{k} = \Gal(\closure{k} / k)$.
Let $\Sha(\mathcal{A}_{\closure{k}} / X_{\closure{k}})^{G_{k}}$ be the $G_{k}$-invariant part,
which is independent of the choice of an algebraic closure $\closure{k}$.
We use similar notation when $k$ is replaced by any perfect field $k'$ over $k$.
Consider the functor
	\[
			k'
		\mapsto
			\Sha(\mathcal{A}_{\closure{k'}} / X_{\closure{k'}})^{G_{k'}}
	\]
on perfect fields $k'$ over $k$,
which commutes with filtered direct limits.
(Note that $\Sha(\mathcal{A}_{\closure{k'}} / X_{\closure{k'}})$
cannot be written as ``$\Sha(A \times_{k} k' / K \tensor_{k} k')$'';
the latter does not even make sense if $k'$ is not algebraic over $k$
since the ring $K \tensor_{k} k'$ in this case is not a field.
The scheme $X_{\closure{k'}}$ has much more closed points than $X_{\closure{k}}$,
which significantly affect the definition of $\Sha(\mathcal{A}_{\closure{k'}} / X_{\closure{k'}})$.)
The above functor uniquely extends to a functor $k^{\ind\rat} \to \Ab$
that commutes with finite products and filtered direct limits.
We still denote this extended functor by
$k' \mapsto \Sha(\mathcal{A}_{\closure{k'}} / X_{\closure{k'}})^{G_{k'}}$
by abuse of notation.
It is obviously a sheaf for the \'etale topology.
It is moreover a sheaf for the pro-\'etale topology since it commutes with filtered direct limits.

\begin{Prop} \label{prop: first cohom is Tate Shafarevich}
	The sheaf $\alg{H}^{1}(X, \mathcal{A})$ on $\Spec k^{\ind\rat}_{\pro\et}$
	is canonically isomorphic to the sheaf
	$k' \mapsto \Sha(\mathcal{A}_{\closure{k'}} / X_{\closure{k'}})^{G_{k'}}$.
\end{Prop}

\begin{proof}
	The sheaf $\alg{H}^{1}(X, \mathcal{A})$ is locally of finite presentation as seen before.
	It is enough to show that
	the group of $k'$-valued points of $\alg{H}^{1}(X, \mathcal{A})$ is canonically isomorphic to
	$\Sha(\mathcal{A}_{\closure{k'}} / X_{\closure{k'}})$
	for any algebraically closed field $k'$ over $k$.
	The former group is $H^{1}(X_{\closure{k'}}, \mathcal{A})$.
	That it is canonically isomorphic to the latter is \cite[Lem.\ 11.5]{Mil06}.
\end{proof}

\begin{Prop} \label{prop: Sha does not have divisible connected unipotent part}
	The group $V_{p} H^{1}(X_{k'}, \mathcal{A})$ as a functor on algebraically closed fields $k'$ over $k$
	is constant.
\end{Prop}

\begin{proof}
	We may assume that $k = \closure{k}$.
	By \cite[Thm.\ 11]{Kat99},
	there exist a proper smooth geometrically connected curve $C$ over $K$ having a $K$-rational point
	and a surjective homomorphism $J \onto A$ from the Jacobian $J$ of $C$ over $K$.
	By Poincar\'e complete reducibility,
	there exists a homomorphism $A \to J$ over $K$ such that
	the composite $A \to J \onto A$ is multiplication by some positive integer $m$.
	Let $\mathcal{J}$ be the N\'eron model of $J$ over $X$.
	Then we have homomorphisms $\mathcal{A} \to \mathcal{J} \to \mathcal{A}$ over $X$
	whose composite is multiplication by $m$.
	Therefore $V_{p} H^{1}(X, \mathcal{A})$ is a direct factor of $V_{p} H^{1}(X, \mathcal{J})$.
	Hence it is enough to show that
	$V_{p} H^{1}(X, \mathcal{J})$ does not depend on the algebraically closed base field $k$.
	Let $\mathcal{C} / X$ be a proper flat regular model of $C / K$ (\cite{Lip78}).
	By \cite[\S 4.6]{Gro66},
	there exists a canonical isomorphism $H^{1}(X, \mathcal{J}) \cong H^{2}(\mathcal{C}, \Gm)$.
	Hence it is enough to show that
	$V_{p} H^{2}(\mathcal{C}, \Gm)$ does not depend on the algebraically closed base field $k$.
	Note that $H^{2}(\mathcal{C}, \Gm)$ is the Brauer group
	of the proper smooth surface $\mathcal{C}$ over $k$.
	
	By \cite[II, (5.8.5)]{Ill79},
	we have a canonical exact sequence
		\[
				0
			\to
				\NS(\mathcal{C}) \tensor \Q_{p}
			\to
				H^{2}(\mathcal{C}, \Q_{p}(1))
			\to
				V_{p} H^{2}(\mathcal{C}, \Gm)
			\to
				0,
		\]
	where $\NS$ denotes the N\'eron-Severi group
	and the middle term is
		\[
			\bigl(
				\invlim_{n}
					H^{2}(\mathcal{C}, \mu_{p^{n}})
			\bigr) \tensor \Q.
		\]
	The group $\NS(\mathcal{C})$ does not depend on $k$.
	By \cite[II, Thm.\ 5.5.3]{Ill79},
	there exists a canonical exact sequence
		\[
				0
			\to
				H^{2}(\mathcal{C}, \Q_{p}(1))
			\to
				H_{\mathrm{crys}}^{2}(\mathcal{C} / W(k)) \tensor \Q
			\overset{F - p}{\to}
				H_{\mathrm{crys}}^{2}(\mathcal{C} / W(k)) \tensor \Q
			\to
				0,
		\]
	where the last two groups are the rational crystalline cohomology.
	They are finite-dimensional over $W(k)[1 / p]$.
	Hence $H^{2}(\mathcal{C}, \Q_{p}(1))$ is finite-dimensional over $\Q_{p}$
	whose dimension does not depend on $k$.
	Therefore $V_{p} H^{2}(\mathcal{C}, \Gm)$ does not depend on $k$.
	This finishes the proof.
\end{proof}

\begin{Prop} \label{prop: H1 is locally algebraic}
	The group $\alg{H}^{1}(X, \mathcal{A})$ is in $\Loc^{f} \Alg_{\uc} / k$.
\end{Prop}

\begin{proof}
	This follows from the previous propositions,
	Prop.\ \ref{prop: structure of cohomology of Neron models}
	and Prop.\ \ref{prop: criteiron for algebraicity of identity component}.
\end{proof}

\begin{Prop} \label{prop: our duality contains height pairing}
	The morphism \eqref{eq: duality pairing for Neron over X} induces a morphism
		\[
					\Gamma(X, \mathcal{A}_{0}^{\vee})
				\tensor
					\Gamma(X, \mathcal{A})
			\to
				\Z.
		\]
	This agrees with the height pairing \cite[III, \S 3]{MB85}.
\end{Prop}

\begin{proof}
	This morphism is equal to the morphism
		\[
					\Gamma(X, \mathcal{A}_{0}^{\vee})
				\tensor
					\Gamma(X, \mathcal{A})
			\to
				H^{1}(X, \Gm)
			\to
				\Z.
		\]
	Hence a more explicit description can be given as follows.
	Let $P$ be the extension of the Poincar\'e bundle to $\mathcal{A}_{0}^{\vee} \times_{X} \mathcal{A}$.
	Let $f \colon X \to \mathcal{A}_{0}^{\vee}$ and $g \colon X \to \mathcal{A}$ be sections over $X$.
	By pulling back $P$ by $f \times g \colon X \to \mathcal{A}_{0}^{\vee} \times_{X} \mathcal{A}$,
	we have a line bundle on $X$.
	Its degree is the value of the pairing at $(f, g)$.
	This pairing is equal to the height pairing by \cite[III, \S 3]{MB85}.
\end{proof}

\begin{Prop}
	Consider the morphism
		\[
					\pi_{0}(\alg{\Gamma}(X, \mathcal{A}_{0}^{\vee}))_{/ \tor}
				\times
					\pi_{0}(\alg{\Gamma}(X, \mathcal{A}))_{/ \tor}
			\to
				\Z
		\]
	coming from \eqref{eq: duality pairing for Neron over X}.
	The induced morphism
		\[
				\pi_{0}(\alg{\Gamma}(X, \mathcal{A}_{0}^{\vee}))_{/ \tor}
			\to
				\pi_{0}(\alg{\Gamma}(X, \mathcal{A}))_{/ \tor}^{\LDual}
		\]
	is injective with finite cokernel.
\end{Prop}

\begin{proof}
	This follows from the previous proposition
	and the non-degeneracy of the height pairing \cite[Thm.\ 9.15]{Con09}.
\end{proof}

We summarize the results obtained so far.

\begin{Prop} \label{prop: results before the formal steps}
	Let $C = R \alg{\Gamma}(X, \mathcal{A})$
	and $D = R \alg{\Gamma}(X, \mathcal{A}_{0}^{\vee})$,
	which are objects of $D(k)$.
	The sheaf $H^{0} C$ is an extension of a finitely generated \'etale group
	by an abelian variety;
	$H^{1} C \in \Loc^{f} \Alg / k$;
	$H^{2} C \in \Ind^{f} \Alg / k$, which is divisible;
	and $H^{n} C = 0$ for other values of $n$.
	The same are true for $D$.
	There is a canonical pairing $C \tensor^{L} D \to \Z$ in $D(k)$.
	The induced morphism
		\[
				(\pi_{0} H^{0} D)_{/ \tor}
			\to
				(\pi_{0} H^{0} C)_{/ \tor}^{\LDual}
		\]
	is injective with finite cokernel.
	For any $n \ge 1$, the induced morphism
		\[
				D \tensor^{L} \Z / n \Z
			\to
				(C \tensor^{L} \Z / n \Z)^{\SDual}[1]
		\]
	is an isomorphism.
\end{Prop}


\subsection{Formal steps towards duality for N\'eron models}
\label{sec: Formal steps towards duality for Neron models}

Throughout this subsection,
we fix two objects $C, D \in D(k)$ and a morphism $C \tensor^{L} D \to \Z$,
and assume the following:
\begin{enumerate}
	\item \label{item: structure of cohom of C}
		The sheaf $H^{0} C$ is an extension of a finitely generated \'etale group
		by an abelian variety;
		$H^{1} C \in \Loc^{f} \Alg / k$;
		$H^{2} C \in \Ind^{f} \Alg / k$, which is divisible;
		and $H^{n} C = 0$ for other values of $n$.
	\item \label{item: structure of cohom of D}
		The same are true for $D$.
	\item \label{item: height pairing for C and D}
		The morphism
			\[
					(\pi_{0} H^{0} D)_{/ \tor}
				\to
					(\pi_{0} H^{0} C)_{/ \tor}^{\LDual}
			\]
		induced from $C \tensor^{L} D \to \Z$ is injective 
		and its cokernel $\delta_{\Height}$ is finite.
	\item \label{item: mod n duality for C and D}
		For any $n \ge 1$, the induced morphism
			\[
					D \tensor^{L} \Z / n \Z
				\to
					(C \tensor^{L} \Z / n \Z)^{\SDual}[1]
			\]
		is an isomorphism.
\end{enumerate}
It follows from \eqref{item: structure of cohom of C} that
$C \tensor^{L} \Z / n \Z \in D^{b}(\Alg_{\uc} / k)$ for any $n \ge 1$.
The same is true for $D$.
Hence the isomorphism in \eqref{item: mod n duality for C and D} belongs to $D^{b}(\Alg_{\uc} / k)$.

The goal of this subsection is to prove that there exist canonical morphisms
	\[
			C^{\SDual}
		\to
			V H^{1} D
		\to
			D^{\SDual \SDual}[1]
	\]
such that the triangle
	\[
			V H^{1} D[-1]
		\to
			D^{\SDual \SDual}
		\to
			C^{\SDual}
		\to
			V H^{1} D
	\]
is distinguished (Prop.\ \ref{prop: formal duality for Neron}).
In order to prove this, we first extract as much information as possible
from the limit in $n$ of the isomorphism \eqref{item: mod n duality for C and D}
(Prop.\ \ref{prop: duality for Neron after derived completion} to
\ref{prop: adelic duality for Neron}).
Then we describe each cohomology object of
$D^{\SDual \SDual}$ and $C^{\SDual}$ (Prop.\ \ref{prop: each term of duality for Neron}).
They are concentrated in degrees $-1, 0, 1, 2$.
With these two steps and \eqref{item: height pairing for C and D},
we can show that there exist canonical distinguished triangles
	\begin{gather*}
				D^{\SDual \SDual} \tensor \Q
			\to
				C^{\SDual} \tensor \Q
			\to
				V H^{1} D,
		\\
				R \invlim_{n} D^{\SDual \SDual}
			\to
				R \invlim_{n} C^{\SDual}
			\to
				V H^{1} D
	\end{gather*}
(Prop.\ \ref{prop: each part of Q tensor duality} to
\ref{prop: duality triangle tensor Q} for the first triangle and
Prop.\ \ref{prop: each part of inverse by multiplication duality}
to \ref{prop: duality triangle after inverse multiplication} for the second).
These are easier to establish than the integral statement
since $D^{\SDual \SDual} \tensor \Q$ and $C^{\SDual} \tensor \Q$
are concentrated in degrees $-1$, $0$,
and $R \invlim_{n} D^{\SDual \SDual}$ and $R \invlim_{n} C^{\SDual}$
are concentrated in degrees $1$, $2$.
Observe that these ranges of degrees have no intersection.
This disjointness puts a strong restriction on possible choices of a mapping cone
of $D^{\SDual \SDual} \to C^{\SDual}$.
With this, we can get the desired canonical distinguished triangle.

We begin with taking limits:

\begin{Prop} \label{prop: duality for Neron after derived completion}
	The morphism $D \to C^{\SDual}$ induces an isomorphism
		\[
				R \invlim_{n}(D \tensor^{L} \Z / n \Z)
			\isomto
				R \invlim_{n}(C^{\SDual} \tensor^{L} \Z / n \Z).
		\]
\end{Prop}

\begin{proof}
	For any $n \ge 1$, we have
		\[
				C^{\SDual} \tensor^{L} \Z / n \Z
			=
				(C \tensor^{L} \Z / n \Z)^{\SDual}[1].
		\]
	Hence Prop.\ \ref{prop: all mod n isom imply completion isom} implies the result.
\end{proof}

To calculate the right-hand side,
it is easier to write it with torsion objects only:

\begin{Prop} \label{prop: Q mod Z duality for Neron models}
	We have
		\[
				R \invlim_{n} \bigl(
					R \sheafhom_{k}(C, \Z) \tensor^{L} \Z / n \Z
				\bigr)
			=
				R \sheafhom_{k}(C \tensor^{L} \Q / \Z, \Q / \Z).
		\]
	Hence the isomorphism in the previous proposition can also be written as
		\[
				R \invlim_{n}(D \tensor^{L} \Z / n \Z)
			\isomto
				R \sheafhom_{k}(C \tensor^{L} \Q / \Z, \Q / \Z).
		\]
\end{Prop}

\begin{proof}
	We have
		\[
				R \sheafhom_{k}(C \tensor^{L} \Q / \Z, \Q / \Z)
			=
				R \invlim_{n}
				R \sheafhom_{k}(C \tensor^{L} \Z / n \Z, \Q / \Z)
		\]
	by \cite[Prop.\ (2.2.3)]{Suz14} (or its proof) in the notation therein.
	For any $n \ge 1$, we have
		\begin{align*}
					R \sheafhom_{k}(C \tensor^{L} \Z / n \Z, \Q / \Z)
			&	=
					R \sheafhom_{k}(C \tensor^{L} \Z / n \Z, \Z[1])
			\\
			&	=
					R \sheafhom_{k}(C, \Z) \tensor^{L} \Z / n \Z
		\end{align*}
	since derived-tensoring $\Z / n \Z$ kills uniquely divisibles.
	The result follows by taking the limit.
\end{proof}

\begin{Prop} \label{prop: each part of Q mod Z duality for Neron models}
	The groups $H^{2} C, H^{2} D \in \Ind^{f} \Alg / k$ are \'etale.
	The groups $H^{1} C, H^{1} D \in \Loc^{f} \Alg / k$ are divisibly ML.
	The isomorphism in Prop.\ \ref{prop: Q mod Z duality for Neron models} yields, on cohomology,
	the following duality pairings and morphisms:
		\begin{enumerate}
			\item \label{item: formal duality between TGamma of A and H2 of A dual}
				Pontryagin duality between
					$
							T (H^{0} C)_{0}
						\in
							\Pro \FEt / k
					$
				and
					$
							H^{2} D
						\in
							\Ind \FEt / k
					$.
			\item \label{item: formal duality between TGamma of A dual and H2 of A}
				Pontryagin duality between
					$
							T (H^{0} D)_{0}
						\in
							\Pro \FEt / k
					$
				and
					$
							H^{2} C
						\in
							\Ind \FEt / k
					$.
			\item \label{item: formal duality between torsion MW of A dual and finite quotient of H1 A}
				Pontryagin duality between
					$
							(\pi_{0} H^{0} C)_{\tor},
							\pi_{0}(H^{1} D)_{/ \divis}
						\in
							\FEt / k
					$.
			\item \label{item: formal duality between torsion MW of A and finite quotient of H1 A dual}
				Pontryagin duality between
					$
							(\pi_{0} H^{0} D)_{\tor},
							\pi_{0} (H^{1} C)_{/ \divis}
						\in
							\FEt / k
					$.
			\item \label{item: formal transcendental pairing}
				An injection
					$
							(\pi_{0} H^{1} C)_{\divis}^{\PDual}
						\into
							T (H^{1} D)_{\divis}
					$
				in $\Pro \FEt / k$ whose cokernel $\delta_{\Tran}$ is finite.
			\item \label{item: formal Cassels Tate pairing}
				A surjection
					$
							(H^{1} C)_{0}^{\SDual'}
						\onto
							((H^{1} D)_{/ \divis})_{0}
					$
				of connected unipotent quasi-algebraic groups
				whose kernel $\delta_{\CT}$ is finite.
			\item \label{item: formal discriminant exact sequence}
				An exact sequence
				$0 \to \delta_{\Tran} \to \delta_{\Height} \to \delta_{\CT} \to 0$
				in $\FEt / k$.
		\end{enumerate}
\end{Prop}

The suggestive subscripts $\CT$ and $\Tran$ will be explained after
Thm.\ \ref{thm: duality for Neron}.
The directions of the morphisms in \eqref{item: formal transcendental pairing}
and \eqref{item: formal Cassels Tate pairing} may look wrong,
but they are indeed correct.

\begin{proof}
	We set
		\[
				X = C \tensor^{L} \Q / \Z,
			\quad
				Y = R \sheafhom_{k}(X, \Q / \Z),
			\quad
				Z = R \invlim_{n}(D \tensor^{L} \Z / n \Z).
		\]
	We are going to write down the effects of the isomorphism $Z \isomto Y$ on their cohomology objects.
	For any $i$, we have
		\begin{gather*}
					0
				\to
					(H^{i} C) \tensor \Q / \Z
				\to
					H^{i} X
				\to
					(H^{i + 1} C)_{\tor}
				\to
					0,
			\\
					0
				\to
					(H^{-i + 1} X)_{0}^{\SDual'}
				\to
					H^{i} Y
				\to
					\pi_{0}(H^{-i} X)^{\PDual}
				\to
					0,
			\\
					0
				\to
					(H^{i} D)^{\wedge}
				\to
					H^{i} Z
				\to
					T(H^{i + 1} D)
				\to
					0,
		\end{gather*}
	where the second line comes from \cite[Prop.\ (2.4.1) (b)]{Suz14} and \cite[III, Thm.\ 0.14]{Mil06},
	and the third line comes from Prop.\ \ref{prop: cohomologies of derived completion}
	and \cite[Prop.\ (2.1.2) (f)]{Suz14}.
	We have $H^{-1} X = (H^{0} C)_{\tor}$.
	From the structure of $H^{0} C$ and the torsionness of $H^{1} C$ and $H^{2} C$,
	we have
		\begin{gather*}
					0
				\to
					T(H^{0} C)_{0} \tensor \Q / \Z
				\to
					H^{-1} X
				\to
					(\pi_{0} H^{0} C)_{\tor}
				\to
					0,
			\\
					0
				\to
					\pi_{0}(H^{0} C)_{/ \tor} \tensor \Q / \Z
				\to
					H^{0} X
				\to
					H^{1} C
				\to
					0,
			\\
					H^{1} X
				=
					H^{2} C,
		\end{gather*}
	and $H^{i} X = 0$ for other values of $i$.
	In particular, $H^{-1} X$ has trivial identity component.
	Using these, we have
		\begin{gather*}
					H^{-1} Y
				=
					\pi_{0}(H^{2} C)^{\PDual},
			\\
					0
				\to
					(H^{2} C)_{0}^{\SDual}
				\to
					H^{0} Y
				\to
					\pi_{0}(H^{0} X)^{\PDual}
				\to
					0,
			\\
					0
				\to
					(H^{0} X)_{0}^{\SDual'}
				\to
					H^{1} Y
				\to
					(H^{-1} X)^{\PDual}
				\to
					0,
		\end{gather*}
	and $H^{i} Y = 0$ for other values of $i$.
	On the other hand, we know that
	$T H^{0} D = T(H^{0} D)_{0}$ and
	$(H^{0} D)^{\wedge} = \pi_{0}(H^{0} D)^{\wedge}$
	from the structure of $H^{0} D$
	and that $(H^{2} D)^{\wedge} = 0$ from the divisibility of $H^{2} D$.
	Hence
		\begin{gather*}
					H^{-1} Z
				=
					T (H^{0} D)_{0},
			\\
					0
				\to
					\pi_{0}(H^{0} D)^{\wedge}
				\to
					H^{0} Z
				\to
					T(H^{1} D)_{\divis}
				\to
					0,
			\\
					0
				\to
					(H^{1} D)^{\wedge}
				\to
					H^{1} Z
				\to
					T H^{2} D
				\to
					0,
		\end{gather*}
	and $H^{i} Z = 0$ for other values of $i$.
	
	Since $H^{1} D \in \Loc^{f} \Alg_{\uc} / k$,
	we know that $H^{0} Z \cong H^{0} Y$ has trivial identity component.
	Hence so does $H^{2} C$.
	Thus $H^{2} C$ is \'etale.
	Since $H^{1} C \in \Loc^{f} \Alg_{\uc} / k$,
	we know that the identify component of $H^{0} X$ is quasi-algebraic.
	Hence so is $H^{1} Y \cong H^{1} Z$,
	and so is $T H^{2} D$.
	Therefore $H^{2} D$ is \'etale.
	
	Hence $H^{i} Z \cong H^{i} Y$ has trivial identity component for $i \ne 1$.
	Comparing $H^{-1}$, $H^{0}$, $(H^{1})_{0}$, $\pi_{0} H^{1}$ of $Y$ and $Z$, we have
		\begin{gather*}
					T(H^{0} D)_{0}
				\cong
					(H^{2} C)^{\PDual},
			\\
					0
				\to
					(\pi_{0} H^{0} D)^{\wedge}
				\to
					(\pi_{0} H^{0} X)^{\PDual}
				\to
					T(H^{1} D)_{\divis}
				\to
					0,
			\\
					((H^{1} D)^{\wedge})_{0}
				\cong
					(H^{0} X)_{0}^{\SDual'},
			\\
					0
				\to
					\pi_{0}((H^{1} D)^{\wedge})
				\to
					(H^{-1} X)^{\PDual}
				\to
					T H^{2} D
				\to
					0,
		\end{gather*}
	respectively.
	We have \eqref{item: formal duality between TGamma of A dual and H2 of A} from the first line.
	Comparing the last line with the exact sequence for $H^{-1} X$ given above
	and using the finiteness of $(\pi_{0} H^{0} C)_{\tor}$,
	we obtain \eqref{item: formal duality between TGamma of A and H2 of A dual},
	\eqref{item: formal duality between torsion MW of A dual and finite quotient of H1 A}
	and that $H^{1} D$ is divisibly ML.
	From the second line,
	we have $(\pi_{0} H^{0} X)_{/ \divis} = (\pi_{0} H^{0} D)_{\tor}^{\PDual}$,
	which is finite.
	On the other hand, the exact sequence for $H^{0} X$ given above gives
	$(\pi_{0} H^{0} X)_{/ \divis} = (\pi_{0} H^{1} C)_{/ \divis}$.
	Hence we obtain \eqref{item: formal duality between torsion MW of A and finite quotient of H1 A dual}
	and that $H^{1} C$ is divisibly ML.
	By Prop.\ \ref{prop: derived limit for multiplication},
	we have $(H^{1} D)^{\wedge} = (H^{1} D)_{/ \divis}$.
	
	From the second line, we have
		\begin{equation} \label{eq: source of discriminant ex seq, mapping from MW}
				0
			\to
				(\pi_{0} H^{0} D)_{/ \tor}^{\wedge}
			\to
				(\pi_{0} H^{0} X)_{\divis}^{\PDual}
			\to
				T(H^{1} D)_{\divis}
			\to
				0,
		\end{equation}
	where $(\pi_{0} H^{0} X)_{\divis}^{\PDual} = ((\pi_{0} H^{0} X)_{\divis})^{\PDual}$.
	Back to the relation between $X$ and $C$, we have
		\[
				0
			\to
				(\pi_{0} H^{0} C)_{/ \tor} \tensor \Q / \Z
			\to
				(H^{0} X)_{0 + \divis}
			\to
				(H^{1} C)_{0 + \divis}
			\to
				0.
		\]
	The long exact sequence of $\sheafext_{k}^{\var}(\var, \Q / \Z)$
	for this short exact sequence and
		$
				((H^{1} D)_{/ \divis})_{0}
			\cong
				(H^{0} X)_{0}^{\SDual'}
		$
	yields
		\begin{equation} \label{eq: source of discriminant ex seq, mapping to MW}
			\begin{aligned}
						0
				&	\to
						(\pi_{0} H^{1} C)_{\divis}^{\PDual}
					\to
						(\pi_{0} H^{0} X)_{\divis}^{\PDual}
					\to
						(\pi_{0} H^{0} C)_{/ \tor}^{\LDual \wedge}
				\\
				&	\to
						(H^{1} C)_{0}^{\SDual'}
					\to
						((H^{1} D)_{/ \divis})_{0}
					\to
						0.
			\end{aligned}
		\end{equation}
	We will apply the lemma below to the two exact sequences
	\eqref{eq: source of discriminant ex seq, mapping from MW}
	and \eqref{eq: source of discriminant ex seq, mapping to MW}.
	The composite
		\[
				(\pi_{0} H^{0} D)_{/ \tor}^{\wedge}
			\into
				(\pi_{0} H^{0} X)_{\divis}^{\PDual}
			\to
				(\pi_{0} H^{0} C)_{/ \tor}^{\LDual \wedge}
		\]
	is the completion of the injective morphism
	$(\pi_{0} H^{0} D)_{/ \tor} \into (\pi_{0} H^{0} C)_{/ \tor}^{\LDual}$
	with finite cokernel $\delta_{\Height}$.
	Hence we can apply the lemma, yielding
	\eqref{item: formal transcendental pairing},
	\eqref{item: formal Cassels Tate pairing} and
	\eqref{item: formal discriminant exact sequence}.
\end{proof}

\begin{Lem} \label{lem: discriminant exact sequence}
	Let
		\begin{gather*}
				0 \to F \to W \to E' \to 0,
			\\
				0 \to E \to W \to F' \to G \to G' \to 0
		\end{gather*}
	be exact sequences in an abelian category
	such that the composite $F \to W \to F'$ is injective.
	Then the composite $E \to W \to E'$ is injective, and we have an exact sequence
		\[
			0 \to \Coker(E \into E') \to \Coker(F \into F') \to \Ker(G \onto G') \to 0.
		\]
\end{Lem}

\begin{proof}
	Elementary.
\end{proof}

If the groups $(H^{1} C)_{0 \cap \divis}$, $(H^{1} D)_{0 \cap \divis}$
(which are in general finite \'etale $p$-groups by Prop.\ \ref{prop: connected divisible part}) are zero,
then the morphisms in \eqref{item: formal transcendental pairing}, \eqref{item: formal Cassels Tate pairing}
are simplified to more symmetric expressions
	\[
			(H^{1} C)_{\divis}^{\PDual}
		\into
			T (H^{1} D)_{\divis},
		\quad
			(H^{1} C)_{0}^{\SDual'}
		\onto
			(H^{1} D)_{0}.
	\]
We do not assume these conditions.
See Rmk.\ \ref{rmk: connected divisible part} for a little more details about this point.

\begin{Prop} \label{prop: cohom of derived completion of Neron}
	The cohomology objects $H^{n}$ of the complex
	$R \invlim_{n}(D \tensor^{L} \Z / n \Z)$
	are described as follows:
		\begin{gather*}
					H^{-1}
				=
					T(H^{0} D)_{0}
			\\
					0
				\to
					(\pi_{0} H^{0} D)^{\wedge}
				\to
					H^{0}
				\to
					T (H^{1} D)_{\divis}
				\to
					0,
			\\
					0
				\to
					(H^{1} D)_{/ \divis}
				\to
					H^{1}
				\to
					T H^{2} D
				\to
					0,
		\end{gather*}
	and $H^{n} = 0$ for $n \ne -1, 0, 1$.
	The cohomology objects $H^{n}$ of the complex
	$R \invlim_{n} (C^{\SDual} \tensor^{L} \Z / n \Z)$
	are described as follows:
		\begin{gather*}
					H^{-1}
				=
					(H^{2} C)^{\PDual},
			\\
					0
				\to
					(\pi_{0} H^{1} C)^{\PDual}
				\to
					H^{0}
				\to
					(\pi_{0} H^{0} C)_{/ \tor}^{\LDual \wedge}
				\to
					\delta_{\CT}
				\to
					0,
			\\
					0
				\to
					((H^{1} D)_{/ \divis})_{0}
				\to
					H^{1}
				\to
					(H^{0} C)_{\tor}^{\PDual}
				\to
					0,
		\end{gather*}
	and $H^{n} = 0$ for $n \ne -1, 0, 1$.
\end{Prop}

\begin{proof}
	This is mostly given in the proof of Prop.\ \ref{prop: each part of Q mod Z duality for Neron models}.
	We show here only the second line for the second complex.
	Similar to \eqref{eq: source of discriminant ex seq, mapping to MW},
	we have an exact sequence
		\[
			\begin{aligned}
						0
				&	\to
						(\pi_{0} H^{1} C)^{\PDual}
					\to
						(\pi_{0} H^{0} X)^{\PDual}
					\to
						(\pi_{0} H^{0} C)_{/ \tor}^{\LDual \wedge}
				\\
				&	\to
						(H^{1} C)_{0}^{\SDual'}
					\to
						((H^{1} D)_{/ \divis})_{0}
					\to
						0.
			\end{aligned}
		\]
	The $H^{0}$ of the second complex in the statement is $(\pi_{0} H^{0} X)^{\PDual}$.
	The kernel of the last surjection is $\delta_{\CT}$.
	Hence the get the desired exact sequence.
\end{proof}

Tensoring $\Q$ to the isomorphism in Prop.\ \ref{prop: duality for Neron after derived completion},
we have an isomorphism
	\begin{equation} \label{eq: adelic duality for Neron}
			D \ctensor \Adele^{\infty}
		\isomto
			C^{\SDual} \ctensor \Adele^{\infty},
	\end{equation}
where $\Adele^{\infty} = \Hat{\Z} \tensor_{\Z} \Q \in \Ind \Pro \FEt / k$
as in \S \ref{sec: The ind-rational pro-etale site and some derived limits}.

\begin{Prop} \label{prop: each part of adelic duality}
	The cohomology objects $H^{n}$ of the complex
	$D \ctensor \Adele^{\infty}$
	are described as follows:
		\begin{gather*}
					H^{-1}
				=
					V (H^{0} D)_{0},
			\\
					0
				\to
						(\pi_{0} H^{0} D)_{/ \tor}
					\tensor
						\Adele^{\infty}
				\to
					H^{0}
				\to
					V (H^{1} D)_{\divis}
				\to
					0,
			\\
					H^{1}
				=
					V H^{2} D,
		\end{gather*}
	and $H^{n} = 0$ for other degrees.
	The cohomology objects $H^{n}$ of the complex
	$C^{\SDual} \ctensor \Adele^{\infty}$
	are described as follows:
		\begin{gather*}
					H^{-1}
				=
					(V H^{2} C)^{\PDual},
			\\
					0
				\to
					(V (H^{1} C)_{\divis})^{\PDual}
				\to
					H^{0}
				\to
						(\pi_{0} H^{0} C)_{/ \tor}^{\LDual}
					\tensor
						\Adele^{\infty}
				\to
					0,
			\\
					H^{1}
				=
					(V (H^{0} C)_{0})^{\PDual},
		\end{gather*}
	and $H^{n} = 0$ for other degrees.
\end{Prop}

\begin{proof}
	This follows from the previous proposition.
\end{proof}

\begin{Prop} \label{prop: adelic duality for Neron}
	The isomorphism \eqref{eq: adelic duality for Neron} induces Pontryagin duality between
	$V (H^{0} D)_{0}$ and $V H^{2} C$
	and between $V H^{2} D$ and $V (H^{0} C)_{0}$.
	In the previous proposition,
	both of the sequence for $H^{0}$ of $D \ctensor \Adele^{\infty}$
	and the sequence for $H^{0}$ of $C \ctensor \Adele^{\infty}$
	canonically split.
	The parts
		$
				(\pi_{0} H^{0} D)_{/ \tor}
			\tensor
				\Adele^{\infty}
		$
	and
		$
				(\pi_{0} H^{0} C)_{/ \tor}
			\tensor
				\Adele^{\infty}
		$
	are Pontryagin dual to each other.
	The parts $V (H^{1} D)_{\divis}$ and $V (H^{1} C)_{\divis}$ are Pontryagin dual to each other.
\end{Prop}

\begin{proof}
	The induced morphism
		\[
					(\pi_{0} H^{0} D)_{/ \tor}
				\tensor
					\Adele^{\infty}
			\to
				(\pi_{0} H^{0} C)_{/ \tor}^{\LDual}
			\tensor
				\Adele^{\infty}
		\]
	is an isomorphism.
	The rest follows from this.
\end{proof}

\begin{Prop} \label{prop: each term of duality for Neron}
	The cohomology objects $H^{n}$ of the complex $D^{\SDual \SDual}$
	are described as follows:
		\begin{gather*}
					H^{-1}
				=
					T (H^{0} D)_{0},
			\quad
					H^{0}
				=
					\pi_{0} H^{0} D,
			\\
					H^{1}
				=
					H^{1} D,
			\quad
					H^{2}
				=
					H^{2} D,
		\end{gather*}
	and $H^{n} = 0$ for other values of $n$.
	The cohomology objects $H^{n}$ of the complex $C^{\SDual}$
	are described as follows:
		\begin{gather*}
					H^{-1}
				=
					(H^{2} C)^{\PDual},
			\\
				\begin{aligned}
							0
					&	\to
							(\pi_{0} H^{1} C)^{\PDual}
						\to
							H^{0}
						\to
							(\pi_{0} H^{0} C)_{/ \tor}^{\LDual}
					\\
					&	\to
							(H^{1} C)_{0}^{\SDual'}
						\to
							H^{1}
						\to
							(\pi_{0} H^{0} C)_{\tor}^{\PDual}
						\to
							0,
				\end{aligned}
			\\
					H^{2}
				=
					(T (H^{0} C)_{0})^{\PDual},
		\end{gather*}
	and $H^{n} = 0$ for other values of $n$.
\end{Prop}

\begin{proof}
	First we treat $D^{\SDual \SDual}$.
	The cohomology objects of the mapping cone
		\[
			\bigl[
					(H^{0} D)_{0}
				\to
					D
			\bigr]
		\]
	are in $\FGEt / k$ in degree $0$, in $\Ind \Alg_{\uc} / k$ in degrees $1$, $2$
	and zero in other degrees.
	Each of them is Serre reflexive by \cite[Prop.\ (2.4.1) (b)]{Suz14}.
	Hence this mapping cone itself is Serre reflexive.
	On the other hand, the double Serre dual of the abelian variety
	$(H^{0} C)_{0}$ over $k$
	is its Tate module placed in degree $-1$ by \cite[Prop.\ (2.4.1) (c)]{Suz14}.
	Therefore we have a distinguished triangle
		\[
				T (H^{0} D)_{0}[1]
			\to
				D^{\SDual \SDual}
			\to
				\bigl[
						(H^{0} D)_{0}
					\to
						D
				\bigr].
		\]
	The desired description follows from this.
	
	For $C^{\SDual}$, consider the hyperext spectral sequence
		\[
					E_{2}^{i j}
				=
					\sheafext_{k}^{i} \bigl(
						H^{- j} C, \Z
					\bigr)
			\Longrightarrow
				H^{i + j}
				R \sheafhom_{k} \bigl(
					C, \Z
				\bigr).
		\]
	For any object $E$ of $\Ind \Alg / k$ or $\FGEt / k$,
	we have
		\begin{gather*}
					\sheafhom_{k}(E, \Z)
				=
					(\pi_{0} E)_{/ \tor}^{\LDual},
				\quad
					\sheafext_{k}^{1}(E, \Z)
				=
					(\pi_{0} E)_{\tor}^{\PDual},
			\\
					\sheafext_{k}^{2}(E, \Z)
				=
					E_{0}^{\SDual'},
				\quad
					\sheafext_{k}^{\ge 3}(E, \Z)
				=
					0
		\end{gather*}
	by \cite[Prop.\ (2.4.1) (a), (b)]{Suz14}.
	We have
		\[
				(H^{1} C)_{/ \tor}
			=
				(H^{2} C)_{/ \tor}
			=
				(H^{2} C)_{0}
			=
				0
		\]
	by Prop.\ \ref{prop: each part of Q mod Z duality for Neron models}.
	Hence the $E_{2}^{i j}$-term of the above spectral sequence is zero
	unless $(i, j)$ is $(-1, -2)$, $(-1, -1)$, $(0, 0)$, $(2, -1)$ or $(2, 0)$.
	We have
		\[
				(H^{0} C)_{0}^{\SDual'}
			=
				(T (H^{0} C)_{0})^{\PDual}
		\]
	by \cite[Prop.\ (2.4.1) (c)]{Suz14}.
	Therefore the $E_2$-sheet gives the desired description.
\end{proof}

\begin{Prop} \label{prop: period map and its completion}
	The morphism
		\[
				(\pi_{0} H^{0} C)_{/ \tor}^{\LDual}
			\to
				(H^{1} C)_{0}^{\SDual'}
		\]
	in the previous proposition and the morphism
		\[
				(\pi_{0} H^{0} C)_{/ \tor}^{\LDual \wedge}
			\to
				(H^{1} C)_{0}^{\SDual'}
		\]
	in (the proof of) Prop.\ \ref{prop: cohom of derived completion of Neron}
	are compatible under the natural morphism $(\var) \to (\var)^{\wedge}$.
\end{Prop}

\begin{proof}
	The first morphism is the differential
		\begin{equation} \label{eq: comparison of periods, E two differential}
				\sheafhom_{k}(H^{0} C, \Z)
			\to
				\sheafext_{k}^{2}(H^{1} C, \Z)
		\end{equation}
	of the spectral sequence with $E_{2}^{i j} = \sheafext_{k}^{i}(H^{-j} C, \Z)$.
	The second is the connecting morphism
		\[
				\sheafhom_{k}((H^{0} C) \tensor \Q / \Z, \Q / \Z)
			\to
				\sheafext_{k}^{1}(H^{1} C, \Q / \Z)
		\]
	for the short exact sequence
		\begin{equation} \label{eq: comparison of periods, connecting for derived tensor}
				0
			\to
				(H^{0} C) \tensor \Q / \Z
			\to
				H^{0}(C \tensor^{L} \Q / \Z)
			\to
				H^{1} C
			\to
				0.
		\end{equation}
	We want to show that these morphisms are compatible under the natural morphism
		\[
				\sheafhom_{k}(H^{0} C, \Z)
			\to
				\sheafhom_{k}((H^{0} C) \tensor \Q / \Z, \Q / \Z)
		\]
	and the isomorphism
		\[
				\sheafext_{k}^{1}(H^{1} C, \Q / \Z)
			=
				\sheafext_{k}^{2}(H^{1} C, \Z).
		\]
	
	By general nonsense,
	\eqref{eq: comparison of periods, E two differential}
	comes from the (shifted) connecting morphism
	$(H^{1} C)[-2] \to H^{0} C$
	of the truncation distinguished triangle
	$H^{0} C \to \tau_{\le 1} C \to (H^{1} C)[-1]$.
	This triangle induces a distinguished triangle
		\[
				(H^{0} C) \tensor^{L} \Q / \Z
			\to
				(\tau_{\le 1} C) \tensor^{L} \Q / \Z
			\to
				(H^{1} C) \tensor^{L} \Q / \Z[-1].
		\]
	Taking $H^{0}$, we have an exact sequence
		\[
				0
			\to
				(H^{0} C) \tensor \Q / \Z
			\to
				H^{0}((\tau_{\le 1} C) \tensor^{L} \Q / \Z)
			\to
				H^{1} C
			\to
				0,
		\]
	which recovers \eqref{eq: comparison of periods, connecting for derived tensor}.
	From these observations,
	we can finish the comparison
	by applying $\sheafext_{k}^{\var}(\var, \Z)$.
\end{proof}

\begin{Prop} \label{prop: double dual dies after tensoring Q mod Z or adeles}
	We have
		\begin{gather*}
					R \invlim_{n} (C \tensor^{L} \Z / n \Z)
				=
					R \invlim_{n} (C^{\SDual \SDual} \tensor^{L} \Z / n \Z)
			\\
					C \ctensor \Adele^{\infty}
				=
					C^{\SDual \SDual} \ctensor \Adele^{\infty}.
		\end{gather*}
	The same is true for $D$.
\end{Prop}

\begin{proof}
	By Prop.\ \ref{prop: each term of duality for Neron},
	we have a distinguished triangle
		\[
				V (H^{0} C)_{0}[1]
			\to
				C
			\to
				C^{\SDual \SDual}.
		\]
	We have
		\[
				R \invlim_{n} \bigl(
					V (H^{0} C)_{0} \tensor^{L} \Z / n \Z
				\bigr)
			=
				0
		\]
	since $V (H^{0} C)_{0}$ is uniquely divisible.
	This implies the result.
\end{proof}

\begin{Prop} \label{prop: each part of Q tensor duality}
	The cohomology objects $H^{n}$ of the complex
	$D^{\SDual \SDual} \tensor \Q$
	are described as follows:
		\[
					H^{-1}
				=
					V (H^{0} D)_{0},
			\quad
					H^{0}
				=
					(\pi_{0}H^{0} D)_{/ \tor} \tensor \Q,
		\]
	and $H^{n} = 0$ for other values of $n$.
	The cohomology objects $H^{n}$ of the complex
	$C^{\SDual} \tensor \Q$
	are described as follows:
		\[
					H^{-1}
				=
					(V H^{2} C)^{\PDual},
			\quad
					0
				\to
					(V H^{1} C)^{\PDual}
				\to
					H^{0}
				\to
					( \pi_{0} H^{0} C)_{/ \tor}^{\LDual} \tensor \Q
				\to
					0,
		\]
	and $H^{n} = 0$ for other values of $n$.
\end{Prop}

\begin{proof}
	Tensor $\Q$ with the isomorphisms and exact sequences in
	Prop.\ \ref{prop: each term of duality for Neron}.
\end{proof}

\begin{Prop} \label{prop: description of duality morphism tensor Q}
	The morphism
		\[
				D^{\SDual \SDual} \tensor \Q
			\to
				C^{\SDual} \tensor \Q
		\]
	induces an injection onto a direct summand in degree $0$
	with cokernel $(V H^{1} C)^{\PDual}$
	and isomorphisms on cohomology in other degrees.
\end{Prop}

\begin{proof}
	From degree zero, we have a morphism
		\[
				(\pi_{0} H^{0} D)_{/ \tor} \tensor \Q
			\to
				(\pi_{0} H^{0} C)_{/ \tor}^{\LDual} \tensor \Q.
		\]
	This is an isomorphism since $\delta_{\Height}$ is finite.
	Together with the previous proposition, the statement about degree zero follows.
	For degree $-1$, consider the commutative diagram
		\[
			\begin{CD}
					D^{\SDual \SDual} \tensor \Q
				@>>>
					C^{\SDual} \tensor \Q
				\\
				@VVV
				@VVV
				\\
					D^{\SDual \SDual} \ctensor \Adele^{\infty}
				@>>>
					C^{\SDual} \ctensor \Adele^{\infty}.
			\end{CD}
		\]
	The vertical morphisms induce isomorphisms in degree $-1$
	by the previous proposition and Prop.\ \ref{prop: adelic duality for Neron}.
	The last mentioned proposition also shows that
	the lower horizontal morphism induces an isomorphism in degree $-1$.
	Hence so is the upper morphism.
\end{proof}

In particular, we have two canonical morphisms
	\[
			D^{\SDual \SDual} \tensor \Q
		\to
			C^{\SDual} \tensor \Q
		\to
			(V H^{1} C)^{\PDual}.
	\]

\begin{Prop} \label{prop: duality triangle tensor Q}
	There exists a unique morphism
		\[
				(V H^{1} C)^{\PDual}[-1]
			\to
				D^{\SDual \SDual} \tensor \Q
		\]
	such that the resulting triangle
		\[
				(V H^{1} C)^{\PDual}[-1]
			\to
				D^{\SDual \SDual} \tensor \Q
			\to
				C^{\SDual} \tensor \Q
			\to
				(V H^{1} C)^{\PDual}
		\]
	is distinguished.
\end{Prop}

\begin{proof}
	The existence is clear from the previous proposition.
	The uniqueness follows from the following general lemma below.
\end{proof}

\begin{Lem}
	If $E \to F \to G$ are two morphisms in the derived category of an abelian category
	such that $E$ (resp.\ $G$) is concentrated in non-positive (resp.\ non-negative) degrees
	and if there exists a morphism $G[-1] \to E$ that yields a distinguished triangle
	$G[-1] \to E \to F \to G$,
	then such a morphism $G[-1] \to E$ is unique.
\end{Lem}

\begin{proof}
	Let $f, g \colon G[-1] \rightrightarrows E$ be two morphisms
	such that the two triangles $G[-1] \rightrightarrows E \to F \to G$ are both distinguished.
	By an axiom of triangulated category,
	there exists an automorphism $h$ on $G$ such that the diagram
		\[
			\begin{CD}
				E @>>> F @>>> G @>> f > E[1] \\
				@| @| @VV h V @| \\
				E @>>> F @>>> G @> g >> E[1]
			\end{CD}
		\]
	is commutative.
	Hence the diagram
		\[
			\begin{CD}
				E @>>> F @>>> G @>> f > E[1] \\
				@VV 0 V @VV 0 V @VV h - 1 V @VV 0 V \\
				E @>>> F @>>> G @> g >> E[1]
			\end{CD}
		\]
	is commutative.
	Therefore there exists a morphism $r \colon E[1] \to G$ such that
	$h - 1 = r \compose f$.
	But the assumptions on $E$ and $G$ imply that $\Hom(E[1], G) = 0$.
	Hence $r = 0$, $h = 1$ and thus $f = g$.
\end{proof}

\begin{Prop} \label{prop: each part of inverse by multiplication duality}
	The cohomology objects $H^{n}$ of the complex
	$R \invlim_{n} D^{\SDual \SDual}$
	are described as follows:
		\[
					0
				\to
						(\pi_{0} H^{0} D)_{/ \tor}
					\tensor
						\Adele^{\infty} / \Q
				\to
					H^{1}
				\to
					V H^{1} D
				\to
					0,
			\quad
					H^{2}
				=
					V H^{2} D,
		\]
	and $H^{n} = 0$ for other values of $n$.
	The cohomology objects $H^{n}$ of the complex
	$R \invlim_{n} C^{\SDual}$
	are described as follows:
		\[
					H^{1}
				=
						(\pi_{0} H^{0} C)_{/ \tor}^{\LDual}
					\tensor
						\Adele^{\infty} / \Q,
			\quad
					H^{2}
				=
					(V (H^{0} C)_{0})^{\PDual},
		\]
	and $H^{n} = 0$ for other values of $n$.
\end{Prop}

\begin{proof}
	For $R \invlim_{n} D^{\SDual \SDual}$,
	apply $R \invlim_{n}$ to the groups in Prop.\ \ref{prop: each term of duality for Neron}
	and use Prop.\ \ref{prop: inverse limit of multiplication for divisibly ML}
	and \ref{prop: inverse limit of multiplication for finite gen etale}.
	
	For $R \invlim_{n} C^{\SDual}$, we have
		\begin{align*}
					R \invlim_{n}
						R \sheafhom_{k} \bigl(
							C, \Z
						\bigr)
			&	=
					R \sheafhom_{k} \bigl(
						C \tensor \Q, \Z
					\bigr)
			\\
			&	=
					R \sheafhom_{k} \bigl(
						H^{0} C \tensor \Q, \Z
					\bigr)
			\\
			&	=
					R \invlim_{n}
						R \sheafhom_{k} \bigl(
							H^{0} C, \Z
						\bigr)
		\end{align*}
	by \cite[Prop.\ (2.3.3) (c)]{Suz14} and
	the torsionness result of higher cohomology in
	Prop.\ \ref{prop: each part of Q mod Z duality for Neron models}.
	The cohomology objects $H^{n}$ of the complex
		$
			R \sheafhom_{k} \bigl(
				H^{0} C, \Z
			\bigr)
		$
	are
		\[
				H^{0}
			=
				(\pi_{0} H^{0} C)_{/ \tor}^{\LDual},
			\quad
				H^{1}
			=
				(\pi_{0} H^{0} C)_{\tor}^{\PDual},
			\quad
				H^{2}
			=
				(T (H^{0} C)_{0})^{\PDual}
		\]
	by the same argument as the proof of Prop.\ \ref{prop: each term of duality for Neron}.
	Applying $R \invlim_{n}$ to these groups and
	using Prop.\ \ref{prop: inverse limit of multiplication for divisibly ML}
	and \ref{prop: inverse limit of multiplication for finite gen etale},
	we get the desired description.
\end{proof}

\begin{Prop} \label{prop: description of invserse multplication duality morphism}
	The morphism
		\[
				R \invlim_{n}
					D^{\SDual \SDual}
			\to
				R \invlim_{n}
					C^{\SDual}
		\]
	is a surjection onto a direct summand on cohomology in degree $1$
	with kernel $V H^{1} D$
	and an isomorphism on cohomology in any other degree.
\end{Prop}

\begin{proof}
	The same argument as the proof of Prop.\ \ref{prop: description of duality morphism tensor Q} works,
	this time using the isomorphism
		\[
					(\pi_{0} H^{0} D)_{/ \tor}
				\tensor
					\Adele^{\infty} / \Q
			\isomto
					(\pi_{0} H^{0} C)_{/ \tor}^{\LDual}
				\tensor
					\Adele^{\infty} / \Q
		\]
	and the commutative diagram
		\[
			\begin{CD}
					D^{\SDual \SDual} \ctensor \Adele^{\infty}
				@>>>
					C^{\SDual} \ctensor \Adele^{\infty}.
				\\
				@VVV
				@VVV
				\\
					R \invlim_{n}
						D^{\SDual \SDual}[1]
				@>>>
					R \invlim_{n}
						C^{\SDual}[1].
			\end{CD}
		\]
\end{proof}

In particular, we have two canonical morphisms
	\[
			V H^{1} D[-1]
		\to
			R \invlim_{n}
				D^{\SDual \SDual}
		\to
			R \invlim_{n}
				C^{\SDual}
	\]

\begin{Prop} \label{prop: duality triangle after inverse multiplication}
	There exists a unique morphism
		\[
				R \invlim_{n}
					C^{\SDual}
			\to
				V H^{1} D
		\]
	such that the resulting triangle
		\[
				V H^{1} D[-1]
			\to
				R \invlim_{n}
					D^{\SDual \SDual}
			\to
				R \invlim_{n}
					C^{\SDual}
			\to
				V H^{1} D
		\]
	is distinguished.
\end{Prop}

\begin{proof}
	Similar to the proof of Prop.\ \ref{prop: duality triangle tensor Q}.
\end{proof}

\begin{Prop} \label{prop: uniqueness of the mapping cone of the duality}
	Consider the natural commutative diagram
		\[
			\begin{CD}
					R \invlim_{n}
						D^{\SDual \SDual}
				@>>>
					R \invlim_{n}
						C^{\SDual}
				@>>>
					V H^{1} D
				\\
				@VVV
				@VVV
				@.
				\\
					D^{\SDual \SDual} \tensor \Q
				@>>>
					C^{\SDual} \tensor \Q
				@>>>
					(V H^{1} C)^{\PDual},
			\end{CD}
		\]
	where the rows are distinguished.
	The isomorphism
	$V H^{1} D \isomto (V H^{1} C)^{\PDual}$
	in Prop.\ \ref{prop: adelic duality for Neron} is the unique morphism
	that completes the above diagram into a morphism of distinguished triangles.
\end{Prop}

\begin{proof}
	A morphism $V H^{1} D \to (V H^{1} C)^{\PDual}$
	with the required property exists by an axiom of triangulated categories.
	We need to show that such a morphism has to be the isomorphism of
	Prop.\ \ref{prop: adelic duality for Neron}.
	
	Consider the morphism of distinguished triangles
		\[
			\begin{CD}
					R \invlim_{n}
						D^{\SDual \SDual}
				@>>>
					D^{\SDual \SDual} \tensor \Q
				@>>>
					D^{\SDual \SDual} \ctensor \Adele^{\infty}
				\\
				@VVV
				@VVV
				@V \wr VV
				\\
					R \invlim_{n}
						C^{\SDual}
				@>>>
					C^{\SDual} \tensor \Q
				@>>>
					C^{\SDual} \ctensor \Adele^{\infty}
			\end{CD}
		\]
	coming from Prop.\ \ref{prop: direct inverse adelic rationalization triangle},
	\ref{prop: adelic duality for Neron}
	and \ref{prop: double dual dies after tensoring Q mod Z or adeles}.
	Denote the upper triangle by $E \to F \to G$ and
	the lower by $E' \to F' \to G'$.
	By Prop.\ \ref{prop: each part of adelic duality},
	\ref{prop: each part of Q tensor duality} and
	\ref{prop: each part of inverse by multiplication duality},
	we have a commutative diagram with exact rows
		\[
			\begin{CD}
					0
				@>>>
					H^{0} F
				@>>>
					H^{0} G
				@>>>
					H^{1} E
				@>>>
					0
				\\
				@.
				@VVV
				@VV \wr V
				@VVV
				@.
				\\
					0
				@>>>
					H^{0} F'
				@>>>
					H^{0} G'
				@>>>
					H^{1} E'
				@>>>
					0.
			\end{CD}
		\]
	Hence any morphism $V H^{1} D \to (V H^{1} C)^{\PDual}$
	with the required property has to be the connecting morphism
		\[
				\Ker(H^{1} E \to H^{1} E')
			\to
				\Coker(H^{0} F \to H^{0} F')
		\]
	of the snake lemma for this diagram.
	Using Prop.\ \ref{prop: adelic duality for Neron},
	\ref{prop: description of duality morphism tensor Q}
	and \ref{prop: description of invserse multplication duality morphism},
	we can see that this is indeed the isomorphism of Prop.\ \ref{prop: adelic duality for Neron}.
\end{proof}

By Prop.\ \ref{prop: duality triangle tensor Q} and \ref{prop: duality triangle after inverse multiplication},
we have morphisms
	\begin{gather}
		\label{eq: canonical morphism to the cone of the duality morphism}
				C^{\SDual}
			\to
				C^{\SDual} \tensor \Q
			\to
				(V H^{1} C)^{\PDual},
		\\
		\label{eq: canonical morphism from the cone of the duality morphism}
				V H^{1} D[-1]
			\to
				R \invlim_{n}
					D^{\SDual \SDual}
			\to
				D^{\SDual \SDual}.
	\end{gather}
Together with the isomorphism
$(V H^{1} C)^{\PDual} \cong V H^{1} D$
of Prop.\ \ref{prop: adelic duality for Neron},
we have a triangle
	\[
			V H^{1} D[-1]
		\to
			D^{\SDual \SDual}
		\to
			C^{\SDual}
		\to
			V H^{1} D
	\]

\begin{Prop} \label{prop: formal duality for Neron}
	The above triangle is distinguished.
\end{Prop}

\begin{proof}
	Let $W$ be any mapping cone of the morphism $D^{\SDual \SDual} \to C^{\SDual}$.
	Consider the commutative diagram of distinguished triangles
		\[
			\begin{CD}
					R \invlim_{n} D^{\SDual \SDual}
				@>>>
					R \invlim_{n} C^{\SDual}
				@>>>
					R \invlim_{n} W
				\\
				@VVV
				@VVV
				@VVV
				\\
					D^{\SDual \SDual}
				@>>>
					C^{\SDual}
				@>>>
					W
				\\
				@VVV
				@VVV
				@VVV
				\\
					R \invlim_{n}(D^{\SDual \SDual} \tensor^{L} \Z / n \Z)
				@>>>
					R \invlim_{n}(C^{\SDual} \tensor^{L} \Z / n \Z)
				@>>>
					R \invlim_{n}(W \tensor^{L} \Z / n \Z).
			\end{CD}
		\]
	The left lower horizontal morphism is an isomorphism
	by Prop.\ \ref{prop: duality for Neron after derived completion} and
	\ref{prop: double dual dies after tensoring Q mod Z or adeles}.
	The (canonical choice of a) mapping fiber of the left upper horizontal morphism is $V H^{1} D[-1]$
	by Prop.\ \ref{prop: duality triangle after inverse multiplication}.
	Therefore $W$ can actually be taken as $V H^{1} D$,
	i.e., there exists a morphism of distinguished triangle
		\[
			\begin{CD}
					V H^{1} D[-1]
				@>>>
					R \invlim_{n} D^{\SDual \SDual}
				@>>>
					R \invlim_{n} C^{\SDual}
				\\
				@|
				@VVV
				@VVV
				\\
					V H^{1} D[-1]
				@>>>
					D^{\SDual \SDual}
				@>>>
					C^{\SDual}.
			\end{CD}
		\]
	The lower horizontal morphism in the left square has to be
	the morphism \eqref{eq: canonical morphism from the cone of the duality morphism}.
	For this choice of $W$, we show that the connecting morphism $C^{\SDual} \to V H^{1} D$ for the lower triangle
	is equal to the morphism \eqref{eq: canonical morphism to the cone of the duality morphism}
	composed with the isomorphism
	$(V H^{1} C)^{\PDual} \cong V H^{1} D$
	of Prop.\ \ref{prop: adelic duality for Neron}.
	Consider the commutative diagram
		\[
			\begin{CD}
					D^{\SDual \SDual}
				@>>>
					C^{\SDual}
				@>>>
					V H^{1} D
				\\
				@VVV
				@VVV
				@.
				\\
					D^{\SDual \SDual} \tensor \Q
				@>>>
					C^{\SDual} \tensor \Q
				@>>>
					(V H^{1} C)^{\PDual}.
			\end{CD}
		\]
	By an axiom of triangulated categories, there exists a morphism
	$f \colon V H^{1} D \to (V H^{1} C)^{\PDual}$
	that completes the diagram into a morphism of distinguished triangles.
	Hence we have a morphism of distinguished triangles
		\[
			\begin{CD}
					R \invlim_{n}
						D^{\SDual \SDual}
				@>>>
					R \invlim_{n}
						C^{\SDual}
				@>>>
					V H^{1} D
				\\
				@VVV
				@VVV
				@V f VV
				\\
					D^{\SDual \SDual} \tensor \Q
				@>>>
					C^{\SDual} \tensor \Q
				@>>>
					(V H^{1} C)^{\PDual}.
			\end{CD}
		\]
	Applying Prop.\ \ref{prop: uniqueness of the mapping cone of the duality} to this diagram,
	we know that $f$ has to be the isomorphism
	given by Prop.\ \ref{prop: adelic duality for Neron}.
	Hence the morphism
	$C^{\SDual} \to (V H^{1} C)^{\PDual}$
	has to come from the morphism
	\eqref{eq: canonical morphism to the cone of the duality morphism}.
\end{proof}

Suppose that we have two other objects $C', D' \in D(k)$
and a morphism $C' \tensor^{L} D' \to \Z$
satisfying the same assumptions
\eqref{item: structure of cohom of C}--\eqref{item: mod n duality for C and D}
listed at the beginning of this subsection.
Suppose also that we have a perfect pairing $C'' \times D'' \to \Q / \Z$
of finite \'etale groups over $k$.
Suppose finally that we have two distinguished triangles
$C' \to C \to C''$ and $D \to D' \to D''$ such that
the morphisms
	\[
		\begin{CD}
			D @>>> D' @>>> D'' \\
			@VVV @VVV @V \wr VV \\
			C^{\SDual} @>>> C'^{\SDual} @>>> C''^{\PDual},
		\end{CD}
	\]
form a morphism of distinguished triangles.

\begin{Prop} \label{prop: formal component group compatibility diagram}
	The morphism $D \to D'$ induces an isomorphism
	$V H^{1} D \isomto V H^{1} D'$.
	The diagram
		\[
			\begin{CD}
				D^{\SDual \SDual} @>>> D'^{\SDual \SDual} @>>> D'' \\
				@VVV @VVV @V \wr VV \\
				C^{\SDual} @>>> C'^{\SDual} @>>> C''^{\PDual} \\
				@VVV @VVV @VVV \\
				V H^{1} D @> \sim >> V H^{1} D' @>>> 0
			\end{CD}
		\]
	is a commutative diagram of distinguished triangles.
\end{Prop}

\begin{proof}
	Since $C''$ and $D''$ are finite,
	the morphism $H^{1} D \to H^{1} D'$ is surjective with finite kernel.
	Hence $V H^{1} D \isomto V H^{1} D'$.
	For the commutativity, the only part to check is hidden in the above diagram:
	the composite $V H^{1} D'[-1] \to D'^{\SDual \SDual} \to D''$ is zero;
	and the composite $C''^{\PDual}[-1] \to C^{\SDual} \to V H^{1} D$ is zero.
	These are obvious since there are no non-zero morphisms between
	finite groups and uniquely divisible groups with any shift.
\end{proof}


\subsection{Main theorem}
\label{sec: Main theorem}

Now we apply the results of the previous subsection to $C = R \alg{\Gamma}(X, \mathcal{A})$
and $D = R \alg{\Gamma}(X, \mathcal{A}_{0}^{\vee})$.
Statement \eqref{item: main thm: distinguished triangle}
in the following theorem proves Thm.\ \ref{mainthm} in Introduction.

\begin{Thm} \label{thm: duality for Neron} \BetweenThmAndList
	\begin{enumerate}
		\item \label{item: main thm: cohomology representable, zero above degree two}
			We have $\alg{H}^{n}(X, \mathcal{A}) \in \Loc \Alg / k$ for any $n$
			and $\alg{H}^{n}(X, \mathcal{A}) = 0$ for $n \ne 0, 1, 2$.
		\item \label{item: main thm: structure of each part}
			\begin{itemize}
				\item
					$\alg{\Gamma}(X, \mathcal{A})_{0}$ is an abelian variety.
				\item
					$\pi_{0} \alg{\Gamma}(X, \mathcal{A}) \in \FGEt / k$.
				\item
					$\alg{H}^{1}(X, \mathcal{A})_{0}$ is unipotent quasi-algebraic.
				\item
					$\pi_{0} \alg{H}^{1}(X, \mathcal{A}) \in \Et / k$ is torsion of cofinite type.
				\item
					$\alg{H}^{1}(X, \mathcal{A})_{\divis} \in \Et / k$ is torsion of cofinite type.
				\item
					$\alg{H}^{1}(X, \mathcal{A})_{/ \divis} \in \Alg_{\uc} / k$.
				\item
					$\alg{H}^{2}(X, \mathcal{A}) \in \Et / k$ is divisible torsion of cofinite type.
			\end{itemize}
		\item \label{item: main thm: dual structure statement}
			The same statements as \eqref{item: main thm: cohomology representable, zero above degree two}
			and \eqref{item: main thm: structure of each part}
			hold with $\mathcal{A}$ replaced by $\mathcal{A}_{0}^{\vee}$.
		\item \label{item: main thm: first cohom is Tate Shafarevich}
			The group of $k'$-valued points of $\alg{H}^{1}(X, \mathcal{A})$ for $k' \in k^{\ind\rat}$ is given by
			$\Sha(\mathcal{A}_{\closure{k'}} / X_{\closure{k'}})^{G_{k'}}$
			in the notation of the paragraph before Prop.\ \ref{prop: first cohom is Tate Shafarevich}.
		\item \label{item: main thm: distinguished triangle}
			There exists a canonical distinguished triangle
				\[
						R \alg{\Gamma}(X, \mathcal{A}_{0}^{\vee})^{\SDual \SDual}
					\to
						R \alg{\Gamma}(X, \mathcal{A})^{\SDual}
					\to
						V \alg{H}^{1}(X, \mathcal{A}_{0}^{\vee})_{\divis}.
				\]
			The first morphism is induced from the morphism \eqref{eq: duality pairing for Neron over X}.
		\item \label{item: main thm: duality on each part}
			This distinguished triangle induces, on cohomology,
			the following duality pairings and morphisms:
				\begin{enumerate}
					\item \label{item: duality between TGamma of A dual and H2 of A}
						Pontryagin duality between
						$T (\alg{\Gamma}(X, \mathcal{A}_{0}^{\vee})_{0})$
						and
						$\alg{H}^{2}(X, \mathcal{A})$.
					\item \label{item: duality between TGamma of A and H2 of A dual}
						Pontryagin duality between
						$T (\alg{\Gamma}(X, \mathcal{A})_{0})$
						and
						$\alg{H}^{2}(X, \mathcal{A}_{0}^{\vee})$.
					\item \label{item: duality between torsion MW of A dual and finite quotient of H1 A}
						Pontryagin duality between
						$(\pi_{0} \alg{\Gamma}(X, \mathcal{A}))_{\tor}$ and
						$\pi_{0} \alg{H}^{1}(X, \mathcal{A}_{0}^{\vee})_{/ \divis}$.
					\item \label{item: duality between torsion MW of A and finite quotient of H1 A dual}
						Pontryagin duality between
						$(\pi_{0} \alg{\Gamma}(X, \mathcal{A}_{0}^{\vee}))_{\tor}$ and
						$\pi_{0} \alg{H}^{1}(X, \mathcal{A})_{/ \divis}$.
					\item \label{item: height pairing}
						An injection
							$
									\pi_{0}(
										\alg{\Gamma}(X, \mathcal{A}_{0}^{\vee})
									)_{/ \tor}
								\into
									\pi_{0}(
										\alg{\Gamma}(X, \mathcal{A})
									)_{/ \tor}^{\LDual}
							$
						whose cokernel $\delta_{\Height}$ is finite \'etale.
					\item \label{item: transcendental pairing}
						An injection
							$
									(\pi_{0} \alg{H}^{1}(X, \mathcal{A}))_{\divis}^{\PDual}
								\into
									T \alg{H}^{1}(X, \mathcal{A}_{0}^{\vee})_{\divis}
							$
						whose cokernel $\delta_{\Tran}$ is finite \'etale.
					\item \label{item: Cassels Tate pairing}
						A surjection
							$
									\alg{H}^{1}(X, \mathcal{A})_{0}^{\SDual'}
								\onto
									(\alg{H}^{1}(X, \mathcal{A}_{0}^{\vee})_{/ \divis})_{0}
							$
						whose kernel $\delta_{\CT}$ is finite \'etale.
					\item \label{item: discriminant exact sequence}
						An exact sequence
						$0 \to \delta_{\Tran} \to \delta_{\Height} \to \delta_{\CT} \to 0$.
				\end{enumerate}
			In particular, the morphism \eqref{item: transcendental pairing} induces a Pontryagin duality
			between $V \alg{H}^{1}(X, \mathcal{A})_{\divis}$ and
			$V \alg{H}^{1}(X, \mathcal{A}_{0}^{\vee})_{\divis}$.
		\item \label{item: main thm: agrees with height pairing}
			The morphism \eqref{item: height pairing} agrees with the height pairing.
		\item \label{item: main thm: component groups compatibility diagram}
			The distinguished triangle \eqref{item: main thm: distinguished triangle}
			and the corresponding triangle with $A, A^{\vee}$ switched
			fit in the commutative diagram of distinguished triangles
				\[
					\begin{CD}
							R \alg{\Gamma}(X, \mathcal{A}_{0}^{\vee})^{\SDual \SDual}
						@>>>
							R \alg{\Gamma}(X, \mathcal{A})^{\SDual}
						@>>>
							V \alg{H}^{1}(X, \mathcal{A}_{0}^{\vee})_{\divis}
						\\
						@VVV @VVV @V \wr VV
						\\
							R \alg{\Gamma}(X, \mathcal{A}^{\vee})^{\SDual \SDual}
						@>>>
							R \alg{\Gamma}(X, \mathcal{A}_{0})^{\SDual}
						@>>>
							V \alg{H}^{1}(X, \mathcal{A}^{\vee})_{\divis}
						\\
						@VVV @VVV @VVV
						\\
							\bigoplus_{x}
								\Res_{k_{x} / k} \pi_{0}(\mathcal{A}_{x}^{\vee})
						@> \sim >>
							\bigoplus_{x}
								\Res_{k_{x} / k} \pi_{0}(\mathcal{A}_{x})^{\PDual}
						@>>>
							0.
					\end{CD}
				\]
			The left lower horizontal morphism is the sum over all closed points $x \in X$ of
			the Weil restrictions of Grothendieck's pairings,
			which is an isomorphism \cite[Thm.\ C]{Suz14}.
	\end{enumerate}
\end{Thm}

\begin{proof}
	\eqref{item: main thm: cohomology representable, zero above degree two}
	and \eqref{item: main thm: structure of each part} follow from
	Prop.\ \ref{prop: results before the formal steps}
	and \ref{prop: each part of Q mod Z duality for Neron models}.
	\eqref{item: main thm: dual structure statement} follows from
	Prop.\ \ref{prop: cohomology of connected Neron and Neron}.
	\eqref{item: main thm: first cohom is Tate Shafarevich} is
	Prop.\ \ref{prop: first cohom is Tate Shafarevich}.
	\eqref{item: main thm: distinguished triangle} follows from
	Prop.\ \ref{prop: formal duality for Neron}.
	\eqref{item: main thm: duality on each part} follows from
	Prop.\ \ref{prop: each part of Q mod Z duality for Neron models}
	and \ref{prop: results before the formal steps}.
	\eqref{item: main thm: agrees with height pairing} is
	Prop.\ \ref{prop: our duality contains height pairing}.
	\eqref{item: main thm: component groups compatibility diagram} follows from
	Prop.\ \ref{prop: component group compatibitily morphism}
	and \ref{prop: formal component group compatibility diagram}.
\end{proof}

Several comments are in order.
We consider \eqref{item: Cassels Tate pairing} as a geometric analogue of the Cassels-Tate pairing
in view of Prop.\ \ref{prop: agrees with the Cassels Tate} below,
whence the symbol $\delta_{\CT}$.
The divisible part of the Tate-Shafarevich group when $k$ is algebraically closed
is of transcendental nature,
whence the symbol $\delta_{\Tran}$.
Confusingly, the morphisms in
\eqref{item: transcendental pairing} and \eqref{item: Cassels Tate pairing}
are \emph{from} the duals, not \emph{to}.

The exact sequence \eqref{item: discriminant exact sequence} is mysterious.
It came from applying the formal procedure of Lemma \ref{lem: discriminant exact sequence}
to the two exact sequences
\eqref{eq: source of discriminant ex seq, mapping from MW}
and \eqref{eq: source of discriminant ex seq, mapping to MW}
(where one should note that the central term $H^{0} X$ in the notation there
is the ``Selmer scheme'' explained below).
The morphism $\delta_{\Height} \onto \delta_{\CT}$ is induced by the morphism
	\[
			\pi_{0}(
				\alg{\Gamma}(X, \mathcal{A})
			)_{/ \tor}^{\LDual}
		\to
			\alg{H}^{1}(X, \mathcal{A})_{0}^{\SDual'}
	\]
of Prop.\ \ref{prop: each term of duality for Neron} and
\ref{prop: period map and its completion}
(i.e.\ the latter morphism factors through $\delta_{\Height} \onto \delta_{\CT}$).
This morphism is analogous to Artin's \emph{period map} \cite{Art74} for supersingular K3 surfaces.

Milne \cite[III, paragraph before Thm.\ 11.6]{Mil06} made the hypothesis that
the sheaf $R^{2} \pi_{X}^{\perf} \mathcal{A}$ on $\Spec k^{\perf}_{\et}$ has no connected part.
This sheaf restricted to $\Spec k^{\ind\rat}_{\et}$ is $\alg{H}^{2}(X, \mathcal{A})$
by Prop.\ \ref{prop: comparison with Artin Milne}.
Hence the result $\alg{H}^{2}(X, \mathcal{A}) \in \Et / k$ in
\eqref{item: main thm: structure of each part} above
says that his hypothesis is true at least ``birationally''.

Also \cite[III, Thm.\ 11.6]{Mil06}
(plus the fact that $H^{2}(K \tensor \closure{k}, N) = 0$ for finite flat $N$
and hence $H^{2}(K \tensor \closure{k}, A) = 0$) says that
$\alg{H}^{2}(X, \mathcal{A})$ injects into the Pontryagin dual of
$T (\alg{\Gamma}(X, \mathcal{A}_{0}^{\vee})_{0})$.
Thus \eqref{item: duality between TGamma of A dual and H2 of A} above
shows that this injection is actually bijective.

We have an exact sequence
	\[
			0
		\to
			\pi_{0}(\alg{\Gamma}(X, \mathcal{A}))_{/ \tor} \tensor \Q / \Z
		\to
			H^{0} \bigl(
				R \alg{\Gamma}(X, \mathcal{A}) \tensor^{L} \Q / \Z
			\bigr)
		\to
			\alg{H}^{1}(X, \mathcal{A})
		\to
			0
	\]
in $\Loc^{f} \Alg_{\uc} / k$.
The first term is the Mordell-Weil group tensored with $\Q / \Z$.
In Introduction, we called the last term the Tate-Shafarevich scheme.
This terminology is justified by \eqref{item: main thm: first cohom is Tate Shafarevich}.
Along this line, the middle term might be called the \emph{Selmer scheme}.

\begin{Rmk} \label{rmk: connected divisible part}
	As we saw after Lem.\ \ref{lem: discriminant exact sequence},
	if the finite \'etale $p$-groups $\alg{H}^{1}(X, \mathcal{A})_{0 \cap \divis}$
	and $\alg{H}^{1}(X, \mathcal{A}_{0}^{\vee})_{0 \cap \divis}$
	(intersection of connected part and divisible part) are zero,
	then the morphisms in \eqref{item: transcendental pairing}
	and \eqref{item: Cassels Tate pairing}
	are simplified to more symmetric expressions
		\[
					\alg{H}^{1}(X, \mathcal{A})_{\divis}^{\PDual}
				\into
					T \alg{H}^{1}(X, \mathcal{A}_{0}^{\vee})_{\divis}
			\quad \text{and} \quad
					\alg{H}^{1}(X, \mathcal{A})_{0}^{\SDual'}
				\onto
					\alg{H}^{1}(X, \mathcal{A}_{0}^{\vee})_{0}.
		\]
	It is not clear whether these vanishing conditions are always satisfied or not.
	To see cases where the conditions are indeed satisfied,
	let $A$ be the Jacobian
	of a proper smooth geometrically connected curve $C$ over $K$
	with a $K$-rational point (in particular, $A \cong A^{\vee}$)
	and $\mathcal{C}$ a proper flat regular model over $X$ of $C$.
	As we saw in the proof of Prop.\ \ref{prop: Sha does not have divisible connected unipotent part},
	there exists a canonical isomorphism $H^{1}(X, \mathcal{A}) \cong H^{2}(\mathcal{C}, \Gm)$.
	Let $\alg{H}^{2}(\mathcal{C}, \Gm)$ be the (pro-)\'etale sheafification of the presheaf
	$k' \in k^{\ind\rat} \mapsto H^{2}(\mathcal{C}_{k'}, \Gm)$,
	which is locally of finite presentation.
	Then the above isomorphism extends to an isomorphism
	$\alg{H}^{1}(X, \mathcal{A}) \cong \alg{H}^{2}(\mathcal{C}, \Gm)$.
	Note that $\alg{H}^{1}(X, \mathcal{A}_{0})$ surjects onto $\alg{H}^{1}(X, \mathcal{A})$
	with finite kernel by
	Prop.\ \ref{prop: cohomology of connected Neron and Neron}.
	
	If $\mathcal{C}$ is an Artin supersingular K3 surface (hence an elliptic surface fibered over $X$),
	then \cite[(4.2), (4.4)]{Art74} shows that
	$\mathcal{C}$ is Shioda supersingular (i.e.\ satisfies the \emph{Tate conjecture}) and
	$\alg{H}^{2}(\mathcal{C}, \Gm) \cong \Ga$.
	In particular, $\alg{H}^{2}(\mathcal{C}, \Gm)_{0 \cap \divis} = 0$.
	If $\mathcal{C}$ is an Artin non-supersingular K3 surface,
	then \cite[Prop.\ 4.7, Lem.\ 2.1]{MR15} applied to the de Rham-Witt complex of $\mathcal{C}$ with $r = 1$, $j = 3$
	together with \cite[II, (5.7.6), \S 7.2 (a)]{Ill79} (or \cite[Prop.\ (4.4)]{Yui86}) shows that
	$\alg{H}^{2}(\mathcal{C}, \Gm)$ is \'etale.
	Hence again $\alg{H}^{2}(\mathcal{C}, \Gm)_{0 \cap \divis} = 0$.
	In this case, \cite[II, Prop.\ 5.2]{Ill79} shows that
	the dimension of $V_{p} \alg{H}^{2}(\mathcal{C}, \Gm)$ over $\Q_{p}$
	is $22 - 2h - \rho$,
	where $22$ is the second Betti number of $\mathcal{C}$,
	$h$ the height of the formal Brauer group of $\mathcal{C}$ and
	$\rho$ the geometric Picard number of $\mathcal{C}$.
	As soon as the inequality $\rho \le 22 - 2 h$ is strict,
	the group $\alg{H}^{2}(\mathcal{C}, \Gm)_{\divis}$ is non-zero.
	
	Similar arguments apply to the case where $\mathcal{C}$ is an abelian surface, showing that
	either $\alg{H}^{2}(\mathcal{C}, \Gm)_{0} = 0$ (non-supersingular case) or
	$\alg{H}^{2}(\mathcal{C}, \Gm)_{\divis} = 0$ (supersingular case).
	Similarly, we know that $\alg{H}^{2}(\mathcal{C}, \Gm)_{\divis} = 0$
	whenever the geometric Picard number of $\mathcal{C}$ is equal to the second Betti number of $\mathcal{C}$.
	This includes the cases of rational surfaces, ruled surface,
	Enriques surfaces (\cite[Thm.\ 4]{BM76}; we drop the condition that $C$ has a $K$-rational point)
	and quasi-elliptic surfaces (\cite[Prop.\ 12]{BM76}; we drop the condition that $C$ is smooth over $K$).
	If $\mathcal{C}$ is the product of $X$ with another curve $Y$
	such that $X$ is a supersingular elliptic curve
	and $Y$ is a genus two curve whose Jacobian is absolutely simple with $p$-rank one
	(see \cite[\S 8 Examples; Cor.\ 3]{HOMNS11} for existence of such a curve;
	again we drop the condition that $C$ has a $K$-rational point),
	then a calculation similar to the above shows that $\alg{H}^{2}(\mathcal{C}, \Gm)_{0} \cong \Ga$
	and $T_{p} \alg{H}^{2}(\mathcal{C}, \Gm)_{\divis}$ has rank four.
	It is not clear whether $\alg{H}^{2}(\mathcal{C}, \Gm)_{0 \cap \divis}$ is trivial or not in this case.
	
	If there is a case where $\alg{H}^{1}(X, \mathcal{A})_{0 \cap \divis}$ is non-zero,
	it might mean that supersingularity and non-supersingularity
	are somehow ``mixed up'' in  $\mathcal{A}$ (or $\mathcal{C}$) in an interesting way.
	Perhaps it could be the Tate conjecture that does not allow such a mixing phenomenon.
\end{Rmk}


\section{Some more results}
\label{sec: Some more results}

Let $X$ and $A$ be as in the beginning of the previous section.

\subsection{Duality for N\'eron models over open curves}
\label{sec: Duality for Neron models over open curves}

In this subsection, we set
	\[
			V
		=
			V \alg{H}^{1}(X, \mathcal{A}_{0}^{\vee})_{\divis}
		=
			(V \alg{H}^{1}(X, \mathcal{A})_{\divis})^{\PDual},
	\]
the last isomorphism coming from
Thm.\ \ref{thm: duality for Neron} \eqref{item: main thm: duality on each part}.
Hence \eqref{item: main thm: distinguished triangle} of the same theorem
can be written as the canonical distinguished triangle
	\[
			R \alg{\Gamma}(X, \mathcal{A}_{0}^{\vee})^{\SDual \SDual}
		\to
			R \alg{\Gamma}(X, \mathcal{A})^{\SDual}
		\to
			V.
	\]

Let $U \subset X$ be a dense open subscheme.
We have a pairing
	\[
				R \alg{\Gamma}(U, \mathcal{A}_{0}^{\vee})
			\tensor^{L}
				R \alg{\Gamma}_{c}(U, \mathcal{A})
		\to
			R \alg{\Gamma}_{c}(U, \Gm)[1]
		\to
			\Z,
	\]
where the last morphism is the trace morphism \eqref{eq: global trace morphism}.
This induces a morphism
	\[
			R \alg{\Gamma}(U, \mathcal{A}_{0}^{\vee})
		\to
			R \alg{\Gamma}_{c}(U, \mathcal{A})^{\SDual}.
	\]

\begin{Prop} \label{prop: localization and duality for Neron over open curve}
	We have a canonical morphism between canonical distinguished triangles
		\[
			\begin{CD}
					\displaystyle
					\bigoplus_{x \not\in U}
						R \alg{\Gamma}_{x}(\Hat{\Order}_{x}, \mathcal{A}_{0}^{\vee})
				@>>>
					R \alg{\Gamma}(X, \mathcal{A}_{0}^{\vee})
				@>>>
					R \alg{\Gamma}(U, \mathcal{A}_{0}^{\vee})
				\\
				@VVV
				@VVV
				@VVV
				\\
					\displaystyle
					\bigoplus_{x \not\in U}
						R \alg{\Gamma}(\Hat{\Order}_{x}, \mathcal{A})^{\SDual}
				@>>>
					R \alg{\Gamma}(X, \mathcal{A})^{\SDual}
				@>>>
					R \alg{\Gamma}_{c}(U, \mathcal{A})^{\SDual}.
			\end{CD}
		\]
	The left vertical morphism is an isomorphism.
\end{Prop}

\begin{proof}
	This follows from the same method as the proof of
	Prop.\ \ref{prop: mod n duality for Neron models}.
\end{proof}

\begin{Prop} \label{prop: Q mod Z duality for Neron over open curve}
	The morphisms
		\begin{align*}
						R \alg{\Gamma}(U, \mathcal{A}_{0}^{\vee})
					\tensor^{L}
						\Q / \Z
			&	\to
						R \alg{\Gamma}(U, \mathcal{A}_{0}^{\vee})^{\SDual \SDual}
					\tensor^{L}
						\Q / \Z
			\\
			&	\to
						R \alg{\Gamma}_{c}(U, \mathcal{A})^{\SDual}
					\tensor^{L}
						\Q / \Z
		\end{align*}
	are isomorphisms.
\end{Prop}

\begin{proof}
	We show that the first morphism is an isomorphism.
	We have
		\[
					R \alg{\Gamma}(X, \mathcal{A}_{0}^{\vee})
				\tensor^{L}
					\Q / \Z
			\isomto
					R \alg{\Gamma}(X, \mathcal{A}_{0}^{\vee})^{\SDual \SDual}
				\tensor^{L}
					\Q / \Z
		\]
	by the same proof as Prop.\ \ref{prop: double dual dies after tensoring Q mod Z or adeles}.
	We know that
	$R \alg{\Gamma}_{x}(\Hat{\Order}_{x}, \mathcal{A}_{0}^{\vee}) \in D^{b}(\Ind \Alg_{\uc})$
	is Serre reflexive for any $x$
	by Prop.\ \ref{prop: structure of local cohom of abelian varieties}
	and \cite[Prop.\ (2.4.1) (b)]{Suz14}.
	Hence
		\[
					R \alg{\Gamma}_{x}(\Hat{\Order}_{x}, \mathcal{A}_{0}^{\vee})
				\tensor^{L}
					\Q / \Z
			\isomto
					R \alg{\Gamma}_{x}(\Hat{\Order}_{x}, \mathcal{A}_{0}^{\vee})^{\SDual \SDual}
				\tensor^{L}
					\Q / \Z.
		\]
	Therefore the previous proposition gives the result.
	
	The composite morphism has already been shown to be an isomorphism
	essentially in the proof of Prop.\ \ref{prop: mod n duality for Neron models}.
\end{proof}

\begin{Prop}
	We have
		\begin{gather*}
					R \alg{\Gamma}(X, \mathcal{A}_{0}^{\vee})^{\SDual \SDual} \tensor \Q
				\isomto
					R \alg{\Gamma}(U, \mathcal{A}_{0}^{\vee})^{\SDual \SDual} \tensor \Q,
			\\
					R \alg{\Gamma}(X, \mathcal{A})^{\SDual} \tensor \Q
				\isomto
					R \alg{\Gamma}_{c}(U, \mathcal{A})^{\SDual} \tensor \Q.
		\end{gather*}
	They induce a canonical isomorphism between distinguished triangles
		\[
			\begin{CD}
					R \alg{\Gamma}(X, \mathcal{A}_{0}^{\vee})^{\SDual \SDual} \tensor \Q
				@>>>
					R \alg{\Gamma}(X, \mathcal{A})^{\SDual} \tensor \Q
				@>>>
					V
				\\
				@VV \wr V
				@VV \wr V
				@|
				\\
					R \alg{\Gamma}(U, \mathcal{A}_{0}^{\vee})^{\SDual \SDual} \tensor \Q
				@>>>
					R \alg{\Gamma}_{c}(U, \mathcal{A})^{\SDual} \tensor \Q
				@>>>
					V.
			\end{CD}
		\]
\end{Prop}

\begin{proof}
	An object of $\Ind \Alg_{\uc} / k$ is torsion
	since a unipotent quasi-algebraic group in positive characteristic is torsion.
	Hence Prop.\ \ref{prop: structure of local cohom of abelian varieties}
	and \ref{prop: local duality}
	show that $R \alg{\Gamma}_{x}(\Hat{\Order}_{x} / k, \mathcal{A}_{0}^{\vee})$
	and $R \alg{\Gamma}(\Hat{\Order}_{x} / k, \mathcal{A})^{\SDual}$ for any $x$
	are killed after tensored with $\Q$.
	This implies the result.
\end{proof}

From this, we have canonical morphisms
	\begin{align*}
		&
				R \alg{\Gamma}_{c}(U, \mathcal{A})^{\SDual}
			\to
				R \alg{\Gamma}_{c}(U, \mathcal{A})^{\SDual} \tensor \Q
			\to
				V
		\\
		&	\to
				R \alg{\Gamma}(X, \mathcal{A}_{0}^{\vee})^{\SDual \SDual}[1]
			\to
				R \alg{\Gamma}(U, \mathcal{A}_{0}^{\vee})^{\SDual \SDual}[1].
	\end{align*}
Here is the duality for $\mathcal{A}$ over $U$:

\begin{Prop} \label{prop: duality for Neron over open curve}
	Consider the triangle
		\[
				V[-1]
			\to
				R \alg{\Gamma}(U, \mathcal{A}_{0}^{\vee})^{\SDual \SDual}
			\to
				R \alg{\Gamma}_{c}(U, \mathcal{A})^{\SDual}
			\to
				V
		\]
	coming from the above morphisms.
	This triangle is distinguished.
	The diagram
		\[
			\begin{CD}
					R \alg{\Gamma}(X, \mathcal{A}_{0}^{\vee})^{\SDual \SDual}
				@>>>
					R \alg{\Gamma}(X, \mathcal{A})^{\SDual}
				@>>>
					V
				\\
				@VVV
				@VVV
				@|
				\\
					R \alg{\Gamma}(U, \mathcal{A}_{0}^{\vee})^{\SDual \SDual}
				@>>>
					R \alg{\Gamma}_{c}(U, \mathcal{A})^{\SDual}
				@>>>
					V
			\end{CD}
		\]
	is a morphism of distinguished triangles.
\end{Prop}

\begin{proof}
	Choose a mapping cone distinguished triangle
		\begin{equation} \label{eq: cone for duality morphism of Neron over open curve}
				R \alg{\Gamma}(U, \mathcal{A}_{0}^{\vee})^{\SDual \SDual}
			\to
				R \alg{\Gamma}_{c}(U, \mathcal{A})^{\SDual}
			\to
				E.
		\end{equation}
	Applying $(\var) \tensor \Q$ and $(\var) \tensor^{L} \Q / \Z$,
	we get a commutative diagram of distinguished triangles
		\[
			\begin{CD}
					R \alg{\Gamma}(U, \mathcal{A}_{0}^{\vee})^{\SDual \SDual}
				@>>>
					R \alg{\Gamma}_{c}(U, \mathcal{A})^{\SDual}
				@>>>
					E
				\\
				@VVV
				@VVV
				@VVV
				\\
					R \alg{\Gamma}(U, \mathcal{A}_{0}^{\vee})^{\SDual \SDual} \tensor \Q
				@>>>
					R \alg{\Gamma}_{c}(U, \mathcal{A})^{\SDual} \tensor \Q
				@>>>
					E \tensor \Q
				\\
				@VVV
				@VVV
				@VVV
				\\
					R \alg{\Gamma}(U, \mathcal{A}_{0}^{\vee})^{\SDual \SDual} \tensor^{L} \Q / \Z
				@>>>
					R \alg{\Gamma}_{c}(U, \mathcal{A})^{\SDual} \tensor^{L} \Q / \Z
				@>>>
					E \tensor^{L} \Q / \Z.
			\end{CD}
		\]
	We have $E \tensor^{L} \Q / \Z = 0$ by
	Prop.\ \ref{prop: Q mod Z duality for Neron over open curve}.
	Hence $E \cong E \tensor \Q$ and $E$ is uniquely divisible.
	By an axiom of triangulated categories,
	the distinguished triangle in the middle row is isomorphic to the distinguished triangle
		\[
				R \alg{\Gamma}(U, \mathcal{A}_{0}^{\vee})^{\SDual \SDual} \tensor \Q
			\to
				R \alg{\Gamma}_{c}(U, \mathcal{A})^{\SDual} \tensor \Q
			\to
				V
		\]
	in the previous proposition.
	Therefore the distinguished triangle
	\eqref{eq: cone for duality morphism of Neron over open curve}
	can be chosen so that $E = V$ and the diagram
		\[
			\begin{CD}
					R \alg{\Gamma}(U, \mathcal{A}_{0}^{\vee})^{\SDual \SDual}
				@>>>
					R \alg{\Gamma}_{c}(U, \mathcal{A})^{\SDual}
				@>>>
					V
				\\
				@VVV
				@VVV
				@|
				\\
					R \alg{\Gamma}(U, \mathcal{A}_{0}^{\vee})^{\SDual \SDual} \tensor \Q
				@>>>
					R \alg{\Gamma}_{c}(U, \mathcal{A})^{\SDual} \tensor \Q
				@>>>
					V
			\end{CD}
		\]
	is a morphism of distinguished triangles.
	By an axiom of triangulated categories,
	the right half of the commutative diagram in
	Prop.\ \ref{prop: localization and duality for Neron over open curve}
	can be extended to a morphism of distinguished triangles
		\[
			\begin{CD}
					R \alg{\Gamma}(X, \mathcal{A}_{0}^{\vee})^{\SDual \SDual}
				@>>>
					R \alg{\Gamma}(X, \mathcal{A})^{\SDual}
				@>>>
					V
				\\
				@VVV
				@VVV
				@VVV
				\\
					R \alg{\Gamma}(U, \mathcal{A}_{0}^{\vee})^{\SDual \SDual}
				@>>>
					R \alg{\Gamma}_{c}(U, \mathcal{A})^{\SDual}
				@>>>
					V.
			\end{CD}
		\]
	Tensoring $\Q$ has to recover the isomorphism of distinguished triangles
	in the previous proposition.
	Hence the right vertical morphism is the identity.
	Thus we have a morphism of distinguished triangles
		\[
			\begin{CD}
					V[-1]
				@>>>
					R \alg{\Gamma}(X, \mathcal{A}_{0}^{\vee})^{\SDual \SDual}
				@>>>
					R \alg{\Gamma}(X, \mathcal{A})^{\SDual}
				\\
				@|
				@VVV
				@VVV
				\\
					V[-1]
				@>>>
					R \alg{\Gamma}(U, \mathcal{A}_{0}^{\vee})^{\SDual \SDual}
				@>>>
					R \alg{\Gamma}_{c}(U, \mathcal{A})^{\SDual}.
			\end{CD}
		\]
	With these diagrams, the result follows.
\end{proof}


\subsection{Link to the finite base field case}
\label{sec: Link to the finite base field case}

In this section, we assume that the base field $k$ is the finite field $\F_{q}$ of $q$ elements.
The main reference about the finite base field case of duality for abelian varieties is
\cite[II, \S 5; III, \S 9]{Mil06}.
As in \cite[\S 10]{Suz14},
let $k_{\pro\zar}$ be the full subcategory of the category of $k$-algebras
consisting of filtered unions of finite products of copies of $k$.
For $k' \in k_{\pro\zar}$, we say that a finite family $\{k'_{i}\}$ of $k'$-algebras is a covering
if $\prod k'_{i}$ is faithfully flat over $k'$.
This defines a topology on $k_{\pro\zar}$.
We denote the resulting site by $\Spec k_{\pro\zar}$.
The identity functor defines a morphism of sites
$\Spec k^{\ind\rat}_{\pro\et} \to \Spec k_{\pro\zar}$.
By abuse of notation, we denote its pushforward functor
$D(k^{\ind\rat}_{\pro\et}) \to D(k_{\pro\zar})$
by $R \Gamma(k_{\pro\et}, \var)$,
with cohomologies $H^{n}(k_{\pro\et}, \var)$ and $H^{0} = \Gamma$.
If $C \in \Ab(k^{\ind\rat}_{\pro\et})$ is locally of finite presentation,
then $H^{n}(k_{\pro\et}, C)$ is a constant sheaf for any $n$.

We will first relate
$R \sheafhom_{k}(\var, \Q / \Z)$ ($= R \sheafhom_{k^{\ind\rat}_{\pro\et}}(\var, \Q / \Z)$)
to $R \sheafhom_{k_{\pro\zar}}(\var, \Q / \Z)$.
For any $C \in D(k^{\ind\rat}_{\pro\et})$, we have morphisms
	\begin{align*}
				R \Gamma \bigl(
					k_{\pro\et},
					R \sheafhom_{k}(C, \Q / \Z)
				\bigr)
		&	\to
				R \sheafhom_{k_{\pro\zar}} \bigl(
					R \Gamma(k_{\pro\et}, C),
					R \Gamma(k_{\pro\et}, \Q / \Z)
				\bigr)
		\\
		&	\to
				R \sheafhom_{k_{\pro\zar}} \bigl(
					R \Gamma(k_{\pro\et}, C),
					\Q / \Z
				\bigr)[-1],
	\end{align*}
where the last morphism comes from $H^{1}(k, \Q / \Z) = \Q / \Z$ and
$H^{n}(k, \Q / \Z) = 0$ for $n \ge 2$.

\begin{Prop} \label{prop: cohom of Serre dual is Pont dual of cohom}
	If $C \in \Ind \Pro \Alg / k$, then the above morphism
		\[
				R \Gamma \bigl(
					k_{\pro\et},
					R \sheafhom_{k}(C, \Q / \Z)
				\bigr)
			\to
				R \sheafhom_{k_{\pro\zar}} \bigl(
					R \Gamma(k_{\pro\et}, C),
					\Q / \Z
				\bigr)[-1]
		\]
	is an isomorphism.
\end{Prop}

\begin{proof}
	The same proof as \cite[Prop.\ (10.4)]{Suz14} works.
\end{proof}

We denote $(\var)^{\PDual} = R \sheafhom_{k_{\pro\zar}}(\var, \Q / \Z)$.
If $C \in \Ab(k^{\ind\rat}_{\pro\et})$ is
an extension of a torsion (constant) group by a profinite group,
then $C^{\PDual} = \sheafhom_{k_{\pro\zar}}(C, \Q / \Z)$ is the usual Pontryagin dual.
It follows that if $C = (\prod \Z / l \Z) / (\bigoplus \Z / l \Z)$,
where $l$ ranges over all primes, then $C^{\PDual} \cong C[-1]$.

By abuse of notation, we denote the composite of
	\[
				R \alg{\Gamma}(X, \var)
			\colon
				D(X_{\fppf})
			\to
				D(k^{\ind\rat}_{\pro\et})
	\]
and
	\[
				R \Gamma(k_{\pro\et}, \var)
			\colon
				D(k^{\ind\rat}_{\pro\et})
			\to
				D(k_{\pro\zar})
	\]
by $R \Gamma(X, \var)$,
with cohomologies $H^{n}(X, \var)$ and $H^{0} = \Gamma$.
If $F \in \Ab(X_{\fppf})$ is locally of finite presentation,
then $H^{n}(X, F)$ is a constant sheaf for any $n$.

Now let $F$ be $\mathcal{A}$, the N\'eron model of an abelian variety $A$ as before.
We will deduce the duality for $R \Gamma(X, \mathcal{A})$ from
Prop.\ \ref{prop: Q mod Z duality for Neron models}.
The group $\Gamma(X, \mathcal{A})$ is the Mordell-Weil group of $A$, which is finitely generated.
We have an exact sequence
	\[
			0
		\to
			\Sha(A / K)
		\to
			H^{1}(X, \mathcal{A})
		\to
			\bigoplus_{x}
				H^{1}(k_{x}, \pi_{0}(\mathcal{A}_{x}))
	\]
by \cite[III, Prop.\ 9.2]{Mil06}.
In particular, $H^{1}(X, \mathcal{A})$ is a torsion group of cofinite type,
	\[
			H^{1}(X, \mathcal{A})_{\divis}
		=
			\Sha(A / K)_{\divis},
	\]
and the sequence
	\[
			0
		\to
			\Sha(A / K)_{/ \divis}
		\to
			H^{1}(X, \mathcal{A})_{/ \divis}
		\to
			\bigoplus_{x}
				H^{1}(k_{x}, \pi_{0}(\mathcal{A}_{x}))
	\]
is exact.
It seems that the finiteness of $\Sha(A / K)_{/ \divis}$ is not known unconditionally in general.
Hence we cannot assume that $H^{1}(X, \mathcal{A})$ is divisibly ML.
By \cite[II, Prop.\ 5.1 (a); III, Lem.\ 7.10 (c)]{Mil06},
the group $H^{2}(X, \mathcal{A})$ is torsion and
$H^{n}(X, \mathcal{A}) = 0$ for $n \ge 3$.

\begin{Prop} \label{prop: Q mod Z duality for Neron over finite base}
	Applying $R \Gamma(k_{\pro\et}, \var)$ to the second isomorphism in
	Prop.\ \ref{prop: Q mod Z duality for Neron models} gives an isomorphism
		\[
				R \invlim_{n} \bigl(
					R \Gamma(X, \mathcal{A}_{0}^{\vee}) \tensor^{L} \Z / n \Z
				\bigr)
			\isomto
				\bigl(
					R \Gamma(X, \mathcal{A}) \tensor^{L} \Q / \Z
				\bigr)^{\PDual}[-1].
		\]
\end{Prop}

\begin{proof}
	This follows from the previous proposition.
\end{proof}

The morphism in this proposition can alternatively be defined by the morphisms
	\[
				R \Gamma(X, \mathcal{A}_{0}^{\vee})
			\tensor^{L}
				R \Gamma(X, \mathcal{A})
		\to
			R \Gamma(X, \Gm)[1]
		\to
			H^{3}(X, \Gm)[-2]
		\cong
			\Q / \Z[-2].
	\]
The proposition says that this is a perfect pairing up to uniquely divisibles.
This more or less recovers \cite[III, Thm.\ 9.4]{Mil06}.
We will state an integral version using Weil-\'etale cohomology at the end of this subsection.

We will express the conjectural finiteness of $\Sha(A / K)$
in terms of $R \alg{\Gamma}(X, \mathcal{A})$.
Let $T H^{1}(X, \mathcal{A}_{0}^{\vee})_{\divis}$ be
the profinite Tate module of $H^{1}(X, \mathcal{A}_{0}^{\vee})_{\divis}$.
Let $V H^{1}(X, \mathcal{A}_{0}^{\vee})_{\divis}$ be
$(T H^{1}(X, \mathcal{A}_{0}^{\vee})_{\divis}) \tensor \Q$.
We need the following variant of \cite[III, Thm.\ 9.4 (b)]{Mil06} (cf.\ \cite[Prop.\ 2.4 (3)]{KT03}),
which we deduce from the above proposition.

\begin{Prop} \label{prop: duality for H two in finite base field case}
	In $H^{0}$, the isomorphism in Proposition \ref{prop: Q mod Z duality for Neron over finite base}
	induces an exact sequence
		\[
				0
			\to
				\Gamma(X, \mathcal{A}_{0}^{\vee})^{\wedge}
			\to
				H^{2}(X, \mathcal{A})^{\PDual}
			\to
				T H^{1}(X, \mathcal{A}_{0}^{\vee})_{\divis}
			\to
				0.
		\]
	In particular, $H^{2}(X, \mathcal{A})$ is divisibly ML and of cofinite type.
\end{Prop}

\begin{proof}
	Let $C$ be the isomorphic object in Prop.\ \ref{prop: Q mod Z duality for Neron over finite base}.
	We have a distinguished triangle
		\[
				R \invlim_{n} \Bigl(
					\bigl(
						\tau_{\le 1}
						R \Gamma(X, \mathcal{A}_{0}^{\vee}) \tensor^{L} \Z / n \Z
					\bigr)
				\Bigr)
			\to
				C
			\to
				R \invlim_{n} \bigl(
					H^{2}(X, \mathcal{A}_{0}^{\vee}) \tensor^{L} \Z / n \Z
				\bigr)[-2].
		\]
	The third term is concentrated in degrees $\ge 1$.
	Hence the first and second terms have the same $H^{0}$.
	Since $\Gamma(X, \mathcal{A})$ is finitely generated
	and $H^{1}(X, \mathcal{A})$ is torsion of cofinite type,
	we can apply Prop.\ \ref{prop: cohomologies of derived completion}
	(or its version for $\Spec k_{\pro\zar}$) to obtain an exact sequence
		\[
				0
			\to
				\Gamma(X, \mathcal{A}_{0}^{\vee})^{\wedge}
			\to
				H^{0} C
			\to
				T H^{1}(X, \mathcal{A}_{0}^{\vee})_{\divis}
			\to
				0.
		\]
	On the other hand, the cohomology objects of $R \Gamma(X, \mathcal{A}) \tensor^{L} \Q / \Z$
	are torsion (constant) groups,
	with $H^{1}$ given by $H^{2}(X, \mathcal{A})$.
	Hence $H^{0} C = H^{2}(X, \mathcal{A})^{\PDual}$.
	This gives the stated exact sequence.
	Since $\Gamma(X, \mathcal{A})$ is finitely generated
	and $H^{1}(X, \mathcal{A})$ is of cofinite type,
	it follows that $H^{2}(X, \mathcal{A})$ is divisibly ML and of cofinite type.
\end{proof}

We will describe $H^{n}(k_{\pro\et}, V \alg{H}^{1}(X, \mathcal{A})_{\divis})$ for each $n$
in terms of $H^{1}(X, \mathcal{A})$.
Note that for a finite \'etale group $C$ over $k$,
the object $R \Gamma(k_{\pro\et}, C)$ is the mapping fiber of the endomorphism $F - 1$ on $C$,
where $F$ is the $q$-th power Frobenius morphism.
Taking limits, we know that the same is true for an ind-profinite-\'etale group $C$.
Since $\alg{H}^{1}(X, \mathcal{A})_{\divis}$ is a torsion \'etale group of cofinite type
by Thm.\ \ref{thm: duality for Neron} \eqref{item: main thm: structure of each part},
we know that $R \Gamma(k_{\pro\et}, V \alg{H}^{1}(X, \mathcal{A})_{\divis})$
is the mapping fiber of $F - 1$ on $V \alg{H}^{1}(X, \mathcal{A})_{\divis}$.

\begin{Prop} \label{prop: rational points of transcendental part}
	The natural morphism
		\begin{equation} \label{eq: Frob invariant of V H one div}
				V H^{1}(X, \mathcal{A})_{\divis}
			\to
				\Gamma(k_{\pro\et}, V \alg{H}^{1}(X, \mathcal{A})_{\divis})
		\end{equation}
	is an isomorphism.
	There exists a canonical exact sequence
		\begin{equation} \label{eq: Frob coinvariant of V H one div}
				0
			\to
				\frac{
					H^{1}(X, \mathcal{A})_{/ \divis}^{\wedge}
				}{
					H^{1}(X, \mathcal{A})_{/ \divis}
				}
			\to
				H^{1}(k_{\pro\et}, V \alg{H}^{1}(X, \mathcal{A})_{\divis})
			\to
				\bigl(
					V H^{1}(X, \mathcal{A}_{0}^{\vee})_{\divis}
				\bigr)^{\PDual}
			\to
				0.
		\end{equation}
\end{Prop}

A priori, $H^{1}(X, \mathcal{A})_{/ \divis}$ might contain
a subgroup isomorphic to $\bigoplus_{l} \Z / l \Z$ for example.
In this case, the term $H^{1}(X, \mathcal{A})_{/ \divis}^{\wedge} / H^{1}(X, \mathcal{A})_{/ \divis}$ above
would contain $(\prod \Z / l \Z) / (\bigoplus \Z / l \Z)$.

\begin{proof}
	We have
		\[
					R \Gamma(X, \mathcal{A})
				\ctensor
					\Adele^{\infty}
			=
				R \Gamma \bigl(
					k_{\pro\et},
						R \alg{\Gamma}(X, \mathcal{A})
					\ctensor
						\Adele^{\infty}
				\bigr).
		\]
	Denote these isomorphic objects by $C$.
	We first calculate $H^{n} C$ through the left-hand side.
	Recall that $\Gamma(X, \mathcal{A})$ is finitely generated,
	$H^{1}(X, \mathcal{A})$ is torsion of cofinite type,
	and $H^{2}(X, \mathcal{A})$ is torsion, divisibly ML and of cofinite type
	(Prop.\ \ref{prop: duality for H two in finite base field case}).
	Hence Prop.\ \ref{prop: cohomologies of derived completion} gives exact sequences
		\begin{equation} \label{eq: absolute cohom of Neron tensored with adeles}
			\begin{split}
						0
					\to
						\Gamma(X, \mathcal{A}) \tensor \Adele^{\infty}
					\to
						H^{0} C
					\to
						V H^{1}(X, \mathcal{A})_{\divis}
					\to
						0,
				\\
						0
					\to
						H^{1}(X, \mathcal{A})^{\wedge} \tensor \Q
					\to
						H^{1} C
					\to
						V H^{2}(X, \mathcal{A})_{\divis}
					\to
						0
			\end{split}
		\end{equation}
	and $H^{n} C = 0$ for $n \ne 0, 1$.
	We have
		\[
				H^{1}(X, \mathcal{A})^{\wedge} \tensor \Q
			=
				\frac{
					H^{1}(X, \mathcal{A})_{/ \divis}^{\wedge}
				}{
					H^{1}(X, \mathcal{A})_{/ \divis}
				}.
		\]
	
	Next we calculate $H^{n} C$ through the right-hand side.
	By Prop.\ \ref{prop: each part of adelic duality}
	and Thm.\ \ref{thm: duality for Neron} \eqref{item: main thm: structure of each part},
	\eqref{item: duality between TGamma of A dual and H2 of A},
	the $H^{-1}$ (resp.\ $H^{1}$) of $R \alg{\Gamma}(X, \mathcal{A}) \ctensor \Adele^{\infty}$
	is the rational Tate module (resp.\ the Pontryagin dual of the rational Tate module)
	of an abelian variety over $k$.
	For any abelian variety $B$ over $k$, we have
	$R \Gamma(k_{\pro\et}, B) = B(k)$ by Lang's theorem,
	and $B(k)$ is finite.
	Hence $R \Gamma(k_{\pro\et}, V B) = R \Gamma(k_{\pro\et}, (V B)^{\PDual}) = 0$.
	Therefore
		\[
				C
			=
				R \Gamma \Bigl(
					k_{\pro\et},
					H^{0} \bigl(
							R \alg{\Gamma}(X, \mathcal{A})
						\ctensor
							\Adele^{\infty}
					\bigr)
				\Bigl).
		\]
	By Prop.\ \ref{prop: each part of adelic duality} and
	\ref{prop: adelic duality for Neron}, we have a split exact sequence
		\[
				0
			\to
				\pi_{0}(\alg{\Gamma}(X, \mathcal{A})) \tensor \Adele^{\infty}
			\to
				H^{0} \bigl(
						R \alg{\Gamma}(X, \mathcal{A})
					\ctensor
						\Adele^{\infty}
				\bigr)
			\to
				V \alg{H}^{1}(X, \mathcal{A})_{\divis}
			\to
				0
		\]
	in $\Ab(k^{\ind\rat}_{\pro\et})$.
	Since $\alg{\Gamma}(X, \mathcal{A})_{0}$ is the perfection of an abelian variety over $k$,
	the group $H^{n}(k_{\pro\et}, \alg{\Gamma}(X, \mathcal{A})_{0})$ is zero for $n \ge 1$
	by Lang's theorem and finite for $n = 0$.
	Therefore $H^{n}(k_{\pro\et}, \alg{\Gamma}(X, \mathcal{A}))$ is $\Gamma(X, \mathcal{A})$ if $n = 0$ and
	$H^{n}(k_{\pro\et}, \pi_{0}(\alg{\Gamma}(X, \mathcal{A})))$ if $n \ge 1$,
	which is zero if $n \ge 2$.
	Since $\pi_{0}(\alg{\Gamma}(X, \mathcal{A})) \in \FGEt / k$,
	the Frobenius action on $\pi_{0}(\alg{\Gamma}(X, \mathcal{A})) \tensor \Q$
	and hence on $\pi_{0}(\alg{\Gamma}(X, \mathcal{A})) \tensor \Adele^{\infty}$
	are semisimple.
	Therefore its invariant part and coinvariant part agree.
	Hence
		\[
				H^{n} \bigl(
					k_{\pro\et},
					\pi_{0}(\alg{\Gamma}(X, \mathcal{A})) \tensor \Adele^{\infty}
				\bigr)
			=
				\begin{cases}
						\Gamma(X, \mathcal{A}) \tensor \Adele^{\infty}
					&
						\text{if } n = 0, 1,
					\\
						0
					&
						\text{otherwise}.
				\end{cases}
		\]
	Therefore we have split exact sequences
		\begin{equation} \label{eq: cohom of relative cohom of Neron tensored with adeles}
			\begin{split}
						0
					\to
						\Gamma(X, \mathcal{A}) \tensor \Adele^{\infty}
					\to
						H^{0} C
					\to
						\Gamma \bigl(
							k_{\pro\et},
							V \alg{H}^{1}(X, \mathcal{A})_{\divis}
						\bigr)
					\to
						0,
				\\
						0
					\to
						\Gamma(X, \mathcal{A}) \tensor \Adele^{\infty}
					\to
						H^{1} C
					\to
						H^{1} \bigl(
							k_{\pro\et},
							V \alg{H}^{1}(X, \mathcal{A})_{\divis}
						\bigr)
					\to
						0,
			\end{split}
		\end{equation}
	and $H^{n} C = 0$ for $n \ne 0, 1$.
	
	Comparing the first exact sequences in
	\eqref{eq: absolute cohom of Neron tensored with adeles}
	and \eqref{eq: cohom of relative cohom of Neron tensored with adeles},
	we know that \eqref{eq: Frob invariant of V H one div} is an isomorphism.
	By Prop.\ \ref{prop: duality for H two in finite base field case},
	we have an exact sequence
		\[
				0
			\to
				(V H^{1}(X, \mathcal{A}_{0}^{\vee})_{\divis})^{\PDual}
			\to
				V H^{2}(X, \mathcal{A})_{\divis}
			\to
				(\Gamma(X, \mathcal{A}_{0}^{\vee}) \tensor \Adele^{\infty})^{\PDual}
			\to
				0.
		\]
	With this and the second exact sequence in
	\eqref{eq: absolute cohom of Neron tensored with adeles},
	we obtain a surjection from $H^{1} C$ onto
	$(\Gamma(X, \mathcal{A}_{0}^{\vee}) \tensor \Adele^{\infty})^{\PDual}$.
	Let $D$ be its kernel.
	Consider the diagram with exact rows
		\[
			\begin{CD}
					0
				@>>>
					\Gamma(X, \mathcal{A}) \tensor \Adele^{\infty}
				@>>>
					H^{1} C
				@>>>
					H^{1} \bigl(
						k_{\pro\et},
						V \alg{H}^{1}(X, \mathcal{A})_{\divis}
					\bigr)
				@>>>
					0
				\\
				@.
				@.
				@|
				@.
				@.
				\\
					0
				@>>>
					D
				@>>>
					H^{1} C
				@>>>
					(\Gamma(X, \mathcal{A}_{0}^{\vee}) \tensor \Adele^{\infty})^{\PDual}
				@>>>
					0.
			\end{CD}
		\]
	The composite map from $\Gamma(X, \mathcal{A}) \tensor \Adele^{\infty}$
	to $(\Gamma(X, \mathcal{A}_{0}^{\vee}) \tensor \Adele^{\infty})^{\PDual}$
	is given by the height pairing tensored with $\Adele^{\infty}$, which is an isomorphism.
	Therefore the composite map from $D$ to
		$
			H^{1} \bigl(
				k_{\pro\et},
				V \alg{H}^{1}(X, \mathcal{A})_{\divis}
			\bigr)
		$
	is an isomorphism.
	This gives \eqref{eq: Frob coinvariant of V H one div}
	since $D$ is an extension of 
		$
			\bigl(
				V H^{1}(X, \mathcal{A}_{0}^{\vee})_{\divis}
			\bigr)^{\PDual}
		$
	by $H^{1}(X, \mathcal{A})_{/ \divis}^{\wedge} / H^{1}(X, \mathcal{A})_{/ \divis}$.
\end{proof}

\begin{Prop} \label{prop: finiteness of Sha}
	The following are equivalent:
		\begin{itemize}
			\item
				$\Sha(A / K)$ is finite.
			\item
				$R \Gamma(k_{\pro\et}, V \alg{H}^{1}(X, \mathcal{A})_{\divis})$ is zero.
		\end{itemize}
\end{Prop}

\begin{proof}
	The finiteness of $\Sha(A / K)$ is equivalent to the finiteness of $H^{1}(X, \mathcal{A})$,
	which itself is equivalent to
		\[
				V H^{1}(X, \mathcal{A})_{\divis}
			=
				\frac{
					H^{1}(X, \mathcal{A})_{/ \divis}^{\wedge}
				}{
					H^{1}(X, \mathcal{A})_{/ \divis}
				}
			=
				0.
		\]
	It is also equivalent to the finiteness of $\Sha(A^{\vee} / K)$
	by \cite[I, Rmk.\ 6.14 (c)]{Mil06}.
	Therefore the stated equivalence follows from the previous proposition.
\end{proof}

Note that if $\Sha(A / K)$ is finite,
then the full BSD conjecture for $A$ is true
by the result of Kato-Trihan \cite{KT03}.

Recall from \cite[\S 2.2]{KT03} that a complex, called the arithmetic cohomology
and denoted by $R \Gamma_{\mathit{ar}}$, is defined to be the mapping fiber of the morphism
	\[
				R \Gamma(U, \mathcal{A}_{\tor})
			\oplus
				\bigoplus_{x \not\in U}
					\bigl(
						\Gamma(\Hat{K}_{x}, A) \tensor^{L} \Q / \Z
					\bigr)[-1]
		\to
			\bigoplus_{x \not\in U}
				R \Gamma(\Hat{K}_{x}, A_{\tor}),
	\]
where $U$ is an open subscheme of $X$ over which $A$ has good reduction.
Here we understand the functor $R \Gamma(\Hat{K}_{x}, \var)$
as the composite of $R \alg{\Gamma}(\Hat{K}_{x}, \var)$
and $R \Gamma(k_{\pro\zar}, \var)$,
and $\Gamma = H^{0} R \Gamma$.
To be clear, we denote this complex by $R \Gamma_{\mathit{ar}, A}$
and its cohomology objects by $H^{n}_{\mathit{ar}, A}$.
This mapping fiber is defined on the level of complexes, briefly explained in \cite[\S 2.1]{KT03}.
As they wrote, ``how to use complexes is evident and so we do not explain it''.
We follow the same style in the rest of this subsection.

We similarly have the mapping cone of the morphism
	\[
			R \Gamma(X, \mathcal{A}) \tensor^{L} \Q / \Z
		\to
			\left(
				\prod_{x \in X}
					\tau_{\ge 1}
					R \Gamma(\Hat{\Order}_{x}, \mathcal{A})
			\right) \tensor^{L} \Q / \Z.
	\]
We have
	\[
			\tau_{\ge 1}
			R \Gamma(\Hat{\Order}_{x}, \mathcal{A})
		=
			H^{1}(k_{x}, \pi_{0}(\mathcal{A}_{x}))[-1].
	\]
Hence we have the mapping cone
	\[
			\left[
					R \Gamma(X, \mathcal{A}) \tensor^{L} \Q / \Z
				\to
					\bigoplus_{x}
						H^{1}(k_{x}, \pi_{0}(\mathcal{A}_{x}))
			\right].
	\]
Also we have the mapping cone of the morphism
	\[
			\left(
				\bigoplus_{x \in X}
					\tau_{\le 1}
					R \Gamma_{x}(\Hat{\Order}_{x}, \mathcal{A}_{0})
			\right) \tensor^{L} \Q / \Z
		\to
			R \Gamma(X, \mathcal{A}_{0}) \tensor^{L} \Q / \Z.
	\]
We have
	\[
			\tau_{\le 1}
			R \Gamma_{x}(\Hat{\Order}_{x}, \mathcal{A}_{0})
		=
			\Gamma(k_{x}, \pi_{0}(\mathcal{A}_{x}))[-1].
	\]
Hence we have the mapping cone
	\[
		\left[
				\bigoplus_{x}
					\Gamma(k_{x}, \pi_{0}(\mathcal{A}_{x}))
			\to
				R \Gamma(X, \mathcal{A}_{0}) \tensor^{L} \Q / \Z
		\right].
	\]

\begin{Prop} \label{prop: arithmetic cohomology using Neron models}
	There exist canonical isomorphisms
		\begin{align*}
					R \Gamma_{\mathit{ar}, A}
			&	\cong
					\left[
							R \Gamma(X, \mathcal{A}) \tensor^{L} \Q / \Z
						\to
							\bigoplus_{x}
								H^{1}(k_{x}, \pi_{0}(\mathcal{A}_{x}))
					\right][-2]
			\\
			&	\cong
					\left[
							\bigoplus_{x}
								\Gamma(k_{x}, \pi_{0}(\mathcal{A}_{x}))
						\to
							R \Gamma(X, \mathcal{A}_{0}) \tensor^{L} \Q / \Z
					\right][-1].
		\end{align*}
\end{Prop}

\begin{proof}
	We first show the first isomorphism.
	The mapping cone of
		\[
				\Gamma(\Hat{K}_{x}, A) \tensor^{L} \Q / \Z[-1]
			\to
				R \Gamma(\Hat{K}_{x}, A_{\tor})
		\]
	is $H^{1}(K_{x}, A) \tensor^{L} \Q / \Z[-2]$.
	Hence
		\[
				R \Gamma_{\mathit{ar}, A}
			=
				\left[
						R \Gamma(U, \mathcal{A}) \tensor^{L} \Q / \Z
					\to
						\bigoplus_{x \not\in U}
						H^{1}(K_{x}, A) \tensor^{L} \Q / \Z[-1]
				\right][-2]
		\]
	The morphism
		\[
				R \Gamma(X, \mathcal{A})
			\to
				R \Gamma(U, \mathcal{A})
		\]
	and the morphism
		\[
				\bigoplus_{x \not\in U}
					H^{1}(\Hat{\Order}_{x}, \mathcal{A})[-1]
			\to
				\bigoplus_{x \not\in U}
					H^{1}(\Hat{K}_{x}, A)[-1]
		\]
	have the same mapping cone
	$\bigoplus_{x \not\in U} R \Gamma_{x}(\Hat{\Order}_{x}, \mathcal{A})[1]$.
	This gives the desired result.
	
	For the second isomorphism, use the distinguished triangle
		\[
				R \Gamma(X, \mathcal{A}_{0})
			\to
				R \Gamma(X, \mathcal{A})
			\to
				\bigoplus_{x}
					R \Gamma(k_{x}, \pi_{0}(\mathcal{A}_{x})).
		\]
\end{proof}

\begin{Prop} \label{prop: duality for arithmetic cohom}
	Consider the diagram
		\[
			\begin{CD}
					\bigoplus_{x}
						\Gamma(k_{x}, \pi_{0}(\mathcal{A}_{x}^{\vee}))[-1]
				@>>>
					R \invlim_{n}(
						R \Gamma(X, \mathcal{A}_{0}^{\vee}) \tensor^{L} \Z / n \Z
					)
				\\
				@VV \wr V
				@V \wr VV
				\\
					\bigoplus_{x}
						H^{1}(k_{x}, \pi_{0}(\mathcal{A}_{x})^{\PDual}[-1]
				@>>>
					(R \Gamma(X, \mathcal{A}) \tensor^{L} \Q / \Z)^{\PDual}[-1],
			\end{CD}
		\]
	where the horizontal morphisms come from the morphisms in the previous proposition
	(with $R \invlim_{n}(\var \tensor^{L} \Z / n \Z)$ applied),
	the left vertical one from Grothendieck's pairing
	and the right vertical one from
	Prop.\ \ref{prop: Q mod Z duality for Neron over finite base}.
	This diagram is commutative and canonically induces an isomorphism
		\[
				R \invlim_{n} \bigl(
					R \Gamma_{\mathit{ar}, A^{\vee}} \tensor^{L} \Z / n \Z
				\bigr)
			\isomto
				(R \Gamma_{\mathit{ar}, A})^{\PDual}[-2]
		\]
	on the mapping cones.
\end{Prop}

\begin{proof}
	Consider the morphism
		\[
				\bigoplus_{x}
					\tau_{\le 1}
					R \Gamma_{x}(\Hat{\Order}_{x}, \mathcal{A}_{0}^{\vee})
			\to
				R \Gamma(X, \mathcal{A}_{0}^{\vee}).
		\]
	We have a natural morphism from this morphism to the morphism
		\[
				\bigoplus_{x}
					\tau_{\le 1}
					R \Gamma_{x} \bigl(
						\Hat{\Order}_{x},
						R \sheafhom_{\Hat{\Order}_{x}}(\mathcal{A}^{\vee}, \Gm)
					\bigr)[1]
			\to
				R \Gamma \bigl(
					X, R \sheafhom_{\Hat{\Order}_{x}}(\mathcal{A}^{\vee}, \Gm)
				\bigr)[1]
		\]
	(i.e.\ have a commutative square whose upper and lower sides are these morphisms).
	Using the functoriality/cup product morphisms similar to
	\eqref{eq: derived functoriality morphism over integers, variant}
	and \eqref{eq: global cup product}
	and the morphism $R \Gamma_{x}(\Hat{\Order}_{x}, \Gm) \to R \Gamma(X, \Gm)$
	we have a morphism from this morphism to the morphism
		\begin{align*}
			&
					\bigoplus_{x}
						R \sheafhom_{k_{\pro\zar}} \bigl(
							R \Gamma(\Hat{\Order}_{x}, \mathcal{A}),
							\tau_{\ge 3}
							R \Gamma(X, \Gm)
						\bigr)[1]
			\\
			&	\to
					R \sheafhom_{k_{\pro\zar}} \bigl(
						R \Gamma(X, \mathcal{A}),
						\tau_{\ge 3}
						R \Gamma(X, \Gm)
					\bigr)[1].
		\end{align*}
	We have $\tau_{\ge 3} R \Gamma(X, \Gm) = \Q / \Z[-3]$.
	Hence the mapping cone of this morphism can be written as
		\[
			R \sheafhom_{k_{\pro\zar}} \Bigl(
				\bigl[
						R \Gamma(X, \mathcal{A})
					\to
						\prod_{x}
							\tau_{\ge 1}
							R \Gamma(\Hat{\Order}_{x}, \mathcal{A})
				\bigr],
				\Q / \Z
			\Bigr)[-1].
		\]
	These are functorial on the level of complexes.
	Hence we have a morphism
		\begin{align*}
			&
					\Bigl[
							\bigoplus_{x}
								\tau_{\le 1}
								R \Gamma_{x}(\Hat{\Order}_{x}, \mathcal{A}_{0}^{\vee})
						\to
							R \Gamma(X, \mathcal{A}_{0}^{\vee})
					\Bigr]
			\\
			&	\to
					R \sheafhom_{k_{\pro\zar}} \Bigl(
						\bigl[
								R \Gamma(X, \mathcal{A})
							\to
								\prod_{x}
									\tau_{\ge 1}
									R \Gamma(\Hat{\Order}_{x}, \mathcal{A})
						\bigr],
						\Q / \Z
					\Bigr)[-1].
		\end{align*}
	Applying $R \invlim_{n}(\var \tensor^{L} \Z / n \Z)$,
	we get the result.
\end{proof}

This form of duality on $R \Gamma_{\mathit{ar}, A}$ is also given in \cite[Cor.\ 4.11, Rmk.\ 4.13]{TV17}.

For a torsion group $C$ of cofinite type,
we call $(C_{/ \divis})^{\PDual}$ the essentially torsion part of the profinite group $C^{\PDual}$,
which is the product over all primes $l$ of the torsion part of the $l$-adic component of $C^{\PDual}$.
For example, $\prod_{l} \Z / l \Z$ is the essentially torsion part of itself.
By \cite[Prop.\ 2.4]{KT03}, we know that
$H^{1}_{\mathit{ar}, A}$ is isomorphic to the Selmer group of $A$.
Hence the isomorphism in the above proposition in degree $1$ on the essentially torsion parts gives a perfect pairing
	\[
			\Sha(A^{\vee} / K)_{/ \divis}^{\wedge} \times \Sha(A / K)_{/ \divis}
		\to
			\Q / \Z
	\]
between a profinite group and a torsion group.
For any prime $l$, this pairing on $l$-primary parts is a pairing between finite $l$-groups.

\begin{Prop} \label{prop: agrees with the Cassels Tate}
	This pairing agrees with the Cassels-Tate pairing.
\end{Prop}

\begin{proof}
	The isomorphism in the proposition gives a Pontryagin duality between
	the exact sequences
		\[
				\bigoplus_{x}
					\Gamma(k_{x}, \pi_{0}(\mathcal{A}_{x}^{\vee}))
			\to
				H^{1}(X, \mathcal{A}_{0}^{\vee})_{/ \divis}
			\to
				(H^{1}_{\mathit{ar}, A^{\vee}})_{/ \divis}
			\to
				0
		\]
	and
		\[
				0
			\to
				(H^{1}_{\mathit{ar}, A})_{/ \divis}^{\wedge}
			\to
				H^{1}(X, \mathcal{A})_{/ \divis}^{\wedge}
			\to
				\bigoplus_{x}
					H^{1}(k_{x}, \pi_{0}(\mathcal{A}_{x})).
		\]
	Hence the proof of \cite[III, Cor.\ 9.5]{Mil06} shows that
	the paring induced on the third term of the first exact sequence
	and the first term of the second exact sequence is
	the Cassels-Tate pairing.
\end{proof}

We briefly recall the triangulated functor
	\[
			R \Gamma(k_{W}, \var)
		\colon
			D(k^{\ind\rat}_{\pro\et})
		\to
			D(k_{\pro\zar})
	\]
defined in \cite[\S 10]{Suz14}.
Let $F$ be the $q$-th power Frobenius morphism for any $k$-algebra,
which induces an action on any object of $\Ab(k^{\ind\rat}_{\pro\et})$.
The functor $k' \mapsto k' \tensor_{k} \closure{k}$
defines a premorphism of sites
$f \colon \Spec k^{\ind\rat}_{\pro\et} \to \Spec k_{\pro\zar}$.
Then
	\[
			R \Gamma(k_{W}, C)
		=
			f_{\ast} [C \overset{F - 1}{\to} C][-1]
	\]
for $C \in D(k^{\ind\rat}_{\pro\et})$.

We define
	\[
			R \Gamma(X_{W}, \var)
		=
			R \Gamma(k_{W}, R \alg{\Gamma}(X, \var))
		\colon
			D(X_{\fppf})
		\to
			D(k_{\pro\zar}).
	\]
We denote $(\var)^{\LDual} = R \sheafhom_{k_{\pro\zar}}(\var, \Z)$.

\begin{Prop}
	We have
		\begin{gather*}
					R \Gamma \bigl(
						k_{W},
						R \alg{\Gamma}(X, \mathcal{A}_{0}^{\vee})^{\SDual \SDual}
					\bigr)
				=
					R \Gamma(X_{W}, \mathcal{A}_{0}^{\vee}),
			\\
					R \Gamma \bigl(
						k_{W},
						R \alg{\Gamma}(X, \mathcal{A})^{\SDual}
					\bigr)
				=
					R \Gamma(X_{W}, \mathcal{A})^{\LDual}[-1],
			\\
					R \Gamma(k_{W}, V \alg{H}^{1}(X, \mathcal{A}))
				=
					R \Gamma(k_{\pro\et}, V \alg{H}^{1}(X, \mathcal{A})).
		\end{gather*}
\end{Prop}

\begin{proof}
	For the first isomorphism,
	note that the mapping fiber of
		\[
				R \alg{\Gamma}(X, \mathcal{A}_{0}^{\vee})
			\to
				R \alg{\Gamma}(X, \mathcal{A}_{0}^{\vee})^{\SDual \SDual}
		\]
	is concentrated in degree zero
	whose cohomology is $\invlim_{n} \alg{\Gamma}(X, \mathcal{A}_{0}^{\vee})_{0}$.
	Since the endomorphism $F - 1$ on the abelian variety
	$\alg{\Gamma}(X, \mathcal{A}_{0}^{\vee})_{0}$
	is surjective with finite kernel by Lang's theorem,
	we know that $\invlim_{n} \alg{\Gamma}(X, \mathcal{A}_{0}^{\vee})_{0}$
	is killed by applying $R \Gamma(k_{W}, \var)$.
	This gives the desired result.
	
	Th second isomorphism follows from
		\[
				R \Gamma \bigl(
					k_{W},
					R \alg{\Gamma}(X, \mathcal{A})^{\SDual}
				\bigr)
			=
				R \Gamma \bigl(
					k_{W},
					R \alg{\Gamma}(X, \mathcal{A})
				\bigr)^{\LDual}[-1],
		\]
	which is \cite[Prop.\ (10.4)]{Suz14}.
	For the third, we already saw in the proof of
	Prop.\ \ref{prop: finiteness of Sha} that
	$R \Gamma(k_{\pro\et}, V \alg{H}^{1}(X, \mathcal{A}))$
	is the mapping fiber of $F - 1$.
\end{proof}

Here is a Weil-\'etale analogue of
Prop.\ \ref{prop: Q mod Z duality for Neron over finite base}.

\begin{Prop}
	Applying $R \Gamma(k_{W}, \var)$ to the distinguished triangle in
	Theorem \ref{thm: duality for Neron},
	we have a canonical distinguished triangle
		\[
				R \Gamma(X_{W}, \mathcal{A}_{0}^{\vee})
			\to
				R \Gamma(X_{W}, \mathcal{A})^{\LDual}[-1]
			\to
				R \Gamma(k_{W}, V \alg{H}^{1}(X, \mathcal{A}_{0}^{\vee})).
		\]
	The third term is zero if and only if $\Sha(A / K)$ is finite.
\end{Prop}

\begin{proof}
	This follows from the previous proposition
	and Prop.\ \ref{prop: finiteness of Sha}.
\end{proof}

If $\Sha(A / K)$ is finite,
then the cohomology objects of
$R \Gamma(X_{W}, \mathcal{A})$ and $R \Gamma(X_{W}, \mathcal{A}_{0}^{\vee})$
are finitely generated abelian groups,
and the above proposition gives a duality between these objects
via $R \Hom(\var, \Z)$.

We can also define a Weil-\'etale version $R \Gamma_{\mathit{ar}, A, W}$ of $R \Gamma_{\mathit{ar}, A}$ by
		\begin{align*}
					R \Gamma_{\mathit{ar}, A, W}
			&	=
					\left[
							R \Gamma(X_{W}, \mathcal{A})
						\to
							\bigoplus_{x}
								H^{1}(k_{x}, \pi_{0}(\mathcal{A}_{x}))[-1]
					\right][-1]
			\\
			&	\cong
					\left[
							\bigoplus_{x}
								\Gamma(k_{x}, \pi_{0}(\mathcal{A}_{x}))[-1]
						\to
							R \Gamma(X_{W}, \mathcal{A}_{0})
					\right]
		\end{align*}
and prove the existence of a canonical distinguished triangle
		\[
				R \Gamma_{\mathit{ar}, A^{\vee}, W}
			\to
				R \Gamma_{\mathit{ar}, A, W}^{\LDual}[-1]
			\to
				R \Gamma(k_{W}, V \alg{H}^{1}(X, \mathcal{A}_{0}^{\vee})).
		\]

\begin{Rmk} \label{rmk: duality for finite flat over open curve over finite field}
	We can also apply $R \Gamma(k_{\pro\et}, \var)$ to
	Thm.\ \ref{thm: duality for finite flat over open curve}
	to recover the duality $R \Gamma(U, N) \leftrightarrow R \Gamma_{c}(U, N)$
	for a finite flat group scheme $N$ over an open curve $U$ over $k = \F_{q}$
	stated in \cite[III, Thm.\ 8.2]{Mil06},
	including the topological group structures on the relevant cohomology groups.
	Define $R \Gamma_{c}(U, N) \in D(k_{\pro\zar})$
	by $R \Gamma(k_{\pro\et}, R \alg{\Gamma}_{c}(U, N)) = R \Gamma(k_{W}, R \alg{\Gamma}_{c}(U, N))$.
	Since $R \alg{\Gamma}_{c}(U, N) \in D^{b}(\Pro \Alg_{\uc} / k)$ by
	Thm.\ \ref{thm: duality for finite flat over open curve},
	we know that $R \Gamma_{c}(U, N) \in D^{b}(\Pro \Fin)$
	by \cite[Prop.\ (10.3) (b)]{Suz14},
	where $\Fin$ is the category of finite abelian groups.
	Therefore each cohomology object $H_{c}^{n}(U, N) := H^{n} R \Gamma_{c}(U, N)$
	is a profinite group.
	Now applying $R \Gamma(k_{\pro\et}, \var)$ to
	Thm.\ \ref{thm: duality for finite flat over open curve}
	and using Prop.\ \ref{prop: cohom of Serre dual is Pont dual of cohom},
	we obtain an isomorphism
		\[
				R \Gamma(U, N)
			\isomto
				R \Gamma_{c}(U, N)^{\PDual}[-3]
		\]
	in $D(k_{\pro\zar})$.
	Hence we have a Pontryagin duality
	$H^{n}(U, N) \leftrightarrow H_{c}^{3 - n}(U, N)$ for any $n$
	between the torsion group and the profinite group,
	which recovers \cite[III, Thm.\ 8.2]{Mil06}.
	
	A similar remark applies to Prop.\ \ref{prop: duality for Neron over open curve}
	and \cite[III, Thm.\ 9.4]{Mil06}.
\end{Rmk}


\appendix

\section{Comparison of local duality with and without relative sites}
\label{sec: Comparison of local duality with and without relative sites}

We continue the notation of \S \ref{sec: Local duality without relative sites}.
We recall the relative fppf site $\Spec \Hat{\Order}_{x, \fppf} / k^{\ind\rat}_{x, \et}$
of $\Hat{\Order}_{x}$ over $k_{x}$ from
\cite[Def.\ 2.3.2]{Suz13} and \cite[\S 3.3]{Suz14}.
(Here, again, we abbreviate $(k_{x})^{\ind\rat}$ by $k^{\ind\rat}_{x}$
and $(\Spec k_{x})^{\ind\rat}_{\et}$ by $\Spec k^{\ind\rat}_{x, \et}$.)
The category $\Hat{\Order}_{x} / k^{\ind\rat}_{x}$ is the category of pairs $(S, k_{S})$,
where $k_{S} \in k^{\ind\rat}_{x}$ and $S$ a $\Hat{\alg{O}}_{x}(k_{S})$-algebra.
(In \cite[\S 2.3]{Suz13} and \cite[\S 3.3]{Suz14},
$S$ is required to be finitely presented over $\Hat{\alg{O}}_{x}(k')$.
The generalization to arbitrary algebras here does not change cohomology theory and Ext groups/sheaves
as long as we treat sheaves representable by schemes locally of finite type over $\Hat{\Order}_{x}$.)
A morphism $(S, k_{S}) \to (S', S_{S'})$ consists of a $k_{x}$-algebra homomorphism $k_{S} \to k_{S'}$
and an $\Hat{\alg{O}}_{x}(k_{S})$-algebra homomorphism $S \to S'$,
with composition defined in the obvious way.
A finite family of morphisms $\{(S, k_{S}) \to (S_{i}, k_{S_{i}})\}$ is called an fppf/\'etale covering
if each $k_{S} \to k_{S_{i}}$ is \'etale
and $S \to \prod_{i} S_{i}$ faithfully flat.
This defines a site, which we denote by
$\Spec \Hat{\Order}_{x, \fppf} / k^{\ind\rat}_{x, \et}$.
Its category of sheaves is denoted by $\Ab(\Hat{\Order}_{x, \fppf} / k^{\ind\rat}_{x, \et})$.
The categories $\Ch(\Hat{\Order}_{x, \fppf} / k^{\ind\rat}_{x, \et})$,
$D(\Hat{\Order}_{x, \fppf} / k^{\ind\rat}_{x, \et})$
and the functors $\sheafhom_{\Hat{\Order}_{x, \fppf} / k^{\ind\rat}_{x, \et}}$,
$R \sheafhom_{\Hat{\Order}_{x, \fppf} / k^{\ind\rat}_{x, \et}}$
are defined in a similar way as \S \ref{sec: Local duality without relative sites}.

The functors
	\begin{gather*}
				k^{\ind\rat}_{x}
			\to
				\Hat{\Order}_{x} / k^{\ind\rat}_{x},
			\quad
				k'
			\mapsto
				(\Hat{\alg{O}}_{x}(k'), k'),
		\\
				k^{\ind\rat}_{x}
			\to
				\Hat{K}_{x} / k^{\ind\rat}_{x},
			\quad
				k'
			\mapsto
				(\Hat{\alg{K}}_{x}(k'), k'),
		\\
				\Hat{\Order}_{x} / k^{\ind\rat}_{x}
			\to
				\Hat{K}_{x} / k^{\ind\rat}_{x},
			\quad
				(S, k_{S})
			\mapsto
				(S \tensor_{\Hat{\Order}_{x}} \Hat{K}_{x}, k_{S})
	\end{gather*}
define morphisms of sites
	\begin{gather*}
				\pi_{\Hat{\Order}_{x}}'
			\colon
				\Spec \Hat{\Order}_{x, \fppf} / k^{\ind\rat}_{x, \et}
			\to
				\Spec k^{\ind\rat}_{x, \et},
		\\
				\pi_{\Hat{K}_{x}}'
			\colon
				\Spec \Hat{K}_{x, \fppf} / k^{\ind\rat}_{x, \et}
			\to
				\Spec k^{\ind\rat}_{x, \et},
		\\
				j'
			\colon
				\Spec \Hat{K}_{x, \fppf} / k^{\ind\rat}_{x, \et}
			\to
				\Spec \Hat{\Order}_{x, \fppf} / k^{\ind\rat}_{x, \et}
	\end{gather*}
such that $\pi_{\Hat{K}_{x}}' = \pi_{\Hat{\Order}_{x}}' \compose j'$.
We denote the composite of the derived pushforward
	\[
			R (\pi_{\Hat{\Order}_{x}}')_{\ast}
		\colon
			D(\Hat{\Order}_{x, \fppf} / k^{\ind\rat}_{\et})
		\to
			D(k^{\ind\rat}_{x, \et})
	\]
and the pro-\'etale sheafification
	$
			D(k^{\ind\rat}_{x, \et})
		\to
			D(k^{\ind\rat}_{x, \pro\et})
	$
by
	\[
			R \alg{\Gamma}'(\Hat{\Order}_{x}, \var)
		\colon
			D(\Hat{\Order}_{x, \fppf} / k^{\ind\rat}_{\et})
		\to
			D(k^{\ind\rat}_{x, \pro\et}).
	\]
Similarly, we denote the composite of the derived pushforward
	\[
			R (\pi_{\Hat{K}_{x}}')_{\ast}
		\colon
			D(\Hat{K}_{x, \fppf} / k^{\ind\rat}_{\et})
		\to
			D(k^{\ind\rat}_{x, \et})
	\]
and the pro-\'etale sheafification
	$
			D(k^{\ind\rat}_{x, \et})
		\to
			D(k^{\ind\rat}_{x, \pro\et})
	$
by
	\[
			R \alg{\Gamma}'(\Hat{K}_{x}, \var)
		\colon
			D(\Hat{K}_{x, \fppf} / k^{\ind\rat}_{\et})
		\to
			D(k^{\ind\rat}_{x, \pro\et}).
	\]
In \cite[\S 3.3, paragraph before Prop.\ (3.3.8)]{Suz14},
these functors were denoted by $R \Tilde{\alg{\Gamma}}(\Order_{x}, \var)$
and $R \Tilde{\alg{\Gamma}}(K_{x}, \var)$, respectively.

We compare the above morphisms of sites and the premorphisms of sites in
\S \ref{sec: Local duality without relative sites}.
The functor sending $(S, k_{S}) \in \Hat{\Order}_{x} / k^{\ind\rat}$ to $S$ defines a premorphism of sites
	\[
			\Bar{f}
		\colon
			\Spec \Hat{\Order}_{x, \fppf}
		\to
			\Spec \Hat{\Order}_{x, \fppf} / k^{\ind\rat}_{x, \et}.
	\]
Similarly, the functor sending $(S, k_{S}) \in \Hat{K}_{x} / k^{\ind\rat}$ to $S$
defines a premorphism of sites
	\[
			f
		\colon
			\Spec \Hat{K}_{x, \fppf}
		\to
			\Spec \Hat{K}_{x, \fppf} / k^{\ind\rat}_{x, \et}.
	\]
The pushforward functors $\Bar{f}_{\ast}$ and $f_{\ast}$ are exact.
They send the sheaf representable by a $\Hat{\Order}_{x}$-algebra (resp.\ $\Hat{K}_{x}$-algebra) $S$
to the sheaf representable by $(S, k)$.
Hence a scheme over $\Hat{\Order}_{x}$ (resp.\ $\Hat{K}_{x}$)
regarded as a sheaf on $\Spec \Hat{\Order}_{x, \fppf} / k^{\ind\rat}_{x, \et}$
(resp.\ $\Spec \Hat{K}_{x, \fppf} / k^{\ind\rat}_{x, \et}$)
mentioned in \cite[\S 2.3]{Suz13}
is nothing but its image by $\Bar{f}_{\ast}$ (resp.\ $f_{\ast}$).

\begin{PropApp} \label{prop: comparison with relative site formalism}
	We have $\pi_{\Hat{\Order}_{x}} = \pi_{\Hat{\Order}_{x}}' \compose \Bar{f}$
	and $\pi_{\Hat{K}_{x}} = \pi_{\Hat{K}_{x}}' \compose f$,
	and
		\[
					R \alg{\Gamma}(\Hat{\Order}_{x}, \var)
				=
					R \alg{\Gamma}' \bigl(
						\Hat{\Order}_{x},
						\Bar{f}_{\ast} \var
					\bigr),
			\quad
					R \alg{\Gamma}(\Hat{K}_{x}, \var)
				=
					R \alg{\Gamma}' \bigl(
						\Hat{\Order}_{x},
						f_{\ast} \var
					\bigr).
		\]
	In particular, for a group scheme $A$ over $\Hat{\Order}_{x}$ or $\Hat{K}_{x}$, we have
		\[
					R \alg{\Gamma}(\Hat{\Order}_{x}, A)
				=
					R \alg{\Gamma}' \bigl(
						\Hat{\Order}_{x}, A
					\bigr),
			\quad
					R \alg{\Gamma}(\Hat{K}_{x}, A)
				=
					R \alg{\Gamma}' \bigl(
						\Hat{\Order}_{x}, A
					\bigr).
		\]
\end{PropApp}

\begin{proof}
	The composite $\pi_{\Hat{\Order}_{x}}' \compose \Bar{f}$ is defined by
	the composite of the functors
	$k'_{x} \mapsto (\Hat{\alg{O}}_{x}(k'), k'_{x}) \mapsto \Hat{\alg{O}}_{x}(k'_{x})$.
	which defines $\pi_{\Hat{\Order}_{x}}$.
	Hence $\pi_{\Hat{\Order}_{x}} = \pi_{\Hat{\Order}_{x}}' \compose \Bar{f}$.
	Similar for $\pi_{\Hat{K}_{x}} = \pi_{\Hat{K}_{x}}' \compose f$.
	The equalities of the derived pushforwards follow from
	Prop.\ \ref{prop: K-limp is push injective}, \ref{prop: push sends K limp to K limp}
	and the theorem on derived functors of composition \cite[Prop.\ 10.3.5 (ii)]{KS06}.
\end{proof}

\begin{PropApp} \label{prop: comparison with relative site support cohmology}
	The mapping fiber functor
		\begin{gather*}
					\alg{\Gamma}_{x}'(\Order_{x}, \var)
				:=
					\bigl[
							\alg{\Gamma}'(\Order_{x}, \var)
						\to
							\alg{\Gamma}' \bigl(
								K_{x},
								j'^{\ast} \var
							\bigr)
					\bigr][-1]
				\colon
			\\
					\Ch(\Hat{\Order}_{x, \fppf} / k^{\ind\rat}_{\et})
				\to
					\Ch(k^{\ind\rat}_{x, \pro\et})
		\end{gather*}
	and the functor
		\begin{gather*}
					\alg{\Gamma}_{x}(\Order_{x}, \var)
				=
					\bigl[
							\alg{\Gamma}(\Order_{x}, \var)
						\to
							\alg{\Gamma} \bigl(
								K_{x},
								j^{\ast} \var
							\bigr)
					\bigr][-1]
				\colon
			\\
					\Ch(\Hat{\Order}_{x, \fppf})
				\to
					\Ch(k^{\ind\rat}_{x, \pro\et})
		\end{gather*}
	are compatible under the functor
		\[
				\Bar{f}_{\ast}
			\colon
				\Ch(\Hat{\Order}_{x, \fppf})
			\to
				\Ch(\Hat{\Order}_{x, \fppf} / k^{\ind\rat}_{x, \et}).
		\]
	The right derived functors
		\begin{gather*}
					R \alg{\Gamma}_{x}'(\Order_{x}, \var)
				=
					\bigl[
							R \alg{\Gamma}'(\Order_{x}, \var)
						\to
							R \alg{\Gamma}' \bigl(
								K_{x},
								j'^{\ast} \var
							\bigr)
					\bigr][-1]
				\colon
			\\
					D(\Hat{\Order}_{x, \fppf} / k^{\ind\rat}_{\et})
				\to
					D(k^{\ind\rat}_{x, \pro\et}).
		\end{gather*}
	and
		\begin{gather*}
					R \alg{\Gamma}_{x}(\Order_{x}, \var)
				=
					\bigl[
							R \alg{\Gamma}(\Order_{x}, \var)
						\to
							R \alg{\Gamma} \bigl(
								K_{x},
								j^{\ast} \var
							\bigr)
					\bigr][-1]
				\colon
			\\
					D(\Hat{\Order}_{x, \fppf})
				\to
					D(k^{\ind\rat}_{x, \pro\et})
		\end{gather*}
	are also compatible under the functor
		\[
				\Bar{f}_{\ast}
			\colon
				D(\Hat{\Order}_{x, \fppf})
			\to
				D(\Hat{\Order}_{x, \fppf} / k^{\ind\rat}_{x, \et}).
		\]
\end{PropApp}

\begin{proof}
	We have a commutative diagram of premorphisms of sites
		\[
			\begin{CD}
					\Spec \Hat{K}_{x, \fppf}
				@>> f >
					\Spec \Hat{K}_{x, \fppf} / k^{\ind\rat}_{\et}
				\\
				@VV j V
				@V j' VV
				\\
					\Spec \Hat{\Order}_{x, \fppf}
				@> \Bar{f} >>
					\Spec \Hat{\Order}_{x, \fppf} / k^{\ind\rat}_{\et}.
			\end{CD}
		\]
	Hence the first statement follows from
	Prop.\ \ref{prop: comparison with relative site formalism}.
	We know that
	$j$ and $j'$ are localization morphisms
	\cite[III, \S 5]{AGV72a}.
	Hence their pullback functors admit exact left adjoints
	by \cite[IV, Prop.\ 11.3.1]{AGV72a}
	and hence send K-injectives to K-injectives.
	With Prop.\ \ref{prop: K-limp is push injective}, we know that the first statement implies the second.
\end{proof}

By \cite[Prop.\ (3.3.8)]{Suz14},
the functor $R \alg{\Gamma}_{x}'(\Order_{x}, \var)$ is isomorphic to
the functor denoted by
$R \Tilde{\alg{\Gamma}}_{x}(\Order_{x}, \var)$
in \cite[paragraph before Prop.\ (3.3.8)]{Suz14}.

We can define the morphism of functoriality of $\Bar{f}_{\ast}$
	\[
			\Bar{f}_{\ast} R \sheafhom_{\Order_{x}}(A, B)
		\to
			R \sheafhom_{\Order_{x, \fppf} / k^{\ind\rat}_{\et}}(
				\Bar{f}_{\ast} A,
				\Bar{f}_{\ast} B
			)
	\]
in $D(\Order_{x, \fppf} / k^{\ind\rat}_{\et})$,
functorial on $A, B \in D(\Order_{x, \fppf})$,
in a way similar to the definition of
\eqref{eq: derived functoriality morphism over integers}.
Hence we have a morphism
	\begin{align*}
				R \alg{\Gamma} \bigl(
					\Order_{x},
					R \sheafhom_{\Order_{x}}(A, B)
				\bigr)
		&	=
				R \alg{\Gamma}' \bigl(
					\Order_{x},
					\Bar{f}_{\ast}
					R \sheafhom_{\Order_{x}}(A, B)
				\bigr)
		\\
		&	\to
				R \alg{\Gamma}' \bigl(
					\Order_{x},
					R \sheafhom_{\Order_{x, \fppf} / k^{\ind\rat}_{\et}}(
					\Bar{f}_{\ast} A,
					\Bar{f}_{\ast} B
					)
				\bigr).
	\end{align*}
There is a similar morphism of functoriality of $f_{\ast}$.

\begin{PropApp} \label{prop: comparison with relative site functoriality morphisms}
	Let $A, B \in D(\Order_{x, \fppf})$.
	Under the above morphism,
	the morphism \eqref{eq: derived functoriality morphism over integers}
		\begin{align*}
			&
					R \alg{\Gamma} \bigl(
						\Hat{\Order}_{x},
						R \sheafhom_{\Hat{\Order}_{x}}(A, B)
					\bigr)
			\\
			&	\to
					R \sheafhom_{k^{\ind\rat}_{x, \pro\et}} \bigl(
						R \alg{\Gamma}_{x}(\Hat{\Order}_{x}, A),
						R \alg{\Gamma}_{x}(\Hat{\Order}_{x}, B)
					\bigr)
		\end{align*}
	and the morphism
		\begin{align*}
			&
					R \alg{\Gamma}' \bigl(
						\Hat{\Order}_{x},
						R \sheafhom_{\Hat{\Order}_{x, \fppf} / k^{\ind\rat}_{\et}}(
							\Bar{f}_{\ast} A,
							\Bar{f}_{\ast} B
						)
					\bigr)
			\\
			&	\to
					R \sheafhom_{k^{\ind\rat}_{x, \pro\et}} \bigl(
						R \alg{\Gamma}_{x}'(\Hat{\Order}_{x}, \Bar{f}_{\ast} A),
						R \alg{\Gamma}_{x}'(\Hat{\Order}_{x}, \Bar{f}_{\ast} B)
					\bigr)
		\end{align*}
	in \cite[Prop.\ (3.3.8)]{Suz14} are compatible
	under the isomorphisms in the previous two propositions.
	Similar compatibilities hold for the morphisms
	\eqref{eq: derived functoriality morphism over integers, variant},
	\eqref{eq: derived functoriality morphism over local field}
	and other morphisms in \cite[Prop.\ (3.3.8)]{Suz14}.
\end{PropApp}

\begin{proof}
	This reduces to a corresponding non-derived statement about morphisms between
		\begin{gather*}
				\alg{\Gamma} \bigl(
					\Hat{\Order}_{x},
					\sheafhom_{\Hat{\Order}_{x}}(A, B)
				\bigr),
			\\
				\alg{\Gamma}' \bigl(
					\Hat{\Order}_{x},
					\sheafhom_{\Hat{\Order}_{x, \fppf} / k^{\ind\rat}_{\et}}(
						\Bar{f}_{\ast} A,
						\Bar{f}_{\ast} B
					)
				\bigr),
			\\
				\sheafhom_{k^{\ind\rat}_{x, \pro\et}} \bigl(
					\alg{\Gamma}_{x}(\Hat{\Order}_{x}, A),
					\alg{\Gamma}_{x}(\Hat{\Order}_{x}, B)
				\bigr)
		\end{gather*}
	on the level of complexes of sheaves.
	Evaluate these complexes at each $k'_{x} \in k_{x}^{\ind\rat}$
	and then check the compatibility.
\end{proof}


\end{document}